\documentclass[a4paper,12pt,reqno,makeidx]{amsart}
\usepackage{amsfonts,amssymb,amscd,amsmath,latexsym,amsbsy,enumerate}
\usepackage{stmaryrd,a4wide,mathrsfs,graphicx}
\usepackage[pagebackref,hypertexnames=false, colorlinks, citecolor=black, linkcolor=blue]{hyperref}

\newtheorem{thm}{Theorem}[section]
\newtheorem{cor}[thm]{Corollary}
\newtheorem{lemma}[thm]{Lemma}
\newtheorem{prop}[thm]{Proposition}
\newtheorem{defn}[thm]{Definition}
\theoremstyle{definition}
\newtheorem{remark}[thm]{Remark}
\newtheorem{example}[thm]{Example}
\newtheorem{ass}[thm]{Assumption}
\numberwithin{equation}{section}

\def\al{\alpha}
\def\be{\beta}
\def\ga{\gamma}
\def\de{\delta}
\def\ep{\varepsilon}
\def\la{\lambda}
\def\ka{\kappa}
\def\si{\sigma}
\def\vp{\varphi}
\def\om{\omega}
\def\La{\Lambda}
\def\Ga{\Gamma}
\def\Om{\Omega}
\def\Z{\mathbb{Z}}
\def\R{\mathbb{R}}
\def\C{\mathbb{C}}
\def\N{\mathbb{N}}
\def\cD{\mathcal{D}}
\def\cH{\mathcal{H}}
\def\cK{\mathcal{K}}
\def\cP{\mathcal{P}}
\def\cU{\mathcal{U}}
\def\cO{\mathcal{O}}
\def\cL{\mathcal{L}}
\def\cF{\mathcal{F}}
\def\cN{\mathcal{N}}
\def\cV{\mathcal{V}}
\def\scA{\mathscr{A}}
\def\scB{\mathscr{B}}
\def\lc{\text{\rm lc}}
\def\Tr{\mathrm{Tr}}
\def\Ker{\mathrm{Ker}}
\def\Ran{\mathrm{Ran}}
\def\supp{\mathrm{supp}}
\def\SU{\mathrm{SU}}
\newcommand{\rphis}[5]{\,_{#1}\vp_{#2} \left( \genfrac{.}{.}{0pt}{}{#3}{#4}
\ ;#5 \right)}
\newcommand{\rFs}[5]{\,_{#1}F_{#2} \left( \genfrac{.}{.}{0pt}{}{#3}{#4}
\ ;#5 \right)}
\makeindex

\begin{document}
\title[Spectral theory and special functions]
{Applications of spectral theory to special functions}

\author{Erik Koelink}
\address{Radboud Universiteit, IMAPP, FNWI, PO Box 9010, 6500 GL Nijmegen,
the Netherlands}
\email{e.koelink@math.ru.nl}

\begin{abstract}
Many special functions are eigenfunctions to explicit operators, such as 
difference and differential operators, which is in particular true for 
the special functions occurring in the Askey-scheme, its $q$-analogue and 
extensions. 
The study of the spectral properties of such operators
leads to explicit information for the corresponding special functions. 
We discuss several instances of this application, involving 
orthogonal polynomials and their matrix-valued analogues. 
\end{abstract}

\maketitle
\tableofcontents


\section*{Preamble}
We use standard notation for hypergeometric series, basic hypergeometric series (also 
known as $q$-hypergeometric series) and special functions following 
standard references, such as e.g. Andrews, Askey and Roy \cite{AndrAR}, Gasper and Rahman \cite{GaspR},
Ismail \cite{Isma}, Koekoek and Swarttouw \cite{KoekLS}, \cite{KoekS}, Szeg\H{o} \cite{Szeg}, Temme \cite{Temm}. 
There is an abundance of references, and apart from the references in the books in 
the bibliography, the review paper by Damanik, Pushnitski and Simon \cite{DamaPS} contains many 
references. 
The appendix discusses the spectral theorem, and references are given there.
All measures discussed are Borel measures on the real line, and we 
denote the $\si$-algebra of Borel sets  on $\R$ by $\scB$.
Furthermore, $\N=\{0,1,2,\cdots\}$. 
All the results in these notes have appeared in the literature. 


\section{Introduction}\label{sec:Introduction}

Spectral decompositions of self-adjoint operators on Hilbert 
spaces can at least be traced back to the work of 
Fredholm on the solutions of integral equations.
The study of Sturm-Liouville differential operators 
was a great impetus for the development of spectral analysis, 
see e.g. \cite{Dieu}. 
For some explicit Sturm-Liouville type differential operators 
there is a link to well-known special functions, such as 
e.g. Jacobi polynomials, which shows the close connection between 
special functions and spectral theory. 
At the moment, this is for instance an important ingredient in the study of
so-called exceptional orthogonal polynomials, see e.g. \cite{Dura-LN}. 

Spectral theory is, loosely speaking,  essentially a study of the eigenvalues, 
or spectral data, of 
a suitable operator, and to determine such an operator 
completely in terms of its eigenvalues. For a self-adjoint matrix
this means that we look for its eigenvalues, which are real in this case,
and the corresponding eigenspaces, which are orthogonal in this case.
So we can write the self-adjoint matrix as a sum of multiplication 
and projection operators, and this is the most basic form 
of the spectral theorem for self-adjoint operators.
We recall the spectral theorem in its most general 
form in Appendix \ref{app:spectralthm}.

The application to differential operators, and also to various developments 
in physics, such as  quantum mechanics, is still very important.
Through this application, there have been many developments for 
special functions. One of the classical applications is to 
study the second order differential operator 
\[
D^{\al,\be} = (1-x^2) \frac{d^2}{dx^2} + \bigl( \be -\al - (\al+\be+2)x\bigr) \frac{d}{dx}
\]
on the weighted $L^2(w^{\al,\be})$ space for the weight $w^{\al,\be}(x)= C (1-x)^\al(1+x)^\be$ on $[-1,1]$
for a suitable normalisation constant $C$.
Then $D^{\al,\be}$ can be understood as an unbounded self-adjoint 
operator with compact resolvent. The spectral measure is then given 
by projections on the orthonormal Jacobi polynomials,\index{orthogonal polynomials!Jacobi polynomials} which are eigenfunctions 
of $D^{\al,\be}$. 
Similarly, the differential operator can also be studied on 
$[1,\infty)$ with respect to a suitable weight, and then its 
spectral decomposition leads to the Jacobi-function transform, 
see e.g. \cite[Ch.~XIII]{DunfS}, \cite{Koor-Jacobi} and references.

Another classical application of spectral analysis is a proof of 
Favard's theorem, see Corollary \ref{cor:Favardsthm}, stating that polynomials satisfying
a suitable three-term recurrence relation, are orthogonal polynomials.
This follows from studying a so-called Jacobi operator on the 
Hilbert space $\ell^2(\N)$ of square summable sequences. 
The spectral analysis of such a Jacobi operator is closely related to 
the moment problem, and this link can be found at several places
in the literature such as e.g. \cite{Deif}, \cite{DunfS}, \cite{Koel-Laredo}, \cite{Schm}, 
\cite{ShohT}, \cite{Simo}. 
The Haussdorf moment problem, i.e. on a finite interval, played an important
role in the development of functional analysis, notably the development of
functionals and related theorems, 
see \cite[\S I.3]{Monn}.

One particular application is to have other explicit operators, e.g. 
differential operators or difference operators, 
realised as Jacobi operators and next use this connection to obtain
results for these explicit operators.
In Section \ref{sec:JMatrixmethod} we give a couple of examples, 
including the original (as far as we are aware) motivating 
example of the Schr\"odinger operator with Morse potential due 
to the chemist Broad, see references in Section \ref{ssec:SchrodingerMorsepotential}.

\begin{figure}[th]
\begin{center}
\includegraphics[height=11truecm]{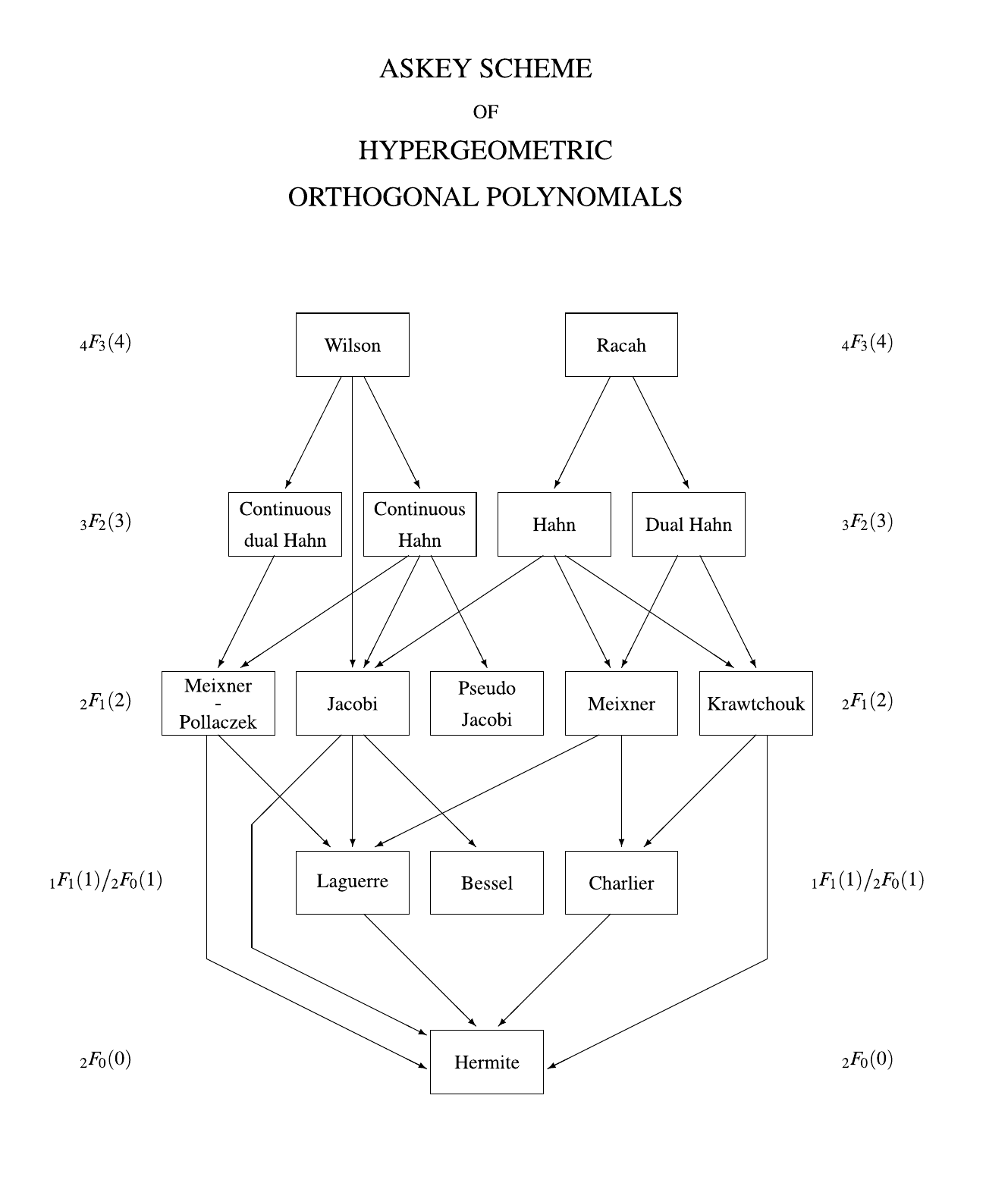}
\end{center}
\caption{The Askey scheme as in \cite{KoekLS}.}\label{fig:Askeyscheme}
\end{figure}

As is well-known the Askey scheme of hypergeometric orthogonal polynomials 
and its $q$-analogue, see e.g. \cite{KoekLS}, \cite{KoekS}, and 
initially observed by Askey in \cite[Appendix]{AskeW}, see also 
the first Askey-scheme in Labelle \cite{Labe} --drawn by hand--, 
consists of those polynomials which are also eigenfunctions to 
a second-order operator, which can be a differential 
operator, a difference operator or a $q$-difference operator
of some kind. See Figures \ref{fig:Askeyscheme} and 
\ref{fig:qAskeyscheme}, taken from Koekoek, Lesky, Swarttouw \cite{KoekLS}
for the current state of affairs. 
Naturally, many of these operators, like the 
differential operator for the Jacobi polynomials, have been studied 
in detail. This is in particular valid for the operators occurring in 
the Askey-scheme. For the other operators, especially the 
difference operators for the orthogonal polynomials in the $q$-analogue of
the Askey scheme corresponding to indeterminate moment
problems, see \cite{Chri-PhD}. 
On the other hand, it is natural to extend the ($q$-)Askey-scheme 
to include also integral transforms with kernels in 
terms of (basic) hypergeometric series, such 
as the Hankel, Jacobi, Wilson transform, and its $q$-analogues 
and to study these transforms and their properties from a 
spectral analytic point of view using the associated operators.
We refer to the schemes \cite[Fig.~1.1, 1.2]{KoelSNATO} 
remarking that in the meantime \cite[Fig.~1.1]{KoelSNATO} 
has been vastly extended to include the Wilson function 
transform by Groenevelt \cite{Groe-WT}, and various 
transformations that can be obtained as limiting cases.
In the terminology of Gr\"unbaum and coworkers, all the 
instances of the ($q$-)Askey-scheme are examples of 
the bispectral property.\index{bispectral property}
This means that the polynomials are eigenfunctions to 
a three-term recurrence operators (acting in the degree) 
and at the same time are eigenfunctions of a suitable
second order differential or difference operator in 
the variable. 
In particular, all these instances give rise to bispectral families
of special functions. 

\begin{figure}[th]
\begin{center}
\hspace*{-1truecm}
\includegraphics[height=17truecm, angle=-90]{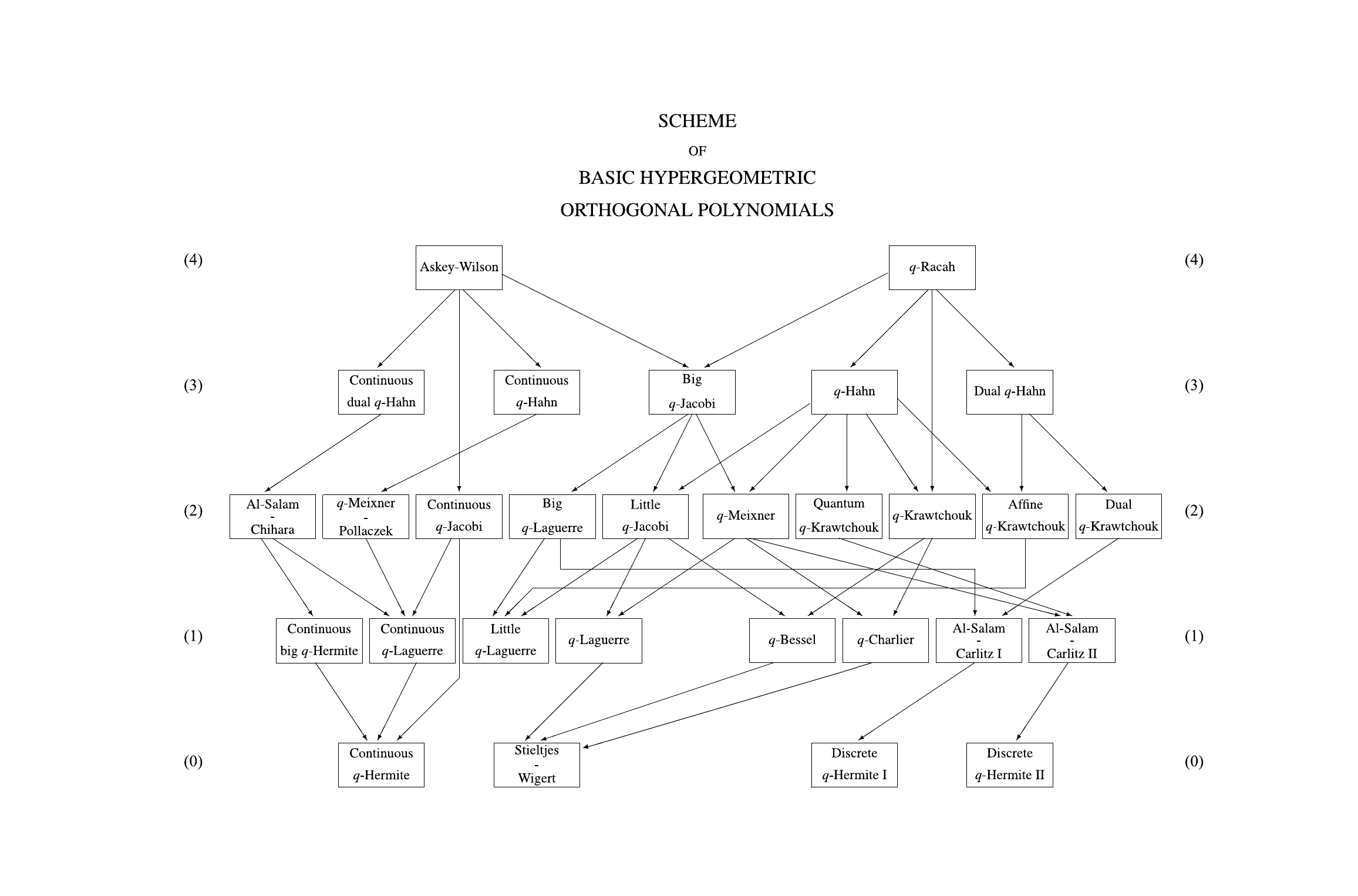}
\end{center}
\caption{The $q$-Askey scheme as in \cite{KoekLS}.}\label{fig:qAskeyscheme}
\end{figure}

Motivated by one of the second order $q$-difference operators
arising in the $q$-analogue of the Askey-scheme, we discuss 
the spectral analysis of three-term recurrence operators 
on $\ell^2(\Z)$ in Section \ref{sec:threetermZ}.
We apply the spectral theorem to a particular example and we obtain a set 
of orthogonality measures for the continuous $q^{-1}$-Hermite 
polynomials. Here we follow the convention $0<q<1$, so that $q^{-1}>1$. 
These measures turn out to be N-extremal, where
N stands for Nevanlinna. This result is originally obtained 
by Ismail and Masson \cite{IsmaM}, and this proof is due to 
Christiansen and the author \cite{ChriK} as a special case of 
results for the symmetric Al-Salam--Chihara polynomials for $q>1$. 
This is partly based on \cite[\S 4]{Koel-Laredo}.
Similar ideas have been used in e.g. \cite{ChriK-JAT}, \cite{GroeK-Meixner}, 
 to study other moment problems and related orthogonal
polynomials. 

In Section \ref{sec:threetermN} we briefly recall the 
relation between three-term recurrence operators and 
orthogonal polynomials.
This is a well known subject in the literature, and 
there are several books and review papers on this 
subject, e.g. \cite{Akhi}, \cite{BuchC}, \cite{DunfS}, \cite{Land}, 
\cite[Ch.~16]{Schm}, \cite{ShohT}, \cite{Simo}, \cite{Ston}.
We base ourselves on \cite{Koel-Laredo}, and we extend 
this approach to the case of matrix-valued orthogonal 
polynomials and block Jacobi operators. 
The spectral approach is essentially due to M.G.~Kre\u\i n \cite{Krei-mm}, 
\cite{Krei}, whose great mathematical legacy is discussed in 
\cite{Adam-etal-KreinCentenary}. 
We discuss briefly a rather general example of arbitrary size. 
In Section \ref{sec:MoreMweightsMVOPS-Jacobi} we discuss 
some of the assumptions made in the Section \ref{sec:MVOP}. 
Here we make also use of previous lecture series by 
Berg \cite{Berg} and Dur\'an and L\'opez-Rodr\'\i guez \cite{DuraLR-Laredo},
but also \cite{ApteN}, \cite{Gero}, \cite{Krei-mm}, \cite{Krei}. 

In Section \ref{sec:JMatrixmethod} we show how realisations of 
explicit operators, such as differential operators, as recurrence operators
can be used to study the spectral theory. This gives rise to 
relations between the spectral decomposition of such an 
operator and the related orthogonal polynomials. 
In the physics literature such a method is known as the $J$-matrix
method, and there is a vast literature of physics applications, see 
e.g. references to work of Al-Haidari, Bahlouli, Bender, Dunne, Yamani
and others in \cite{IsmaK}. 
The first example of Section \ref{sec:JMatrixmethod} 
is the study of the Schr\"odinger operator with a Morse
potential, originally due to Broad \cite{Broa-LNM}, see also \cite{Dies}. 
The second example of Section \ref{sec:JMatrixmethod} is in the same vein, and due to Ismail and the 
author \cite{IsmaK}. This case has recently 
been generalised by Genest et al. \cite{GeneIVZ-PAMS2016} to 
include more parameters and to cover the full family of Wilson 
polynomials. Moroever, in \cite{GeneIVZ-PAMS2016} a link to the Bannai-Ito algebra is established. 
The last example of Section \ref{sec:JMatrixmethod} leads to a more general family of matrix-valued
orthogonal polynomials for operators which have a realisation as  
a $5$-term recurrence operator. We then discuss an example of
such a case, extending the second example of Section \ref{sec:JMatrixmethod}. We apply this 
approach to 
an explicit second order differential operator.
The same realisation of suitable operators as tridiagonal
operators has useful implications in e.g. representation theory,
see e.g. \cite{CiccKK}, \cite{Groe-Laguerre}, \cite{Groe-PhD}, \cite{Groe-JMP2014}, 
\cite{GroeK-JPA2002}, \cite{GroeKK}, 
\cite{GroeKR}, \cite{KoelK}, 
\cite{KoelVdJ}, \cite{MassR}, \cite{Nere} for the case of 
representation theory of the Lie algebra $\mathfrak{su}(1,1)$
and its quantum group analogue. 
Using explicit realisations of representations, these 
results have given very explicit bilinear generating functions,
see e.g. \cite{Groe-Ramanujan}, \cite{KoelVdJ-CA}. 

All the general results as well as the explicit examples have 
appeared in the literature before. 
There are many other references available in the literature, 
and apart from the books --and the references mentioned there-- 
mentioned in the bibliography, one can 
especially consult the references in \cite{DamaPS}, where 
a list  of more than 200 references can be found. 
In particular, there are many papers available that generalise
known results in the general theory of orthogonal polynomials 
to the matrix-valued orthogonal polynomial case, and we 
refer to the references to work by Berg, Cantero, Castro, Dur\'an, 
Geronimo, Gr\"unbaum, 
de la Iglesia, Lop\'ez-Rodr\'\i guez, Marcell\'an, 
Pacharoni, Tirao, Van Assche, etc. to the references in \cite{DamaPS}. 

Let us note that in these notes the emphasis is on explicit operators 
related to explicit sets of special functions, so that 
information on these special functions is obtained from the 
spectral analysis.
On the other hand, there are also many results on  
the spectral analysis of more general
classes of operators. For this subject one can consult 
Simon's book \cite{Simo-book} and the extensive list 
of references given there. 

It may happen that a differential or difference operator
with suitable eigenfunctions in terms of well-known 
special functions cannot be suitably realised as a three-term
recurrence operator on a Hilbert space such as $\ell^2(\Z)$ or $\ell^2(\N)$.
It can then be very useful to look for a larger Hilbert space,  
and an extension of the operator to the larger Hilbert space.
This is different from the extension of a Hilbert space in order
to find self-adjoint extensions. 
Then one needs to find a way of obtaining the extended Hilbert space
and the extension of the operator. 
This is usually governed by the interpretation of these operators
and special functions in a different context, like e.g. 
representation theory. 
Examples are in e.g.  \cite{Groe-WT}, \cite{Groe-CA2009}, \cite{KoelSbig}, \cite{KoelSAW}, \cite{Nere}. 
This leads to extensions of the Askey and $q$-Askey scheme of 
Figures \ref{fig:Askeyscheme}, \ref{fig:qAskeyscheme} with non-polynomial 
function transforms arising as the spectral decomposition of 
suitable differential and difference operators on Hilbert spaces of 
functions, see e.g. Figures 1.1 and 1.2 in \cite{KoelSNATO}. 
Figure 1.2 of \cite{KoelSNATO} is still valid as an extension of 
the $q$-Askey scheme, but Figure 1.1 of \cite{KoelSNATO} has now
Groenevelt's Wilson function transforms \cite{Groe-WT} at the top level. 

\medskip
\textbf{Acknowledgement.} 
Thanks to Ren\'e Swarttouw and Roelof Koekoek for their 
version of the Askey-scheme in Figures \ref{fig:Askeyscheme} and \ref{fig:qAskeyscheme}. 
I thank the organisers of the summer school, in particular Howard Cohl, Mourad Ismail and Kasso Okoudjou, 
for the opportunity
to give the lectures at \emph{Orthogonal Polynomials and Special Functions Summer School
OPSF-S6}, July 2016, University of Maryland. I also thank all the participants of OPSF-6 for 
their feedback. I thank Wolter Groenevelt and Luud Slagter for their input. 
The referees have pointed out many errors and oversights, and I thank them for their 
help in improving these lecture notes. 

\section{Three-term recurrences in $\ell^2(\Z)$}
\label{sec:threetermZ}

In this section we discuss three-term recurrence relations 
on the Hilbert space $\ell^2(\Z)$. We apply to this one 
particular example, which is motivated by a second order
difference operator arising in the $q$-Askey scheme.

We consider sequence spaces and the associated Hilbert spaces as 
in Example \ref{exa:Hilbertspaces}.
For the Hilbert space $\ell^2(\Z)$ with orthonormal basis $\{ e_l\}_{l\in \Z}$
we consider for complex sequences $\{a_l\}_{l\in \Z}$, $\{b_l\}_{l\in \Z}$, 
$\{c_l\}_{l\in \Z}$
the operator 
\begin{equation}\label{eq:defL}
L\, e_l = a_l e_{l+1} + b_l e_l + c_{l}e_{l-1},\qquad l\in \Z,
\end{equation}
with dense domain $\cD$\index{D@$\cD$} the subspace of finite 
linear combinations of the basis vectors.

\begin{lemma}\label{lem:bddL}
$L$ extends to a bounded operator on $\ell^2(\Z)$ if and only if the 
sequences $\{a_l\}_{l\in \Z}$, $\{b_l\}_{l\in \Z}$, 
$\{c_l\}_{l\in \Z}$ are bounded.
\end{lemma}

In Exercise \ref{exer:lembbdL} you are requested to prove Lemma \ref{lem:bddL}. 

In case $L$ is bounded, we see that $L$ acting on 
$v=\sum_{k\in \Z}v_k e_k \in \ell^2(\Z)$ is given by 
\begin{equation}\label{eq:Lv-bdd}
Lv = \sum_{k\in \Z}  (a_{k-1}v_{k-1} + b_k v_k + c_{k+1}v_{k+1}) \, e_k. 
\end{equation}
In case $L$ is not bounded, we have to interpret this in a suitable fashion,
by e.g. initially allowing only for $v\in \ell^2(\Z)$ with
only finitely many non-zero coefficients, i.e. for $v\in\cD$.  
In general we view $L$ as an operator acting on the sequence space of
sequences labeled by $\Z$, and we are in particular interested in the case
of square summable sequences. 

\begin{lemma}\label{lem:adjointL}
For $v=\sum_{k\in \Z}v_k e_k \in \ell^2(\Z)$ define 
\begin{equation*}
L^\ast \, v = \sum_{k=-\infty}^\infty
\bigl( \overline{a_k}v_{k+1}+ \overline{b_k} v_k+ \overline{c_{k}}v_{k-1}\bigr)\, e_k,
\end{equation*}
which, in general, is not an element of $\ell^2(\Z)$.
Define 
\[
\cD^\ast = \{ v\in \ell^2(\Z) \mid L^\ast v\in \ell^2(\Z)\}
\]
The adjoint of $(L, \cD)$ is $(L^\ast, \cD^\ast)$. 
\end{lemma}

In Exercise \ref{exer:adjointL} you are requested to prove Lemma \ref{lem:adjointL}.

As for $L$ in \eqref{eq:Lv-bdd}, we apply $L^\ast$ to 
arbitrary sequences. 

Note that $\cD\subset \cD^\ast$, so that $(L,\cD)$ is symmetric in 
case $L^\ast \vert_{\cD}=L$, which is the case for $\overline{a_k}=c_{k+1}$ and
$\overline{b_k} = b_k$ for all $k\in\Z$. 

From now on we assume that $\overline{a_k}=c_{k+1}$ and
$\overline{b_k} = b_k$ for all $k\in\Z$, and moreover, that $a_k>0$ for all $k\in \Z$.
This last assumption is not essential, since changing each of the basis elements 
by a phase factor shows that we can assume this in case $a_k\not=0$ for all $k\in \Z$. 
Note that in case $a_{k_0}=0$ for some $k_0$ we have $L$-invariant subspaces, and we can consider $L$ on
such an invariant subspace. 
In particular, the dimension of the space of formal solutions to $L^\ast f = zf$ is two.

\begin{example}\label{exa:ASC-q1}
The first example is related to explicit orthogonal polynomials, namely the 
symmetric Al-Salam--Chihara polynomials\index{orthogonal polynomials!Al-Salam--Chihara polynomials} in base $q^{-1}$,
see \cite{GaspR}, \cite{KoekLS}, \cite{KoekS}. 
We are in particular interested in the limit case of the continuous $q^{-1}$ 
Hermite polynomials 
\index{orthogonal polynomials!continuous $q^{-1}$-Hermite polynomials} introduced
by Askey \cite{Aske}. 
These 
polynomials correspond to an indeterminate moment problem, see Section \ref{sec:threetermN}, and 
have been studied in detail by Ismail and Masson \cite{IsmaM}, who have determined the 
explicit expression of the  N-extremal measures,\index{N-extremal measure} 
where N stands for Nevanlinna. 
The N-extremal measures are the measures for which the polynomials are dense in the 
corresponding weighted $L^2$-space. 

The details of Example \ref{exa:ASC-q1} are taken from \cite{ChriK}, in which the case of 
the general symmetric Al-Salam--Chihara polynomials is studied and 
in the notation of \cite{ChriK} this example corresponds to $\be\downarrow 0$. 
The polynomials, after rescaling, are eigenfunctions to a second order $q$-difference
equation for functions supported on a set labeled by $\Z$. 
After rewriting, we find the following three-term recurrence operator: 
\begin{gather*}
L\, e_l = a_l e_{l+1} + b_l e_l + a_{l-1}e_{l-1}, \\
a_l=\frac{\al^2 q^{2l+\frac12}}{1+\al^2 q^{2l+1}}
\frac{1}{\sqrt{(1+\al^2
q^{2l})(1+\al^2 q^{2l+2})}} \qquad 
b_l= \frac{\al^2(1+q)q^{2l-1}}{(1+\al^2q^{2l+1})(1+\al^2q^{2l-1})},
\end{gather*}
where $\al\in (q,1]$. We emphasise that the polynomials being eigenfunctions to $L$
follows from the second order $q$-difference operator for the continuous $q^{-1}$-Hermite
polynomials \cite{Aske}, \cite[(3.26.5)]{KoekS}, and not from the three-term 
recurrence relation for orthogonal polynomials. 
Recall that $0<q<1$. 
It follows immediately from the explicit expressions that
\begin{gather*}
a_l=\begin{cases} \al^2 q^{2l+\frac12} + \cO(q^{4l}),
& l\to\infty, \\[3mm]
\al^{-2} q^{-2l-\frac32} + \cO(q^{-4l}), &l\to-\infty,
      \end{cases} \qquad 
b_l=
\begin{cases} \al^2(1+q) q^{2l-1} + \cO(q^{4l}), & l\to\infty,\\[3mm]
\al^{-2}(1+q)q^{-2l-1} + \cO(q^{-4l}), & l\to
-\infty.
       \end{cases}
\end{gather*}
The exponential decay of the coefficients $a_l$ and $b_l$ in this case for $l\to\pm\infty$, show that 
we can approximate $L$ by the finite rank operators $P_n L$, where $P_n$ is the 
projection on the finite dimensional subspace spanned by the basis vectors 
$\{e_{-n}, e_{-n+1}, \cdots, e_{n-1}, e_n\}$. The approximation holds true in 
operator norm, \index{operator norm}
$\| L- P_nL\| = \cO(q^n)$, so that $L$ is a compact operator. \index{compact operator}
So the operator $L$ has discrete spectrum accumulating at zero, and each of the 
eigenspaces for the non-zero eigenvalues is finite-dimensional. 
\end{example}

Next we consider the formal eigenspaces for $z\in \C$ of $L^\ast$; 
\begin{equation}
\begin{split}
S^+_z &= \{ f= \sum_{k\in\Z} f_k e_k \mid 
L^\ast f= z f, \ \sum_{k>0} |f_k|^2 <\infty\} \\
S^-_z &= \{ f= \sum_{k\in\Z} f_k e_k \mid  
L^\ast f= z f, \ \sum_{k<0} |f_k|^2 <\infty\} \\
\end{split}
\end{equation}
So $\dim S^{\pm}_z \leq 2$. Note that $S^\pm_z$ consist of those eigenvectors
that are square summable at $\pm\infty$, which we call the 
free solutions\index{free solution} at $\pm\infty$. 

For any two sequences $\{v\}_{l\in\Z}$, $\{f\}_{l\in\Z}$, we define 
the Wronskian\index{Wronskian} or 
Casorati determinant\index{Casorati determinant} 
by
\[
[v,f]_l = a_l\bigl( v_{l+1}f_l - f_{l+1}v_l\bigr),
\]
which is a sequence. However, for eigenvectors of $L^\ast$ the Wronskian or 
Casorati determinant is a constant sequence.

\begin{lemma}\label{lem:Wronskian}
Let $v$ and $f$ be formal solutions to $L^\ast u = zu$, then 
\[
[v,f] = [v,f]_l = a_l\bigl( v_{l+1}f_l - f_{l+1}v_l\bigr)
\]
is independent of $l\in \Z$. 
\end{lemma}

In particular, Lemma \ref{lem:Wronskian} can be applied to the solutions in $S^\pm_z$. 
Note that the Casorati determinant $[v,f]\not=0$ for non-trivial solutions 
unless $v$ and $f$ span a one-dimensional subspace of solutions. 

\begin{proof} Since $v$ and $f$ are formal solutions, we have for all $l\in \Z$ 
\begin{equation*}
\begin{split}
a_lv_{l+1}+ b_l v_l+ a_{l-1}v_{l-1} & = z v_l \\ 
a_lf_{l+1}+ b_l f_l+ a_{l-1}f_{l-1} & = z f_l \\ 
\end{split}
\end{equation*}
since we assume the self-adjoint case. Multiplying the first equation by $f_l$ and the second
by $v_l$ and subtracting gives
\begin{equation*}
a_l\bigl( v_{l+1}f_l - f_{l+1}v_l\bigr) + a_{l-1}\bigl(v_{l-1}f_l - f_{l-1}v_l\bigr) =0
\end{equation*}
which means that $[v,f]_l$ is indeed independent of $l\in \Z$. 
\end{proof}

\begin{thm}\label{thm:resolventL} 
Assume that $\dim S^\pm_z=1$ for all $z\in \C\setminus \R$ and that
$S^+_z \cap S^-_z =\{0\}$ for all $z\in \C\setminus \R$. 
Then $(L^\ast, \cD^\ast)$ is self-adjoint. 
The resolvent operator\index{resolvent operator} is given by 
\eqref{eq:def:GreenkernelZ1}, \eqref{eq:def:GreenkernelZ2}. 
\end{thm}

In Section \ref{sec:threetermN} we show that $\dim S^\pm_z\geq 1$ for all $z\in \C\setminus\R$.
Note that $S^+_z \cap S^-_z$ gives the deficiency space at $z\in \C\setminus\R$ of
$(L^\ast, \cD^\ast)$, which has constant dimension on the upper and lower half plane,
see Appendix \ref{ssec:appunboundedsaoperators}. 
Since $L$ has real coefficients, it commutes with complex conjugation, i.e.
 for $f=\sum_l f_l e_l$ define the vector  
$\bar f=\sum_l \overline{f_l} e_l$, then $L^\ast \bar f = \overline{L^\ast f}$,
we see that the deficiency spaces $N_z$ and $N_{\bar z}$ have the same dimension. 
So we can replace the assumption $S^+_z \cap S^-_z =\{0\}$ for all $z\in \C\setminus \R$ 
in Theorem \ref{thm:resolventL} by $S^+_z \cap S^-_z =\{0\}$ for some $z\in \C\setminus \R$. 

\begin{proof} Since the deficiency index $n_z =\dim (S^+_z \cap S^-_z)=0$, we see 
that $(L^\ast, \cD^\ast)$ has deficiency indices $(0,0)$, so that 
by Proposition \ref{prop:App-extensionselfadjointopertaotrs}  it is self-adjoint.  

Now take non-zero 
$\phi_z\in S^+_z$, $\Phi_z\in S^-_z$, which are unique up to a scalar by 
assumption.  Moreover, the Wronskian $[\phi_z,\Phi_z]\not=0$, since $\phi_z$ and $\Phi_z$ are
not multiples of each other. 
We define the Green kernel\index{Green kernel} for $z\in\C\setminus\R$ by
\begin{equation}\label{eq:def:GreenkernelZ1}
G_{k,l}(z) = \frac{1}{[\phi_z,\Phi_z]}\begin{cases}
(\Phi_z)_k\, (\phi_z)_l, & k\leq l, \\
(\Phi_z)_l\, (\phi_z)_k, & k>l.
\end{cases}
\end{equation}
So $\{ G_{k,l}(z)\}_{k=-\infty}^\infty,
\{ G_{k,l}(z)\}_{l=-\infty}^\infty \in \ell^2(\Z)$ and
$\ell^2(\Z)\ni v\mapsto G(z)v$ given by
\begin{equation}\label{eq:def:GreenkernelZ2}
G(z) v = \sum_{k\in \Z} (G(z)v)_k e_k, \qquad 
(G(z)v)_k = \sum_{l=-\infty}^\infty v_l G_{k,l}(z) =
\langle v, \overline{G_{k,\cdot}(z)}\rangle
\end{equation}
is well-defined. Note that $v\in \cD$ implies 
\[
|(G(z)v)_k| \leq 
\sum_{\underset{\scriptstyle{\text{finite}}}{l=-\infty}}^\infty |v_l G_{k,l}(z)| 
\leq \Bigl( \sum_{\underset{\scriptstyle{\text{finite}}}{l=-\infty}}^\infty |v_l|^2 \Bigr)^{1/2}
\Bigl(\sum_{\underset{\scriptstyle{\text{finite}}}{l=-\infty}}^\infty |G_{k,l}(z)|^2\Bigr)^{1/2} 
= \|v\|\, \Bigl(\sum_{\underset{\scriptstyle{\text{finite}}}{l=-\infty}}^\infty |G_{k,l}(z)|^2\Bigr)^{1/2} 
\]
and 
\begin{gather*}
\|G(z) v\|^2 \leq \sum_{k\in \Z} |(G(z)v)_k|^2 \leq 
\|v\|^2 \sum_{k\in \Z} \sum_{\underset{\scriptstyle{\text{finite}}}{l=-\infty}}^\infty |G_{k,l}(z)|^2 
= \|v\|^2 \sum_{\underset{\scriptstyle{\text{finite}}}{l=-\infty}}^\infty 
\sum_{k\in \Z}  |G_{k,l}(z)|^2 <\infty
\end{gather*}
since $\sum_{k\in \Z}  |G_{k,l}(z)|^2 <\infty$ by the definition \eqref{eq:def:GreenkernelZ1}
and $\phi_z\in S^+_z$, $\Phi_z\in S^-_z$. So $G(z) v\in \ell^2(\Z)$. 

We first check $(L^\ast -z )G(z)v=v$ for $v$ in the
dense subspace $\cD$. We do so by calculating the 
$k$-th entry of $[\phi_z,\Phi_z] (L^\ast -z) G(z)v$ as
a sum over $l\in \Z$, which we split up in a sum until 
$k-1$, from $k+1$ and a single term. Explicitly, 
\begin{equation*}
\begin{split}
& [\phi_z,\Phi_z]\bigl( (L^\ast -z) G(z)v\bigr)_k  \\
= & [\phi_z,\Phi_z]\Bigl( a_k \bigl( G(z)v\bigr)_{k+1} +
(b_k-z) \bigl( G(z)v\bigr)_{k} + a_{k-1} \bigl( G(z)v\bigr)_{k-1}\Bigr) \\
= &\sum_{l=-\infty}^{k-1} v_l \bigl( a_k
(\phi_z)_{k+1}+(b_k-z)(\phi_z)_k +a_{k-1}(\phi_z)_{k-1}\bigr)
(\Phi_z)_l \\
& + \sum_{l=k+1}^\infty v_l \bigl( a_k
(\Phi_z)_{k+1}+(b_k-z)(\Phi_z)_k +a_{k-1}(\Phi_z)_{k-1}\bigr)
(\phi_z)_l  \\
& + v_k\bigl( a_k(\Phi_z)_k(\phi_z)_{k+1} +
(b_k-z)(\Phi_z)_k(\phi_z)_k +
a_{k-1}(\Phi_z)_{k-1}(\phi_z)_k\bigr) \\
= & v_k a_k\bigl((\Phi_z)_k(\phi_z)_{k+1} -(\Phi_z)_{k+1}
(\phi_z)_k\bigr) = v_k [\phi_z,\Phi_z] 
\end{split}
\end{equation*}
The first term vanishes, since $\phi_z$ is a formal eigenfunction to $L$. Similarly, the 
second sum vanishes, since $\Phi_z$ is an eigenfunction to $L$. 
Finally, use $(b_k-z)(\Phi_z)_k(\phi_z)_k  +a_{k-1}(\Phi_z)_{k-1}(\phi_z)_k  =
- a_k (\Phi_z)_{k+1}(\phi_z)_k$ and recognise the Casorati determinant.

By assumption, $\phi_z$ and $\Phi_z$ are not linearly dependent, so that the 
Casorati determinant $[\phi_z,\Phi_z]\not=0$. Dividing both sides by 
the Casorati determinant gives the result. Note that this also shows 
that $G(z) v\in \cD^\ast$. So we see see that $(L^\ast -z)G(z)$ is the identity
on the dense subspace $\cD$, and since $L^\ast$ is selfadjoint, we have 
that $R(z)=(L^\ast -z)^{-1}$ is a bounded operator which is equal to $G(z)$. 
\end{proof}

Note that the determination of the spectral measure is governed by the 
structure of the function $z\mapsto [\phi_z,\Phi_z]$, which is 
analytic in the upper and lower half plane. In particular, if it extends
to a  function on $\C$ with poles at the real axis, we see that 
the spectral measure is discrete. This happens in case of 
Example \ref{exa:ASC-q1}. 

\begin{example}\label{exa:ASC-q2}
We continue Example \ref{exa:ASC-q1}, and we describe the solution space in 
some detail. Define the constant
\begin{equation*}
C_l(\al) = \frac{\al^{2l}
q^{l^2-\frac12 l} \sqrt{1+\al^2 q^{2l}}}{(-\al^2q;q)_{2l}} 
= \begin{cases}
\cO( \al^{2l} q^{l^2-\frac12 l}), & l\to \infty \\
\cO( q^{-l^2-\frac12 l})     & l\to -\infty 
\end{cases}
\end{equation*}
and for $z\in \C\setminus \R$ the functions 
\begin{gather*}
(\phi_z)_l = C_l(\al) z^{-l} \rphis{0}{1}{-}{-\al^2 q^{l+1}}{q, -\frac{\al^2 q^{2l+1}}{z}}, \\
(\Phi_z)_l = \frac{1}{C_l(\al)} z^{l} 
\rphis{0}{1}{-}{-\al^{-2} q^{1-l}}{q, -\frac{q^{1-2l}}{\al^2 z}}. 
\end{gather*}
Then the corresponding elements $\phi_z\in S^+_z$ and $\Phi_z\in S^-_z$. 
The $\ell^2$-behaviour follows easily from the asymptotic behaviour of
the constant $C_l(\al)$. The fact that these functions actually are a solution
for the three-term recurrence relation follows from 
contiguous relations for basic hypergeometric series, and we do not give 
the details, see \cite{ChriK} and Exercise \ref{exer:SPmzforcontqinvHermite}. 
Next we calculate $[\phi_z, \Phi_z]=-z(1/z;q)_\infty$ using a limiting argument, 
see Exercise \ref{exer:SPmzforcontqinvHermite} as well.
\end{example}

Now that in the situation of Theorem \ref{thm:resolventL} we have 
explicitly determined the resolvent operator $R(z)=(L^\ast-z)^{-1}$
we can apply the Stieltjes-Perron inversion formula 
of Theorem \ref{thm:App-StieltjesPerron}. For this we need 
\begin{equation}\label{eq:jk}
\bigl< \bigl(L^\ast-z\bigr)^{-1}v,w \bigr>= \sum_{k\leq
j}\frac{(\phi_{z})_j(\Phi_{z})_k} {[\phi_{z},\Phi_{z}]}
(v_k\overline{w}_j+v_j\overline{w}_k)(1-\tfrac{1}{2}\delta_{j,k})
\end{equation}
for $v,w\in\ell^2(\Z)$, which follows by plugging in 
the expression of the Green kernel for the resolvent as
in Theorem \ref{thm:resolventL} and its proof, see 
\eqref{eq:def:GreenkernelZ1}, \eqref{eq:def:GreenkernelZ2}. 
So the outcome of the Stieltjes-Perron inversion formula 
of Theorem \ref{thm:App-StieltjesPerron} depends on the 
behaviour of the extension of the function, 
initially defined on $\C\setminus\R$, 
\begin{equation}\label{eq:defkernelStieltjesPerron}
z \mapsto \frac{(\phi_{z})_j(\Phi_{z})_k} {[\phi_{z},\Phi_{z}]}
\end{equation}
when approaching the real axis from above and below. 

We assume that the function in \eqref{eq:defkernelStieltjesPerron}
is analytic in the upper and lower half plane, which can 
be proved in general, see e.g. \cite{Krei-mm}, \cite{Schm}.
Assume now that it has  
an extension to a function exhibiting a pole at $x_0\in\R$.
Then Theorem \ref{thm:App-StieltjesPerron} shows that 
the spectral measure has a mass point at $x_0$ and 
\[
\left<E\bigl(\{x_0\}\bigr)v,w\right>=-\frac{1}{2\pi i}\oint_{(x_0)}
\bigl<(L^\ast-s)^{-1}v,w \bigr>\,ds, \quad v,w\in\ell^2(\Z).
\]
Moreover, assuming that the pole $x_0$ corresponds to a zero of 
the Casorati determinant or Wronskian $[\phi_{z},\Phi_{z}]$, we find
\begin{gather*}
\frac{1}{2\pi i}\oint_{(x_0)}
\frac{(\phi_s)_j(\Phi_s)_k}{[\phi_s,\Phi_s]}\,ds=
(\phi_{x_0})_j(\Phi_{x_0})_k\,\,\underset{\scriptstyle{z=x_0}}{\text{Res}}\,\frac{1}{[\phi_z,\Phi_z]}.
\end{gather*}
In case $\phi_{x_0}$ is a multiple of $\Phi_{x_0}$, the Casorati determinant 
vanishes, so assume $\Phi_{x_0}= A(x_0) \phi_{x_0}$ and that 
$\phi_{x_0}\in \ell^2(\Z)$, so that 
\begin{gather*}
\left<E\bigl(\{x_0\}\bigr)v,w\right>=-A(x_0)
\sum_{k\leq j}
(\phi_{x_0})_j(\phi_{x_0})_k (v_k\overline{w}_j+v_j\overline{w}_k)(1-\tfrac{1}{2}\delta_{j,k})\,
\underset{\scriptstyle{z=x_0}}{\text{Res}}\,\frac{1}{[\phi_z,\Phi_z]} \\
= -A(x_0) \underset{\scriptstyle{z=x_0}}{\text{Res}}\,\frac{1}{[\phi_z,\Phi_z]}
\langle v, \phi_{x_0} \rangle \langle \phi_{x_0}, w \rangle
\end{gather*}
assuming that $\phi_{x_0}= \sum_{l\in \Z} (\phi_{x_0})_l e_l$ has real-valued coefficients 
$(\phi_{x_0})_l$ for real $x_0$. 
See Exercise \ref{exer:realcoefficients} for the general case.

\begin{example}\label{exa:ASC-q3}
We continue Example \ref{exa:ASC-q1}, \ref{exa:ASC-q2}. 
Since $[\phi_z,\Phi_z] = -z(1/z;q)_\infty$ for $z\not=0$, we see 
that we can take $x_0=q^n$ for $n\in \N$ which is a simple 
zero of the Casorati determinant.
Now the residue calculation can be done explicitly; 
\begin{gather*}
\underset{\scriptstyle{z=q^n}}{\text{Res}}\,\frac{1}{[\phi_z,\Phi_z]} = 
\lim_{z\to q^n} \frac{z-q^n}{[\phi_z,\Phi_z]} =
\lim_{z\to q^n} \frac{z-q^n}{-z(1/z;q)_\infty} \\ =
\lim_{z\to q^n} \frac{z-q^n}{-z(1/z;q)_n (1-q^n/z) (q^{n+1}/z;q)_\infty} 
= \frac{-1}{(q^{-n};q)_n (q;q)_\infty} = \frac{(-1)^{n+1} q^{-\frac12 n(n+1)}}{(q;q)_n (q;q)_\infty}
\end{gather*}
Moreover, since the Casorati determinant vanishes, the two solutions of interest are proportional;
\[
(-1)^n\al^{2n+2}(\Phi_{q^n})_l=\frac{(-\al^2q;q)_\infty}{(-1/\al^2;q)_\infty}
(\phi_{q^n})_l, \qquad \forall \, l\in \Z, 
\]
which can be proved by manipulations of basic hypergeometric series, and 
we refer to \cite{ChriK} for the details. 
In particular, $\phi_{q^n} \in \ell^2(\Z)$ for $n\in \N$ and 
$L^\ast \phi_{q^n} = q^n \phi_{q^n}$. 
So the spectral measure in this case has a discrete mass point at $q^n$, $n
\in \N$, satisfying
\begin{gather*}
\left<E\bigl(\{q^n\}\bigr)v,w\right>= 
- (-1)^n\al^{-(2n+2)}\frac{(-\al^2q;q)_\infty}{(-1/\al^2;q)_\infty}
\frac{(-1)^{n+1} q^{-\frac12 n(n+1)}}{(q;q)_n (q;q)_\infty}
\langle v, \phi_{x_0} \rangle \langle \phi_{x_0}, w \rangle \\
= 
\frac{(-\al^2q;q)_\infty}{(-1/\al^2, q;q)_\infty}
\frac{\al^{-(2n+2)} q^{-\frac12 n(n+1)}}{(q;q)_n}
\langle v, \phi_{x_0} \rangle \langle \phi_{x_0}, w \rangle 
\end{gather*}
It follows that the eigenspace is one-dimensional spanned by $\phi_{q^n}$, 
since $E\bigl(\{q^n\}\bigr)$ is a rank one projection onto 
the space spanned the eigenvector $\phi_{q^n}$.
Plugging in $v=w=\phi_{q^n}$ then gives
\begin{gather*}
\|\phi_{q^n}\|^2 = \left<E\bigl(\{q^n\}\bigr)\phi_{q^n},\phi_{q^n}\right> 
= 
\frac{(-\al^2q;q)_\infty}{(-1/\al^2, q;q)_\infty}
\frac{\al^{-(2n+2)} q^{-\frac12 n(n+1)}}{(q;q)_n} \|\phi_{q^n}\|^4 \quad
\Longrightarrow \\
\|\phi_{q^n}\|^2 = 
\frac{(-1/\al^2, q;q)_\infty}{(-\al^2q;q)_\infty}
\al^{(2n+2)} q^{\frac12 n(n+1)} (q;q)_n
\end{gather*}
Since $\{0\}$ is not a discrete mass point, see Exercise 
\ref{exer:zeroEVcontqinvHermite}, we see that the spectrum 
of $L$ is $q^\N \cup \{0\}$ and that we have an orthogonal
basis of eigenvectors $\{ \phi_{q^n}\}_{n\in \N}$ for $\ell^2(\Z)$.

It turns out that we can rewrite the orthogonality of the 
eigenvectors $\{ \phi_{q^n}\}_{n\in \N}$ in terms of 
orthogonality relations for orthogonal polynomials, 
namely for the continuous $q^{-1}$-Hermite polynomials. 
\index{orthogonal polynomials!continuous $q^{-1}$-Hermite polynomials}
This is not a coincidence, since we started out with the second order $q$-difference
operator having these polynomials as eigenfunctions. 
Of course, this can be done since the continuous $q^{-1}$-Hermite polynomials 
are in the $q$-Askey scheme. 
Writing down the orthogonality relations explicitly gives 
\begin{equation}\label{eq:orthocontqinvHermite}
\sum_{l=-\infty}^\infty \al^{4l} q^{2l^2-l} (1+\al^2 q^{2l})
\, h_n(x_l(\al)|q) h_m(x_l(\al)|q) = \de_{n,m}
q^{-n(n+1)/2} (q;q)_n (-\al^2,-q/\al^2,q;q)_\infty. 
\end{equation}
where the polynomials are generated by the monic three-term recurrence relation
\[
x h_n(x|q) = h_{n+1}(x|q) + q^{-n}(1-q^n) h_{n-1}(x|q),
\qquad h_{-1}(x|q) =0, \ h_0(x|q) =1, 
\]
and the mass points are $x_l(\al)= \frac12 ((\al q^l)^{-1}-\al q^l)$. 
By the completeness of the basis of eigenvectors $\{ \phi_{q^n}\}_{n\in \N}$ 
it follows that the polynomials are dense in the weighted $L^2$-space of the 
corresponding discrete measures in \eqref{eq:orthocontqinvHermite}. 
Since $\al\in (q,1]$, for each $\xi\in \R$ there is a measure of the 
type in \eqref{eq:orthocontqinvHermite} with positive mass in $\xi$.
It follows from the general theory of moment problems \cite{Akhi}, \cite{BuchC} that 
\eqref{eq:orthocontqinvHermite} gives all N-extremal measures for the 
continuous $q^{-1}$-Hermite polynomials. 
The same result (and more) on the N-extremal measures has been obtained previously by
Ismail and Masson \cite{IsmaM} by calculating explicitly the functions
in the Nevanlinna parametrisation. 
\end{example}

\begin{example}\label{exa:l2Z-continuousspectrum} 
The example discussed in Examples \ref{exa:ASC-q1}, \ref{exa:ASC-q2}, \ref{exa:ASC-q3} 
is relatively easy, since $L$ is bounded, and even compact. Another well studied 
three-term recurrence operator
on $\ell^2(\Z)$ is the following unbounded operator 
\begin{gather*}
2L\, e_k \, = a_k\, e_{k+1} + b_k\, e_k + a_{k-1}\, e_{k-1}, \\
a_k \, = \sqrt{\left(1-\frac{q^{k+1}}{z}\right)\left(1-\frac{cq^{k+1}}{d^2z}\right)},
\qquad b_k =\frac{q^{k}(c+q)}{dz}.
\end{gather*}
assuming $z<0$, $0<c<1$, $d\in\R\setminus\{0\}$. 
The operator $L$ is essentially self-adjoint for $0 < c \leq q^2$, and the 
spectral decomposition has an absolutely continuous part and a discrete part, 
with infinite number of points. This can be proved in the same way as in this section,
where basic hypergeometric series play an important role in finding the 
(free) solutions to the eigenvalue equation $L^\ast f = zf$. 
The corresponding spectral decomposition leads to an integral transform known as the little $q$-Jacobi 
function transform, see \cite{KoelSNATO}.\index{little $q$-Jacobi function transform}
The quantum group theoretic interpretation  goes back to 
Kakehi \cite{Kake}, see also \cite[App.~A]{KoelSsu}. 
This result, including a suitable self-adjoint extension for the case $c=q$ 
and its spectral decomposition, can be found in \cite[App.~B, C]{GroeKK}.
In \cite{KoelSNATO} it is described how the little $q$-Jacobi 
function transform can be viewed as a non-polynomial addition to the 
$q$-Askey scheme. 
\end{example}

\begin{remark}
The solution space of the three term recurrence is two-dimensional, 
so that the dimension of $\dim S^\pm_z$ is determined by  
summability conditions at $\pm \infty$.  
In case one of $\dim S^\pm_z$ is bigger than $1$, we have higher deficiency 
indices. In case one of $S^\pm_z$ is one-dimensional, and the other is
$2$-dimensional, we have deficiency indices $(1,1)$. In case both spaces are
two-dimensional, the deficiency indices are $(2,2)$.
This is an observation essentially due to Masson and Repka \cite{MassR}. 
For an example of such a three-term recurrence relation with deficiency 
indices $(1,1)$, see \cite{Koel-IM}. 
\end{remark}

\subsection{Exercises}

\begin{enumerate}[1.]
\item\label{exer:lembbdL} Prove Lemma \ref{lem:bddL}. 
\item\label{exer:adjointL} Prove Lemma \ref{lem:adjointL}. 
\begin{enumerate}[(a)]
\item Recall the definition of the domain of the adjoint operator of $(L,\cD)$
from Section \ref{ssec:appunboundedsaoperators}, so we have to find all 
$w\in \ell^2(\Z)$ for which $\cD\ni v\mapsto \langle Lv, w\rangle$ is continuous.
This is the same as requiring the existence of a constant $C$ so that 
$|\langle Lv, w\rangle|\leq C \|v\|$ for all $v\in \cD$. Write for $v\in \cD$ 
\begin{gather*}
\langle Lv, w\rangle  = \sum_{\underset{\scriptstyle{\textrm{finite}}}{k\in \Z}} 
v_k\, \overline{ \bigl( \overline{a_k}w_{k+1}+ \overline{b_k} w_k+ \overline{c_{k}}w_{k-1}\bigr)}
\end{gather*}
and use Cauchy-Schwarz to prove that $\cD^\ast$ is contained in 
the domain of the adjoint of $(L,\cD)$.
\item Show conversely that any element in the domain of the adjoint is 
element of $\cD^\ast$. (Hint: Use the identity in (a) and take a special choice
for $v\in \cD$ which converges to an element of $\cD^\ast$.)
\item Finish the proof of Lemma \ref{lem:adjointL}. 
\end{enumerate}

\item\label{exer:SPmzforcontqinvHermite} Prove that in Example 
\ref{exa:ASC-q2} the spaces $S^{\pm}_z$ are indeed spanned by the elements given. 
\begin{enumerate}[(a)]
\item Show that $\sum_{l\in\Z}(\phi_z)_le_l$ is a formal eigenvector of $L$.
(Hint: This is not directly deducible from the expression as ${}_0\vp_1$, first transform
to a ${}_2\vp_1$, see \cite{GaspR}, and use contiguous relations for ${}_2\vp_1$.
See \cite{ChriK} for details.)
\item Next show that $\sum_{l>0} |(\phi_z)_l|^2<\infty$. (Hint: use 
the asymptotic behaviour of $C_l(\al)$  as $l\to \infty$.)
\item Conclude that $\phi_z\in S^+_z$.
\item Let $V\colon \ell^2(\Z)\to \ell^2(\Z)$ be the unitary involution $e_l\mapsto e_{-l}$.
Denote $L(\al)=L$ for the operator $L$ as in Example \ref{exa:ASC-q1} 
to stress the dependence on $\al$.
Show that $L(1/\al) = VL(\al)V^\ast$. 
Conclude that $\Phi_z\in S^-_z$.
\item Calculate the Casorati determinant or Wronskian $[\phi_z, \Phi_z]$ 
by taking the limit $l\to \infty$ in Lemma \ref{lem:Wronskian} 
using the asymptotic behaviour of $a_l$ as
in Example \ref{exa:ASC-q1} and
\[
\lim_{x\downarrow 0} \rphis{0}{1}{-}{1/x}{q, \frac{z}{x}} =  (z;q)_\infty
\]
Show that $[\phi_z, \Phi_z]=-z(1/z;q)_\infty$ for $z\not=0$, by taking the limit $l\to\infty$ in 
the Casorati determinant or Wronskian using Lemma \ref{lem:Wronskian}. 
\end{enumerate}
\item\label{exer:zeroEVcontqinvHermite} Show that in Example 
\ref{exa:ASC-q2} there is no eigenvector, i.e. in $\ell^2$, for the eigenvalue 
$0$. (Hint: show that $(-1)^l q^{-\frac12 l}\sqrt{1+\al^2 q^{2l}}$
as well as $(-1)^l q^{-\frac32 l}(1-q^l)(1+\al^2 q^l)\sqrt{1+\al^2 q^{2l}}$
give two linearly independent solutions for the recurrence for $z=0$, 
and that there is no linear combination which is square summable.)

\item\label{exer:realcoefficients} Show that in general we
can take $\overline{\phi_{\bar z}}\in S^+(z)$, next put 
\begin{equation*}
G_{k,l}(z) = \frac{1}{[\overline{\phi_{\bar z}}, \Phi_z]}
\begin{cases}
(\Phi_z)_k(\overline{\phi_{\bar z}})_l, & k\leq l, \\[2pt]
(\Phi_z)_l(\overline{\phi_{\bar z}})_k, & k>l,
\end{cases}
\end{equation*}
and show that the resolvent $R(z)$ can be 
obtained as in the proof of Theorem \ref{thm:resolventL}.

\item\label{exer:2times2} Rewrite the operator $L$ as three-term recurrence relation
labeled by $\N$ by considering $\C^2$-vectors 
\[
u_k = \begin{pmatrix} e_k\\ e_{-k-1}\end{pmatrix}, \quad k\in \N
\]
and define  
\begin{gather*}
\cL u_k = \begin{pmatrix} L e_k\\ L e_{-k-1}\end{pmatrix}.
\end{gather*}
and write $\cL$ as a three-term recurrence in terms of $u_k$ with $2\times 2$
matrices acting on naturally on $\ell^2(\N)\hat \otimes \C^2 \cong \ell^2(\Z)$.
Determine the matrices in the three-term recurrence explicitly in terms of 
the coefficients of $L$ in \eqref{eq:defL}. 
See also Section \ref{ssec:MVOPlinktoell2Z}.

\end{enumerate}


\section{Three-term recurrence relations and orthogonal polynomials}\label{sec:threetermN}

In this section we consider three-term recursion relations labeled by 
$l\in\N$, and we relate such operators to orthogonal polynomials and  
the moment problem.

\subsection{Orthogonal polynomials}

Assume $\mu$ is a positive Borel measure on the real line $\R$ 
with infinite support such that 
all moments 
\[
m_k = \int_\R x^k \, d\mu(x) <\infty
\]
exist. We assume the normalisation of $\mu$ by $m_0=\mu(\R)=1$, so that we have a 
probability measure. 

Note that all polynomials are contained in the Hilbert space $L^2(\mu)$. 
Then we can apply the Gram-Schmidt procedure to $\{1,x,x^2,x^3, \cdots\}$
to obtain a sequence of polynomials $p_n(x)$ of degree $n$ so that
\begin{equation}\label{eq:scalarop-orthorel}
\int_\R p_m(x)\, \overline{p_n(x)}\, d\mu(x) = \de_{m,n}.
\end{equation}
These polynomials form a family of orthogonal polynomials.\index{orthogonal polynomials}
We normalise the leading coefficient of $p_n$ to be positive, which 
can also be viewed as part of the Gram-Schmidt procedure. 
Observe also that, since all moment $m_k$ are real, the polynomials
have real coefficients, so we do not require complex conjugation in
\eqref{eq:scalarop-orthorel}.

\begin{thm}[Three term recurrence
relation] \index{three-term recursion}
\label{thm:3termrec-scalarOP}
Let $\{ p_k\}_{k=0}^\infty$
the orthonormal polynomials in $L^2(\mu)$,
then there exist sequences $\{ a_k\}_{k=0}^\infty$,
$\{ b_k\}_{k=0}^\infty$,  with $a_k>0$ and $b_k\in\R$, such
that
\begin{gather*}
x\, p_k(x) =  a_kp_{k+1}(x) + b_kp_k(x) + a_{k-1} p_{k-1}(x), \qquad
k\geq 1, \\
x\, p_0(x) = a_0p_1(x) + b_0p_0(x).
\end{gather*}
If $\mu$ is compactly supported, then the
sequences $\{ a_k\}_{k=0}^\infty$,
$\{ b_k\}_{k=0}^\infty$ 
are bounded.
\end{thm}

We leave the proof of Theorem \ref{thm:3termrec-scalarOP} as Exercise \ref{exer:thm3termrec-scalarOP}, 
where $a_n$ and $b_n$ are expressed as integrals. 

Conversely, given arbitrary coefficient sequences $\{a_n\}_{n\in \N}$ and $\{b_n\}_{n\in \N}$ 
with $a_n>0$, $b_n\in \R$ for all $n\in \N$, we see that the recursion of 
Theorem \ref{thm:3termrec-scalarOP} determines the polynomials $p_n(x)$ with 
the initial condition $p_0(x)=1$. 
In order to study these polynomials, one can study 
the Jacobi operator\index{Jacobi operator} 
\begin{equation}\label{eq:Jacobioperator}
J\, e_k = \begin{cases} a_k\, e_{k+1} + b_k\, e_k +
a_{k-1} \, e_{k-1}, & k\geq 1, \\
a_0\, e_1 + b_0\, e_0, & k=0. \end{cases}
\end{equation}
as an operator on the Hilbert space $\ell^2(\N)$ 
with orthonormal basis $\{ e_k\}_{k\in \N}$. 
Note that we can study such a Jacobi operator without 
assuming the situation of 
Theorem \ref{thm:3termrec-scalarOP}, i.e. arising from 
a Borel measure with finite moments. 
So we generate polynomials $\{p_n\}_{n\in \N}$ from 
the three-term recurrence 
relation of Theorem \ref{thm:3termrec-scalarOP}, but now
with the coefficients from the Jacobi operator. 
Note that once $p_0(z)$ is fixed, the polynomials are
determined. We assume that $p_0(z)=1$. 
See Section \ref{ssec:Jacobi-operator} for more information. 

Initially, $J$ is defined on the dense linear subspace $\cD$ 
of finite linear combinations of the 
orthonormal basis $\{ e_k\}_{k\in \N}$. 
It follows from \eqref{eq:Jacobioperator} and Theorem 
\ref{thm:3termrec-scalarOP} that, at least formally, we have found
eigenvectors for $J$;
\begin{equation}\label{eq:formaleigvecJ}
J\left( \sum_{k=0}^\infty p_k(z) e_k\right) = z\, \sum_{k=0}^\infty p_k(z) e_k.
\end{equation}
However, we haven't defined $J$ on arbitrary vectors and in general 
$\sum_{k=0}^\infty p_k(z) e_k\notin \ell^2(\N)$, but 
\eqref{eq:formaleigvecJ} indicates that there is a relation between the spectrum 
of $J$ and the orthonormal polynomials. By looking at a partial sum of
\eqref{eq:formaleigvecJ}, the left hand side is well-defined. 

\begin{lemma}\label{lem:finiteJacobioperator} For $M\in \N$ 
\[
J\left( \sum_{k=0}^M p_k(z) e_k\right) = z\, \sum_{k=0}^M p_k(z) e_k 
+ a_{M} p_{M}(z)e_{M+1} - a_Mp_{M+1}(z) e_M
\]
\end{lemma}

Truncating 
$J$ to a $(M+1)\times (M+1)$-matrix, which we denote by $J_M$, we see that 
--using $\{e_0,\cdots, e_M\}$ as the standard basis-- 
\[
J_M\left( \sum_{k=0}^M p_k(z) e_k\right) = z\, \sum_{k=0}^M p_k(z) e_k - a_Mp_{M+1}(z) e_M.
\]
Since $J_M$ is a self-adjoint matrix, and since its eigenspaces are $1$-dimensional,
we obtain the following corollary.

\begin{cor}\label{cor:lemfiniteJacobioperator} For $M\in \N$, the zeroes of 
$p_{M+1}$ are real and simple.
\end{cor}

We now study the orthonormal polynomials of Theorem \ref{thm:3termrec-scalarOP}
by studying the Jacobi operator $(J, \cD)$. 

\begin{lemma}\label{lem:adjointofJacobioperator}
The adjoint $(J^\ast, \cD^\ast)$ is given by 
\[
\cD^\ast = \{ v=\sum_{k=0}^\infty v_ke_k \in \ell^2(\N) \mid \sum_{k=0}^\infty 
(a_kv_k + b_k v_k + a_{k-1}v_{k-1})e_k \in \ell^2(\N)\}
\]
and $J^\ast v = \sum_{k=0}^\infty 
(a_kv_k + b_k v_k + a_{k-1}v_{k-1})e_k$ for $v\in \cD^\ast$ of this form. 
\end{lemma}

The proof of Lemma \ref{lem:adjointofJacobioperator} is 
completely analogous to the proof of Lemma \ref{lem:adjointL}, see
Exercise \ref{exer:lemadjointofJacobioperator}. 

In order to study the Jacobi operator we find another solution 
to the corresponding eigenvalue equation for $J$. Since the formal eigenspace
of $J$ is 1-dimensional, we can only find a solution of the 
equation $\langle Jv, e_k\rangle = x \langle v, e_k\rangle$ for 
$k\geq 1$. 
Let $r_k(x)$ be the sequence of
polynomials generated by the three-term recurrence of Theorem \ref{thm:3termrec-scalarOP}
for $k\geq 1$ 
with initial conditions $r_0(x)=0$ and
$r_1(x)=a_0^{-1}$. Obviously, $r_k$ is a polynomial of degree $k-1$. 
The polynomials
$\{ r_k\}_{k=0}^\infty$ are known as the associated polynomials\index{associated polynomials}
or polynomials of the second kind.\index{polynomials of the second kind}
In case we assume that the Jacobi operator \eqref{eq:Jacobioperator} 
comes from the three-term recurrence relation for orthogonal polynomials 
as in Theorem \ref{thm:3termrec-scalarOP}, we can describe the polynomials
$r_k$ explicitly in terms of the measure $\mu$. This is done in 
Lemma \ref{lem:associatedpols}. 

\begin{lemma}\label{lem:associatedpols} Let 
\[
w(z) = \int_\R \frac{1}{x-z}\, d\mu(x)
\]
be the Stieltjes transform of the measure $\mu$, which 
is well-defined for $z\in \C\setminus\R$. We have that 
\[
r_k(x) = \int_\R \frac{p_k(x)-p_k(y)}{x-y} \, d\mu(y)
\]
and for $z\in \C\setminus\R$ 
\[
\sum_{k=0}^\infty |w(z) p_k(z) + r_k(z)|^2 \leq 
\int_\R \frac{1}{|x-z|^2}\, d\mu(x) \leq \frac{1}{|\Im(z)|^2}<\infty
\]
\end{lemma}

\begin{proof} 
We leave the explicit expression of $r_k$ as Exercise \ref{exer:lemassociatedpols}.
In the Hilbert space $L^2(\mu)$ we consider the expansion of 
the 
function $x\mapsto \frac{1}{x-z}$ for $z\in \C\setminus\R$, which 
is an element of $L^2(\mu)$ by the estimate $|\frac{1}{x-z}|\leq \frac{1}{|\Im(z)|}$ 
and $\mu$ being a probability measure. 
We calculate the inner product of $x\mapsto \frac{1}{x-z}$ with an 
orthonormal polynomial $p_k$;
\begin{gather*}
\int_\R \frac{p_k(x)}{x-z}\, d\mu(x) = 
\int_\R \frac{p_k(x)-p_k(z)}{x-z}\, d\mu(x) + p_k(z) \int_\R \frac{1}{x-z}\, d\mu(x) 
= r_k(z) + w(z) p_k(z)
\end{gather*}
By the Bessel inequality\index{Bessel inequality}  
for the orthonormal sequence $\{p_k\}_{k\in \N}$ in $L^2(\mu)$
the result follows.
\end{proof}

As a corollary to Lemma \ref{lem:associatedpols} we get that 
\begin{equation}\label{eq:Markovscalar}
\lim_{k\to \infty} \frac{r_k(z)}{p_k(z)} = -w(z) 
= \int_\R \frac{1}{z-x}\, d\mu(x), 
\qquad  z\in \C\setminus\R. 
\end{equation}
which is known as Markov's\index{Markov's theorem} theorem, see \cite{Berg-Markov} for an overview. 

In particular, we see that the vector 
\begin{equation*}
f(z) = \sum_{k=0}^\infty \bigl( r_k(z) + w(z) p_k(z)\bigr) e_k \in \ell^2(\N) 
\end{equation*}
for $z\in \C\setminus\R$, and 
satisfying $\langle J^\ast f(z), e_k \rangle = z \langle f(z), e_k \rangle$
for $k\geq 1$. 
We view $f(z)$ as the free solution\index{free solution} in this case. 
So it is a square summable solution for the three-term recurrence relation
for $k\gg 0$. 
From here we can define the Green function and calculate the 
resolvent explicitly. 
Under the assumption that $\sum_{n\in \N} |p_n(z)|^2$ diverges for 
$z\in \C\setminus\R$ this 
can be obtained from Section \ref{ssec:MVresolventoperator} 
by specialising to $N=1$. 

\subsection{Jacobi operators}\label{ssec:Jacobi-operator}

The converse problem, namely finding the orthogonality measure $\mu$ 
for the polynomials $\{p_k\}_{k\in \N}$ generated by 
a three-term recurrence relation as of Theorem \ref{thm:3termrec-scalarOP}, 
can be solved by studying the Jacobi operator of \eqref{eq:Jacobioperator}.
The operator $(J, \cD)$, with adjoint $(J^\ast, \cD^\ast)$ as in 
Lemma \ref{lem:adjointofJacobioperator}, can be studied from a  spectral point of view. 

\begin{prop}\label{prop:spectralJacobioperator}
The deficiency indices $(n_+, n_-)$ of $(J, \cD)$  are $(0,0)$ or $(1,1)$. 
In case $(n_+, n_-)=(0,0)$ the operator $(J, \cD)$ is essentially self-adjoint. 
Let $E$ be the spectral decomposition of $(J^\ast, \cD^\ast)$ in case
$(n_+, n_-)=(0,0)$ and of a self-adjoint extension $(J_\theta, D(J_\theta))$,
$(J, \cD)\subset (J_\theta, D(J_\theta))\subset (J^\ast, \cD^\ast)$
in case $(n_+, n_-)=(1,1)$. 
Then an orthogonality measure for the polynomials is given by 
$\mu(B) = \langle E(B)e_0, e_0\rangle$, $B\in\scB$. 
\end{prop}

\begin{proof} The deficiency indices are equal, since $J^\ast$ commutes 
with conjugation. Since the eigenvalue equation $J^\ast v = z v$ 
is completely determined
by the initial value $\langle v, e_0\rangle$, the deficiency space is
at most $1$-dimensional. 
Note that $J^\ast v = z v$ gives $\langle v, e_n\rangle = p_n(z) 
\langle v, e_0\rangle$, so that 
the defect indices are $(1,1)$ if and only if $\sum_{n=0}^\infty |p_n(z)|^2 <\infty$. 

Also, $e_0$ is a cyclic vector of $\ell^2(\N)$ for $J$, i.e. 
$\ell^2(\N)$ equals the closure of the space of $J^k e_0$, $k\in \N$
and even $e_k=p_k(J)e_0$, which follows by induction on $k$. 
Since $J^\ast$ or $J_\theta$ extend $J$, we have 
\begin{gather*}
\de_{k,l} = \langle e_k, e_l \rangle = 
\langle p_k(J)e_0, p_l(J)e_0\rangle = 
\langle p_l(J)p_k(J)e_0, e_0\rangle \\ 
= \int_\R p_l(\la)p_k(\la) \, dE_{e_0,e_0}(\la) 
= \int_\R p_l(\la)p_k(\la) \, d\mu(\la)
\end{gather*}
using the spectral theorem for self-adjoint operators in Appendix \ref{app:spectralthm}. 
\end{proof}

\begin{cor}[Favard's theorem]\label{cor:Favardsthm}\index{Favard's theorem}
Let the polynomials $p_n$ of degree $n$ be generated by the recursion $p_0(z)=1$, 
$p_1(z)=a_1^{-1}(z-b_0)$
and 
\[
zp_n(z) = a_n p_{n+1}(z) + b_n p_n(z) + a_{n-1}p_{n-1}(z)
\]
for sequences $\{a_n\}_{n\in \N}$,  $\{b_n\}_{n\in \N}$ with $a_n>0$ and $b_n\in \R$ 
for all $n$. Then there exists a Borel measure on $\R$ with finite moments so that 
$ \int_\R p_n(x)p_m(x) \, d\mu(x) = \de_{m,n}$. 
\end{cor}

\begin{remark}\label{rmk:propspectralJacobioperator}
According to Proposition \ref{prop:App-extensionselfadjointopertaotrs} 
the labeling of the self-adjoint extension of Proposition 
\ref{prop:spectralJacobioperator} is given by $U(n_+)=U(1)$, so we can think
of $\theta\in [0,2\pi)$ as parametrising the self-adjoint extensions 
of $(J,\cD)$ in Proposition \ref{prop:spectralJacobioperator}. 
It can then be proved that the corresponding orthogonality
measures for different self-adjoint extensions lead to 
different Borel measures for the orthogonal polynomials, 
see e.g. \cite[Ch.~XII.8]{DunfS}, \cite[Thm.~(3.4.5)]{Koel-Laredo}, \cite[Ch.~16]{Schm}, 
\end{remark}

Note that in particular, we see that the condition 
$\dim S^\pm_z \geq 1$, mentioned immediately after Theorem \ref{thm:resolventL},
follows by considering the two Jacobi operators associated to $L$ by considering
$k\to \infty$ and $k\to -\infty$. 
In a fact, a theorem by Masson and Repka \cite{MassR},
states the deficiency indices of the operator $L$ of
Section \ref{sec:threetermZ} can be obtained by adding the deficiency indices of
the Jacobi operators $k\to \infty$ and $k\to -\infty$.

\subsection{Moment problems}

The moment problem\index{moment problem} is the following:
\begin{enumerate}[1.]
\item Given a sequence $\{m_0,m_1, m_2, \ldots\}$, does there
exist a positive Borel measure $\mu$ on $\R$
such that $m_k=\int x^k\, d\mu(x)$?
\item If the answer to problem 1 is yes, is the measure obtained
unique?
\end{enumerate}

We exclude the case of finite discrete orthogonal polynomials, so
we assume $\supp(\mu)$ is not a finite set. 
This is equivalent to the Hankel matrix $(m_{i+j})_{0\leq i,j\leq N}$ being 
regular for all $N\in \N$. 
We do not discuss the conditions for existence of such a measure. 
The Haussdorf moment problem (1920) requires $\supp(\mu)\subset [0,1]$. 
The Stieltjes moment problem (1894) requires $\supp(\mu)\subset [0,\infty)$. 
The Hamburger moment problem (1922) does not require a 
condition on the support of the measure. 
See Akhiezer \cite{Akhi}, Buchwalter and Cassier \cite{BuchC},
Dunford and Schwartz \cite[Ch.~XII.8]{DunfS}, 
Schm\"udgen \cite[Ch.~16]{Schm}, Shohat and Tamarkin \cite{ShohT},
Simon \cite{Simo}, Stieltjes \cite{Stie}, Stone \cite{Ston} for more 
information. 

The fact that the measure is not determined by its moments was first noticed
by Stieltjes in his famous memoir \cite{Stie}, published 
posthumously. See Kjeldsen \cite{Kjel} for an overview of the early history of the moment
problem. Stieltjes's example is discussed in Exercise \ref{exer:Stieltjesexample}.

So we see that the moment problem is determinate --i.e. the answer to 2 is yes--
if and only if the corresponding Jacobi operator is essentially self-adjoint.

\subsection{Exercises}

\begin{enumerate}[1.]
\item\label{exer:thm3termrec-scalarOP} Prove Theorem \ref{thm:3termrec-scalarOP}.  
\begin{enumerate}[(a)]
\item Prove that there is a three-term recurrence relation. (Hint: Expand $xp_n(x)$ in the polynomials, and use that multiplying by $x$ is 
(a possibly unbounded) symmetric operator on the space of polynomials in 
$L^2(\mu)$, since $\mu$ is a real Borel measure.)
\item Establish $a_n=\int_\R xp_n(x)p_{n+1}(x)\, d\mu(x)$ and 
$b_n=\int_\R x\bigl(p_n(x)\bigr)^2\, d\mu(x)$.
\item Show that if $\mu$ has bounded support that the coefficients $a_n$ and $b_n$ are bounded.
(Hint: If $\supp(\mu)\subset [-M,M]$ then one can estimate $x$ in the integrals by $M$, and 
next use the Cauchy-Schwarz inequality in $L^2(\mu)$.)
\end{enumerate}
\item\label{exer:lemadjointofJacobioperator} Prove Lemma \ref{lem:adjointofJacobioperator}.
(Hint. Consider the proof as in Exercise \ref{exer:adjointL}.)
\item\label{exer:lemassociatedpols} Prove the explicit expression for $r_k$ of Lemma 
\ref{lem:associatedpols}. (Hint: write
\begin{gather*}
x\bigl(p_k(x) - p_k (y)) + (x - y)p_k (y) = \\
a_k \bigl(p_{k+1} (x) - p_{k+1}(y)\bigr) + b_k \bigl(p_k (x) - p_k (y)\bigr) + 
a_{k-1} \bigl(p_{k-1}(x) - p_{k-1}(y)\bigr)
\end{gather*}
using the three-term recurrence relation. Divide by $x-y$ and integrate with respect to 
$\mu$. Then the second term on the left hand side vanishes for $k\geq 1$. 
Check the initial values as well.)
\item\label{exer:Stieltjesexample} 
\begin{enumerate}[(a)]
\item 
Show that for $\ga>0$ 
\begin{gather*}
\int_0^\infty x^n \, e^{-\ga^2\ln^2 x}\sin (2\pi\ga^2\ln x)\, dx = 0, \qquad \forall\, n\in \N.
\end{gather*}
(Hint: switch to $y =\ga \ln(x) -\frac{1}{2\ga}(n+1)$.)
\item Conclude that the moments 
$\int_0^\infty x^n \, e^{-\ga^2\ln^2 x}\bigl(
1+r\sin (2\pi\ga^2\ln x)\bigr)\, dx$ are independent of $r$, and this is a 
positive measure for $r\in \R$ with $|r|\leq 1$. 
\end{enumerate}
\item\label{exer:ChristoffelDarbouxscal} 
Prove the Christoffel-Darboux formula\index{Christoffel-Darboux formula} for 
the orthonormal polynomials using the three-term recurrence relation;
\begin{equation*}
(x-y) \sum_{k=0}^{n-1} p_k(x)p_k(y) = a_{n-1}\bigl(
p_n(x)p_{n-1}(y) - p_{n-1}(x)p_n(y)\bigr)
\end{equation*}
and derive the limiting case 
$$
\sum_{k=0}^{n-1} p_k(x)^2  = a_{n-1}\bigl(
p_n'(x)p_{n-1}(x) - p_{n-1}'(x)p_n(x)\bigr).
$$
\end{enumerate}

\section{Matrix-valued orthogonal polynomials}\label{sec:MVOP}

In this section we study matrix-valued orthogonal polynomials 
using a spectral analytic description of the corresponding 
Jacobi operator. In this we follow  \cite{ApteN},  \cite{DamaPS}, \cite{Gero},
and references given there, in particular in \cite{DamaPS}. 

\subsection{Matrix-valued measures and related polynomials}\label{ssec:MVmeasurespols}

We consider $\C^N$ as a finite dimensional inner product space with standard
orthonormal basis $\{e_i\}_{i=1}^N$. 
By $M_N(\C)$\index{M@$M_N(\C)$} we denote the matrix algebra of linear maps 
$T\colon \C^N\to \C^N$, 
Let $E_{i,j}\in M_N(\C)$\index{E@$E_{i,j}$} be the rank one operators\index{rank one operator} $E_{i,j} v = \langle v,e_j\rangle e_i$, so that 
$E_{i,j}e_k = \de_{k,j}e_i$. So $E_{i,j}$ is the $N\times N$-matrix with all zeroes, except one $1$ at the $(i,j)$-th entry.  
Note that in particular $\C^N$ is a (finite-dimensional) Hilbert space, 
see Example \ref{exa:Hilbertspaces}, 
so that $M_N(\C)$ carries a norm and with this norm 
$M_N(\C)$ is a $\text{C}^\ast$-algebra, see Section \ref{ssec:HilbertCastmodules}. 

A linear map $T\colon \C^N\to \C^N$ is positive\index{positive linear map}, or
positive definite,\index{positive definite linear map} in case 
$\langle Tv,v\rangle > 0$ for all $v\in \C^N\setminus\{0\}$, which we denote by $T> 0$. 
$T$ is positive semi-definite\index{positive semi-definite linear map} if
$\langle Tv,v\rangle \geq 0$ for all $v\in \C^N$, denoted by $T\geq 0$. 
The space of positive linear semi-definite maps, or 
positive semi-definite matrices\index{positive semi-definite matrix} (after fixing a basis), is 
denoted by $P_N(\C)$\index{P1@$P_N(\C)$}.
$P_N(\C)$ is a closed cone in $M_N(\C)$. Its interior $P^o_N(\C)$\index{P2@$P^o_N(\C)$} is the open cone of positive matrices. 
Note that each positive linear map is
Hermitean, see \cite[\S 7.1]{HornJ}.
Then we set $T>S$ if $T-S>0$ and $T\geq S$ if $T-S\geq 0$, 
see Section \ref{ssec:HilbertCastmodules}.

\begin{defn}\label{def:MV-measure}
A matrix-valued measure (or matrix measure) \index{matrix-valued measure}\index{matrix measure}
is a $\si$-additive map $\mu\colon \scB \to P_N(\C)$
where $\scB$ is the Borel $\si$-algebra on $\R$.
\end{defn}

Recall that $\si$-additivity means that for any sequence $E_1, E_2, \cdots$ of pairwise disjoint 
Borel sets, we have 
\[
\mu\left( \bigcup_{k=1}^\infty E_k\right) = \sum_{k=1}^\infty \mu(E_k)
\]
where the right-hand side is unconditionally convergent in $M_N(\C)$. 

Note $\mu_{v,w}\colon \scB \to \C$, $\mu_{v,w}(B) = \langle \mu(B)w,v\rangle$ is 
a complex-valued Borel measure on $\R$, and in particular 
$\mu_{v,v}\colon \scB \to \R$ is a positive Borel measure on $\R$. 
Let $\tau_\mu= \sum_{i=1}^N \mu_{i,i}$ be the positive Borel measure on $\R$ corresponding 
to the trace of $\mu$, i.e. $\tau_\mu(B) = \Tr(\mu(B))$ for all $B\in \scB$. 
Here we use the notation $\mu_{i,j} = \mu_{e_i,e_j}$, but note that the 
trace measure $\tau_\mu$ is independent of the choice of basis for $\C^N$. 
The following result is \cite[Thm.~1.12]{Berg}, see also \cite[\S 1.2]{DamaPS},  
\cite[\S 3]{Krei-mm}, 
\cite{Rose}. 

\begin{thm}\label{thm:RDforMVmeasure}
For a matrix measure $\mu$ there exist functions $W_{i,j} \in L^1(\tau_\mu)$ such that 
\[
\mu_{i,j}(B) = \int_B W_{i,j}(x)\, d\tau_\mu(x), \qquad \forall\, B\in \scB
\]
and $W(x)=\bigl( W_{i,j}(x)\bigr)_{1\leq i,j\leq N} \in P_N(\C)$ for $\tau_\mu$-almost $x$. 
\end{thm}

The proof is based on the fact that for a positive definite matrix $A=(a_{i,j})_{i,j=1}^N$ we have 
$|a_{i,j}|\leq \sqrt{a_{i,i}a_{j,j}} \leq \frac12( a_{i,i}+a_{j,j}) \leq \Tr(A)$ and the Radon-Nikodym theorem, see \cite{Berg} for details.
The first inequality follows from considering a positive definite $2\times 2$-submatrix, and the second
by the arithmic-geometric mean inequality. 
Note that this inequality also implies that $W(x)\leq I$ $\tau_\mu$-almost everywhere (a.e.), see also \cite[Lemma~2.3]{Rose}.
The measure $\tau_\mu$ is regular, see \cite[Thm.~2.18]{Rudi-RCA}, 
\cite[Satz I.2.14]{Wern}.

\begin{ass}\label{ass:finitemomentsandstrictlypos-ae}
From now on we assume for Section \ref{sec:MVOP} that $\mu$ is a matrix measure for which $\tau_\mu$ 
has infinite support and for which all moments exist, i.e.
$(x\mapsto x^k W_{i,j}(x)) \in L^1(\tau_\mu)$ for all $1\leq i,j\leq N$ and all $k\in \N$.
Moreover, we assume that, in the notation of Theorem \ref{thm:RDforMVmeasure}, 
the matrix $W$ is positive definite $\tau_\mu$-a.e., i.e. $W(x)>0$ $\tau_\mu$-a.e.
\end{ass}

Note that we do not assume that the weight is irreducible in a suitable sense, but
we discuss the reducibility issue briefly in Section \ref{ssec:MVreducibility}. 

By 
\begin{equation*}
M_k = \int_\R x^k \, d\mu(x) \in M_N(\C), \qquad (M_k)_{i,j} = \int_\R x^k W_{i,j}(x)\, d\tau_\mu,  
\end{equation*}
we denote the corresponding moments in $M_N(\C)$. Note that the even moments 
are positive definite, i.e. $M_{2k} \in P_N^o(\C)$. 

Given a weight function as Assumption \ref{ass:finitemomentsandstrictlypos-ae} 
we can associate matrix-valued orthogonal polynomials $P_n$ so that 
\begin{equation}\label{eq:MVOPorthonormal}
\int_\R P_n(x)\, W(x)\, P_m^\ast(x) \,  d\tau_\mu(x) = \de_{m,n}I
\end{equation}
where $P_m^\ast(z) = (P_m(\bar z))^\ast$ for $z\in \C$, so that the 
$P_m^\ast(z)= \sum_{k=0}^m A_k^\ast z^k$ if 
$P_m(z)= \sum_{k=0}^m A_k z^k$, where $A_k\in M_N(\C)$ are the coefficients 
of the polynomial $P_m$. Moreover, for all $m\in \N$, the leading coefficient $A_m$ of 
$P_m$ is regular, 
see e.g. \cite{DamaPS}, \cite{Berg}. 
See Exercise \ref{exer:MVOPexistence}.

Note that we do not normalise the first $M_0=\int_\R d\mu(x)$ as the identity matrix $I \in M_N(\C)$. 
So we normalise $P_0(z) = M_0^{-1/2}$, which can be done since $M_0$ is a positive definite matrix,
hence having a square root and an inverse having a square root as well. 

Consider the space 
of $M_N(\C)$-valued functions $F$ so that 
\begin{equation*}
\int_\R F(x)\, W(x)\, F^\ast(x) \,  d\tau_\mu(x) 
\end{equation*}
exists entry wise in $M_N(\C)$. Here, as before,  $F^\ast(z) = \bigl(F(\bar z)\bigr)^\ast$. 
So this means that integrals 
\begin{equation*}
\int_\R \sum_{i,j=1}^N F_{k,i}(x)\, W_{i,j}(x)\, F^\ast_{j,l}(x) \,  d\tau_\mu(x) 
\end{equation*}
exist for $1\leq k,l\leq N$. In general, the sum and integral cannot be interchanged, see \cite[Example, p.~292]{Rose}, but note that 
this can be done in case $F$ is polynomial 
by Assumption \ref{ass:finitemomentsandstrictlypos-ae}. 
The Hilbert $\text{C}^\ast$-module $L_C^2(\mu)$ is 
obtained by modding out by the space of functions for which 
the integral is zero (as the element in the cone of positive matrices in $M_N(\C)$). 
Because of Assumption \ref{ass:finitemomentsandstrictlypos-ae} these are
the $M_N(\C)$-valued functions which are zero $\tau_\mu$-a.e.
In case we do not assume $W$ to be positive definite $\tau_\mu$-a.e., we have to mod 
out by a larger space in general, see Section \ref{ssec:MoreMweights}.

Then $L_C^2(\mu)$ is a left $M_N(\C)$-module and the 
$M_N(\C)$-valued inner product on $L_C^2(\mu)$ is defined by 
\begin{equation*}
\langle F, G\rangle = \int_\R F(x)\, W(x)\, G^\ast(x) \,  d\tau_\mu(x) \in M_N(\C)
\end{equation*}
and satisfying for $F,G, H\in L^2(\mu)$, $A,B\in M_N(\C)$, 
\begin{equation*}
\begin{split}
&\langle AF+BG, H\rangle = A\langle F, H\rangle  + B \langle G, H\rangle, \qquad 
\langle F, G\rangle = \bigl( \langle G,F\rangle\bigr)^\ast, \\
& \qquad
\langle F, F\rangle \geq 0, \quad \langle F, F\rangle =0 \ \Longleftrightarrow\ F=0
\end{split}
\end{equation*}
so that we have a Hilbert $\text{C}^\ast$-module.
The completeness is proved in \cite[Thm.~3.9]{Rose}, using the fact that 
the Hilbert-Schmidt norm on $M_N(\C)$ is equivalent to the operator norm. 
So in particular, the polynomials $P_n$ give an orthonormal collection 
for the Hilbert $\text{C}^\ast$-module $L_C^2(\mu)$. 

With $\mu$ we also associate the Hilbert space $L^2_v(\mu)$, which 
is the space of $\C^N$-valued functions $f$ so that 
\[
 \int_\R f^\ast (x)\, W(x)\, f(x) \,  d\tau_\mu(x) 
 = \int_\R \sum_{i,j=1}^N f_i^\ast (x)\, W_{i,j}(x)\, f_j(x) \, d\tau_\mu(x)< \infty
\]
where $f(z)$ is a column vector and $f^\ast(z) = \bigl( f(\bar z)\bigr)^\ast$ is a row vector. 
Then the inner product in $L^2_v(\mu)$ is given by 
\begin{equation*}
\langle f, g\rangle = \int_\R g^\ast (x)\, W(x)\, f(x) \,  d\tau_\mu(x)
\end{equation*}
Again, we assume we have modded out by $\C^N$-valued functions $f$ with 
$\langle f, f\rangle = 0$, which in this case are functions $f\colon \R\to \C^N$ 
which are zero $\tau_\mu$-a.e. by Assumption \ref{ass:finitemomentsandstrictlypos-ae}. 
The space $L^2_v(\mu)$ is studied in detail, and in greater generality, in 
\cite[XIII.5.6-11]{DunfS}.

If $F\in L_C^2(\mu)$ then $z\mapsto F^\ast(z)v$ is an element of 
$L_v^2(\mu)$. For $f_i\in L_v^2(\mu)$, the $M_N(\C)$-valued function 
$F$ having $f^\ast_i$ as its $i$-th column is in $L^2_C(\mu)$.  

\begin{thm}\label{thm:MVthreetermrecurrence}
There exist sequence of matrices $\{ A_n\}_{n\in \N}$, $\{ B_n\}_{n\in \N}$ so that
$\det(A_n)\not=0$ for all $n\in \N$ and $B_n^\ast=B_n$ for all $n\in \N$ and 
\begin{equation*}
z P_n(z) = 
\begin{cases} A_n P_{n+1}(z) + B_n P_n(z) + A_{n-1}^\ast P_{n-1}(z), & n\geq 1 \\
A_0 P_{1}(z) + B_0 P_0(z), & n=0.
\end{cases}
\end{equation*}
\end{thm}

We leave the proof of Theorem \ref{thm:MVthreetermrecurrence} as Exercise 
\ref{exer:thmMVthreetermrecurrence}.

\begin{remark}\label{rmk:MVchangeofONpols}
(i) For a sequence of unitary matrices $\{ U_n\}_{n\in \N}$, the polynomials 
$\tilde{P}_n(z) = U_n P_n(z)$ are also orthonormal polynomials with 
respect to the same matrix-valued measure $\mu$ and with matrices $A_n$, $B_n$ replaced 
by $\tilde{A}_n = U_nA_n U_{n+1}^\ast$, $\tilde{B}_n = U_nB_n U_{n}^\ast$.
Conversely, if $\{ \tilde{P_n}\}_{n\in\N}$ is a family of orthonormal polynomials, 
then there exist unitary matrices $\{ U_n\}_{n\in \N}$ such that polynomials 
$\tilde{P}_n(z) = U_n P_n(z)$. 

\noindent
(ii) By (i), there is always a choice in fixing the matrix $A_n$. One possible 
choice is to take $A_n$ upper (or lower) triangular. Another normalisation is to 
consider monic matrix-valued polynomials $\{R_n\}_{n\in\N}$ instead. The 
three-term recurrence for $n\geq 1$ becomes 
\begin{equation*}
z R_n(z) =  R_{n+1}(z) +  \bigl( \lc(P_n)^{-1}B_n \lc(P_n)\bigr) R_n(z) + 
\bigl( \lc(P_{n-1})^{-1} A_{n-1} A_{n-1}^\ast \lc(P_{n-1})\bigr) R_{n-1}(z)
\end{equation*}
since $\lc(P_n) = A_n \lc(P_{n+1})$, where $\lc(P_n)\in M_N(\C)$ denotes the 
leading coefficient of the polynomial $P_n$, which is a regular matrix. 
\end{remark}

\begin{example}\label{exa:MVGegenbauerpols} 
This example gives an explicit example of a matrix-valued measure and corresponding
three-term recurrence relation for arbitrary size, which can be considered 
as a matrix-valued analogue of the Gegenbauer or ultraspherical polynomials,\index{matrix-valued Gegenbauer polynomials} 
\index{orthogonal polynomials!matrix-valued Gegenbauer polynomials} 
see e.g. \cite{AskeRS}, \cite{Isma}, \cite{KoekLS}, \cite{KoekS}, \cite{Szeg}, \cite{Temm} for the scalar case.
It is one of few examples for arbitrary size where most, if not all, of the important 
properties are explicitly known. The case $\nu=1$ was originally obtained 
using group theory and analytic methods, see \cite{KoelvPR1}, \cite{KoelvPR2}, 
motivated by \cite{Koor-SIAM}, \cite{GrunPT}, and later analytically 
extended in $\nu$, see \cite{KoeldlRR}. A $q$-analogue for the case $\nu=1$, viewed
as a matrix-valued analogue of a subclass of continuous $q$-ultraspherical
polynomials can be found in \cite{AldeKR}. 
For $2\times 2$-matrix-valued cases, Pacharoni and Zurri\'an \cite{PachZ} have 
also derived analogues of the Gegenbauer polynomials, and there is some
overlap with the irreducible subcases specialised to the $2\times 2$-cases of this general example. 
This family of matrix-valued orthogonal polynomials is studied in 
\cite{KoeldlRR}, and we refer to this paper for details. 

In this example $N=2\ell+1$, where $\ell\in\frac12\N$, and we use the 
numbering from $0$ to $2\ell$ for the indices. 
We use the standard notation for Gegenbauer polynomials, see e.g. \cite[\S 4.5]{Isma}. 
For $\nu> 0$, $W^{(\nu)}(x)$ has the following LDU-decomposition
\begin{equation}\label{eq:MV-GegenbauerLDU}
W^{(\nu)}(x)=
L^{(\nu)}(x)T^{(\nu)}(x)L^{(\nu)}(x)^{t}, \qquad x\in(-1,1),
\end{equation} 
where $L^{(\nu)}\colon [-1,1]\to M_{2\ell+1}(\C)$ is the unipotent lower triangular matrix-valued polynomial 
\[
\bigl(L^{(\nu)}(x)\bigr)_{m,k}=\begin{cases} 0 & \text{if } m<k \\
\displaystyle{\frac{m!}{k! (2\nu+2k)_{m-k}} C^{(\nu+k)}_{m-k}(x)} & \text{if } m\geq k.
\end{cases}
\]
and $T^{(\nu)}\colon (-1,1)\to M_{2\ell+1}(\C)$ is the diagonal matrix-valued function 
\begin{equation*}
\begin{split}
\bigl(T^{(\nu)}(x)\bigr)_{k,k}\, =\, t^{(\nu)}_{k}\,  (1-x^2)^{k+\nu-1/2}, \quad 
 t^{(\nu)}_{k}\, =\, \frac{k!\, (\nu)_k}{(\nu+1/2)_k}
\frac{(2\nu + 2\ell )_{k}\, (2\ell+\nu)}{(2\ell - k+1)_k\, ( 2\nu + k - 1)_{k}}.
\end{split}
\end{equation*}
From this expression it immediately follows that $W^{(\nu)}$ is 
positive definite on $(-1,1)$, since for $\nu>0$ all the constants are
positive. The definition \eqref{eq:MV-GegenbauerLDU} is not used as the 
definition in \cite{KoeldlRR}, but it has the advantage that it proves that
$W^{(\nu)}$ is positive definite immediately. 

So we can consider the corresponding monic matrix-valued orthogonal polynomials for which
we have the orthogonality relations, see \cite[Thm.~3.1]{KoeldlRR}, 
\begin{gather*}
\int_{-1}^1 P_n^{(\nu)}(x) \, W^{(\nu)}(x) \,  \bigl(P_m^{(\nu)}\bigr)^\ast(x) \, dx \, = \, \de_{n,m} H_n^{(\nu)}, \\
\bigl( H^{(\nu)}_n\bigr)_{k,l} = \de_{k,l} \sqrt{\pi}\,  \frac{\Ga(\nu+\frac12)}{\Ga(\nu+1)}
\frac{\nu(2\ell+\nu+n)}{\nu+n} \frac{n!\, (\ell+\frac12+\nu)_n (2\ell+\nu)_n(\ell+\nu)_n}{(2\ell+\nu+1)_n(\nu+k)_n(2\ell+2\nu+n)_n(2\ell+\nu-k)_n} \\
\times \frac{k!\, (2\ell-k)!\, (n+\nu+1)_{2\ell}}{(2\ell)!\, (n+\nu+1)_k (n+\nu+1)_{2\ell-k}}
\end{gather*}
where $\Ga$ denotes the standard $\Ga$-function, $\Ga(z) = \int_0^\infty t^{z-1} e^{-t}dt$,\index{G@$\Ga$-function}
see e.g. \cite{AndrAR}, \cite{Isma}, \cite{Temm}. 
The three-term recurrence relation for the monic matrix-valued orthogonal polynomials is 
\begin{equation*}
xP^{(\nu)}_{n}(x)=P^{(\nu)}_{n+1}(x)+B_n^{(\nu)}P^{(\nu)}_{n}(x)+
C^{(\nu)}_{n}P^{(\nu)}_{n-1}(x) 
\end{equation*}
where the matrices $B^{(\nu)}_n$, $C^{(\nu)}_n$ are given by
\begin{align*}
B^{(\nu)}_n&=\sum_{j=1}^{2\ell}  \frac{j(j+\nu-1)}{2(j+n+\nu-1)(j+n+\nu)}E_{j,j-1} + 
\\
&\qquad\qquad  
 \sum_{j=0}^{2\ell-1}
 \frac{(2\ell-j)(2\ell-j+\nu-1)}{2(2\ell-j+n+\nu-1)(2\ell+n-j+\nu)}E_{j,j+1}  \\
C^{(\nu)}_n&=\sum_{j=0}^{2\ell} \frac{n(n+\nu-1)(2\ell+n+\nu)(2\ell+n+2\nu-1)}
{4(2\ell+n+\nu-j-1)(2\ell+n+\nu-j)(j+n+\nu-1)(j+n+\nu)} E_{j,j}.
\end{align*}
The proofs of the orthogonality relations and the three-term recurrence relation
involve shift operators, where the lowering operator is essentially the derivative and
the raising operator is a suitable adjoint (in the context of a Hilbert 
$\text{C}^\ast$-module) of the derivative. 
The explicit value for $C_n$ follows easily from the quadratic norm, and the 
calculation of $B_n$ requires the use of these shift operators. 

Putting $P_n(x) = \bigl( H^{(\nu)}_n\bigr)^{-1/2} P_n^{(\nu)}(x)$ as the corresponding orthonormal polynomials,
we find the three-term recurrence relation of Theorem \ref{thm:MVthreetermrecurrence} with 
$A_n= \bigl( H^{(\nu)}_n\bigr)^{-1/2} \bigl( H^{(\nu)}_{n+1}\bigr)^{1/2}$, so that 
$A^\ast_{n-1}= \bigl( H^{(\nu)}_n\bigr)^{-1/2} C^{(\nu)}_n \bigl( H^{(\nu)}_{n-1}\bigr)^{1/2}$, and 
$B_n= \bigl( H^{(\nu)}_n\bigr)^{-1/2} B^{(\nu)}_n \bigl( H^{(\nu)}_{n}\bigr)^{1/2}$. 
Finally, note that we have not written the weight measure in terms of the corresponding
tracial weight. Note that 
\begin{gather*}
\Tr\bigl( W^{(\nu)}(x)\bigr) = \sum_{p=0}^{2\ell} 
\left( \sum_{k=p}^{2\ell}\bigl( L^{(\nu)}_{k,p}(x)\bigr)^2\right) \, T^{(\nu)}_{p,p}(x), 
\end{gather*}
so that by \eqref{eq:MV-GegenbauerLDU}, the trace measure $\tau_\mu$ is absolutely continuous
with respect to the standard Gegenbauer weight $(1-x^2)^{\nu-1/2}dx$ on $[-1,1]$. 
Now a result by Rosenberg \cite[p.~294]{Rose} states that the abstractly defined, i.e. 
using the trace measure $\tau_\mu$, 
spaces $L^2_C(\mu)$ and $L^2_v(\mu)$ are indeed the same as the corresponding
spaces using the weight $W^{(\nu)}$ on $[-1,1]$. 
Finally, note that in the limit $n\to \infty$ the recurrence relation reduces 
to a diagonal recurrence, in which the matrices are multiples of the identity.
So this example fits into the approach of Aptekarev and Nikishin \cite{ApteN},
Geronimo \cite{Gero}, Dur\'an \cite{Dura-JAT1999}. 
\end{example}

Starting with the matrix-valued measure and choosing a corresponding set 
of matrix-valued orthonormal polynomials $\{P_n\}_{n\in \N}$, we can associate
the corresponding matrix-valued polynomials of the 
second kind\index{matrix-valued polynomials of the 
second kind}
\begin{equation}\label{eq:defMVpolssecondkind}
Q_n(z) = \int_\R \frac{P_n(z) - P_n(x)}{z-x}\, d\mu(x) 
= \int_\R \frac{P_n(z) - P_n(x)}{z-x}\, W(x) d\tau_\mu(x)
\end{equation}
so that $Q_0(z)=0$ (as a matrix in $M_N(\C)$) and, since $P_1(z) = A_0^{-1}(xM_0^{-1/2}-B_0)$, 
we have $Q_1(z)= A_0^{-1} M_0^{-1/2} M_0 = A_0^{-1}M_0^{1/2}$. 
Note that, in the context of Remark \ref{rmk:MVchangeofONpols}, we have $\tilde{Q}_n(z) = 
U_n Q_n(z)$. 

In the case $\ell=0$ or $N=1$ of Example \ref{exa:MVGegenbauerpols} the associated
polynomials can be expressed in terms of the Gegenbauer polynomials $C^{(\nu+1)}_{n-1}$.
This breaks down in the general case of the matrix-valued Gegenbauer 
polynomials in Example \ref{exa:MVGegenbauerpols}. 

\begin{lemma}\label{lem:MV2ndkindsolverecurrence}
With the notation of Theorem \ref{thm:MVthreetermrecurrence} and \eqref{eq:defMVpolssecondkind}
we have for $n\geq 1$
\begin{equation*}
z Q_n(z) =  A_n Q_{n+1}(z) + B_n Q_n(z) + A_{n-1}^\ast Q_{n-1}(z).
\end{equation*}
\end{lemma}

See Exercise \ref{exer:lemMV2ndkindsolverecurrence} for the proof of 
Lemma \ref{lem:MV2ndkindsolverecurrence}. 

There are many relations between the two solutions, however an easy analogue 
of Lemma \ref{lem:Wronskian} is not available, since the non-commutativity of 
$M_N(\C)$ has to be taken into account. 
For our purposes we need the matrix-valued analogue of the 
Liouville-Ostrogradsky\index{Liouville-Ostrogradsky} result 
in order to describe the Green kernel for the corresponding Jacobi operator. 

\begin{lemma}\label{lem:MVLiouvilleOstrogradsky}
Let $z\in \C$. For $k\geq 1$ we have 
\[
Q_k(z)P_{k-1}^\ast(z) - P_k(z)Q_{k-1}^\ast (z) = A_{k-1}^{-1} 
\]
and for $k\geq 0$ we have $Q_k(z)P_k^\ast(z)= P_k(z) Q_k^\ast(z)$. 
\end{lemma}

We follow \cite[\S 5]{Berg} for its proof.  

\begin{proof} We proceed by joint induction on $k$. The case $k=1$ is 
\begin{equation*}
Q_1(z)P_0^\ast(z) - P_1(z) Q^\ast_0(z) = 
A_0^{-1}M_0^{1/2} (M_0^{-1/2})^\ast - 0 = A_0^{-1}.
\end{equation*}
The case $k=0$ of the second statement is trivial, since $Q_0(z)=0$. For 
$k=1$, we see that both sides equal 
\[
A_0^{-1}(zM_0^{-1/2}-B_0)(A_0^{-1})^\ast 
\]
since $M_0$ and $B_0$ are self-adjoint. 

Now assume that both statements have been proved for $k\leq n$. 
Use Theorem \ref{thm:MVthreetermrecurrence} multiplied from the right by $Q_n^\ast(z)$ 
and Lemma \ref{lem:MV2ndkindsolverecurrence} multiplied from the right by $P_n^\ast(z)$ 
and subtract to get 
\begin{multline*}
A_n \bigl( P_{n+1}(z)Q_n^\ast(z) - Q_{n+1}(z)P_n^\ast(z)\bigr) 
+ (B_n-z)\bigl( P_{n}(z)Q_n^\ast(z) - Q_{n}(z)P_n^\ast(z)\bigr) \\
+ A_{n-1}^\ast \bigl( P_{n-1}(z)Q_n^\ast(z) - Q_{n-1}(z)P_n^\ast(z)\bigr) = 0
\end{multline*}
By the induction hypothesis the middle term vanishes, and the last term is 
$A_{n-1}^\ast (A_{n-1}^{-1})^\ast = I$ by taking adjoints. Hence,
\[
A_n \bigl( P_{n+1}(z)Q_n^\ast(z) - Q_{n+1}(z)P_n^\ast(z)\bigr) = -I
\]
which is the first statement for $k=n+1$. 

To prove the second statement for $k = n+1$, write 
\begin{gather*}
zP_n(z)Q^\ast_{n+1}(z)= 
A_nP_{n+1}(z)Q^\ast_{n+1}(z) + B_nP_{n}(z)Q^\ast_{n+1}(z) + A_{n-1}^\ast P_{n-1}(z)Q^\ast_{n+1}(z) 
\qquad \Longrightarrow \\
A_nP_{n+1}(z)Q^\ast_{n+1}(z) = 
(z-B_n)P_{n}(z)Q^\ast_{n+1}(z) - \qquad\qquad\qquad\qquad \\ \qquad\qquad\qquad\qquad
A_{n-1}^\ast P_{n-1}(z)\Bigl( Q_n^\ast(z)(z-B_n) -Q_{n-1}^\ast(z)A_{n-1} \Bigr)(A_n^\ast)^{-1} 
\end{gather*}
since $Q_{n+1}^\ast(z) = \bigr( Q_n^\ast(z)(z-B_n) -Q_{n-1}^\ast(z)A_{n-1} \bigr)(A_n^\ast)^{-1}$ 
by taking adjoints in Lemma \ref{lem:MV2ndkindsolverecurrence} using the regularity of 
$A_k$ and $B_k$ being self-adjoint. Since this argument only uses the recursion for $k\geq 1$ we
can interchange the roles of the polynomials $P_k$ and $Q_k$. Subtracting the 
two identities then gives 
\begin{multline*}
A_n\Bigl(P_{n+1}(z)Q^\ast_{n+1}(z) - Q_{n+1}(z)P^\ast_{n+1}(z)\Bigr) = 
(z-B_n)\Bigl( P_{n}(z)Q^\ast_{n+1}(z) - Q_{n}(z)P^\ast_{n+1}(z)\Bigr) - 
\\ 
A_{n-1}^\ast \Bigl( P_{n-1}(z)Q_n^\ast(z) - Q_{n-1}(z)P_n^\ast(z) \Bigr)(z-B_n) (A_n^\ast)^{-1}
\\ 
-A_{n-1}^\ast \Bigl( P_{n-1}(z)Q_{n-1}^\ast(z) - Q_{n-1}(z)P_{n-1}^\ast(z)\Bigl) A_{n-1}(A_n^\ast)^{-1} 
\end{multline*}
Applying the induction hypothesis for the second statement, the last term vanishes.
Since we assume the first statement for $k\leq n$, and we  have already proved the first statement 
for $k=n+1$, we find
\[
A_n\Bigl(P_{n+1}(z)Q^\ast_{n+1}(z) - Q_{n+1}(z)P^\ast_{n+1}(z)\Bigr) = 
(z-B_n) (A_n^{-1})^\ast  -
A_{n-1}^\ast  (A_{n-1}^{-1})^\ast (z-B_n) (A_n^\ast)^{-1}
\]
Since the right-hand side is zero and $A_n$ is invertible, the second statement follows for 
$k=n+1$. So we have established the induction step, and the lemma follows.
\end{proof}

\subsection{The corresponding Jacobi operator}

We now consider the Hilbert space $\ell^2(\N) \hat\otimes \C^N$, which we denote by 
$\ell^2(\C^N)$, \index{l@$\ell^2(\C^N)$}
as the Hilbert space tensor product of the Hilbert spaces $\ell^2(\N)$ 
equipped with the standard orthonormal basis $\{e_n\}_{n\in \N}$
and $\C^N$ with the standard orthonormal basis $\{e_n\}_{n=1}^N$,
see Example \ref{exa:Hilbertspaces}(iv).
In explicit examples, such as Example \ref{exa:MVGegenbauerpols}, 
it is convenient to have a slightly different labeling.
Then we can denote 
\begin{equation*}
V = \sum_{n=0}^\infty e_n \otimes v_n \in \ell^2(\C^N) = \ell^2(\N) \hat\otimes \C^N 
\end{equation*}
where $v_n \in \C^N$. The inner product in the Hilbert space 
$\ell^2(\C^N) = \ell^2(\N)\hat\otimes \C^N$ is then 
\begin{equation*}
\langle V, W \rangle = \sum_{n=0}^\infty \langle v_n, w_n\rangle
\end{equation*}
where $W = \sum_{n=0}^\infty e_n \otimes w_n \in \ell^2(\C^N)$.
We denote the inner products in $\ell^2(\C^N)$ and $\C^N$
by the same symbol $\langle \cdot, \cdot\rangle$, where the context dictates
which inner product to take. 
This space can also be thought of sequences $(v_0, v_1, \cdots)$ with 
$v_n\in \C^N$ which are square summable $\sum_{n=0}^\infty \|v_n\|^2 <\infty$.
The case $N=1$ gives back the Hilbert space $\ell^2(\N)$ of square summable
sequences. 

Given the sequences $\{ A_n\}_{n\in \N}$ and $\{ B_n\}_{n\in \N}$ 
in $M_N(\C)$ 
with all matrices $A_n$ regular and all matrices $B_n$ self-adjoint, 
we define the Jacobi operator\index{Jacobi operator} $J$ with domain $\cD$ by
\begin{equation}\label{eq:defMVJacobi}
\begin{split}
JV & = e_0 \otimes (A_0 v_{1} + B_0 v_0) + 
\sum_{n=1}^\infty e_n \otimes \Bigl( A_n v_{n+1} + B_n v_n + A_{n-1}^\ast v_{n-1}\Bigr), \\ 
\cD & =\{ V = \sum_{\underset{\scriptstyle{\text{finite}}}{{n=0}}}^\infty e_n \otimes v_n\}\subset \ell^2(\C^N), 
\end{split}
\end{equation}
so that $(J, \cD)$ is a symmetric operator 
\begin{equation*}
\langle JV, W \rangle = \langle V, JW \rangle, \qquad \forall\, V,W \in \cD.
\end{equation*}
Note that 
\begin{equation*}
J(e_k \otimes v) = \begin{cases} 
e_{k+1}\otimes A^\ast_kv + e_k \otimes B_k v + e_{k-1} \otimes A_{k-1}v, & k\geq 1, \\
e_{1}\otimes A^\ast_1v + e_0 \otimes B_0 v, & k=0.
\end{cases}
\end{equation*}
so that 
\begin{equation*}
e_{k+1}\otimes v = J(e_k \otimes (A^\ast_k)^{-1}v) -  e_k \otimes B_k (A^\ast_k)^{-1}v - e_{k-1} \otimes A_{k-1}(A^\ast_k)^{-1}v
\end{equation*}
for $k\geq 1$ and 
\begin{equation*}
e_{1}\otimes v = J(e_0 \otimes (A^\ast_1)^{-1}v) - e_0 \otimes B_0 (A^\ast_1)^{-1}v
\end{equation*}
Using induction with respect to $k\in \N$ we immediately 
obtain Lemma \ref{lem:MVcyclicvectors}.

\begin{lemma}\label{lem:MVcyclicvectors}
The closure of the linear span of $J^pv$ where $v\in \C^N$ and $p\in \N$ is equal
to $\ell^2(\C^N)$. 
\end{lemma}

It is clear from \eqref{eq:defMVJacobi} and Theorem \ref{thm:MVthreetermrecurrence}
that we can consider $\sum_{n=0}^\infty e_n \otimes P_n(z) v$ formally as
eigenvectors for $J$, and we first take a look at the truncated version. 

\begin{lemma}\label{lem:MVeigenvaluesJ}
Let $V= \sum_{n=0}^M e_n \otimes P_n(z) v\in \cD$, $M\geq 1$, for some $v\in\C^N$, then
\[
JV = z V - e_{M} \otimes A_M P_{M+1}(z)v + e_{M+1}\otimes A^\ast_MP_M(z)v.
\]
Let $\cP_M\colon \ell^2(\C^N)\to \ell^2(\C^N)$ be the projection onto the 
span of $e_n\otimes v$, $0\leq n\leq M$ and $v\in \C^N$, we see that $V$ is an eigenvector of the truncated $\cP_M J\cP_M$ matrix for the eigenvalue
$z$ if and only if $\det(P_{N+1}(z))=0$ and $v\in \Ker(P_{N+1}(z))$. 
In particular, the zeroes of $\det(P_{N+1}(z))$ are real. 
\end{lemma}

\begin{proof} The expression for $JV$ follows from \eqref{eq:defMVJacobi}.
Taking the truncated version kills the last term. 
Then the eigenvectors of the truncated Jacobi operator 
can only occur if $A_M P_{M+1}(z)v=0\in \C^N$, since 
$A_M$ invertible. This gives the statement, and since the truncated Jacobi operator is
self-adjoint, we find that the zeroes of $\det(P_{N+1}(z))$ are real. 
\end{proof}

In case $\{ \|A_n\|\}_{n\in\N}$ and $\{ \|B_n\|\}_{n\in\N}$ are bounded sequences, 
then $J$ is a bounded operator. In that case $J$ extends to 
a bounded self-adjoint operator on $\ell^2(\C^N)$. 
If this is not the case, then we can determine its adjoint 
by the same action on its maximal domain, which is the content of 
Proposition \ref{prop:MVadjointJ}. 

\begin{prop}\label{prop:MVadjointJ}
The adjoint of $(J,\cD)$  is given by $(J^\ast, \cD^\ast)$ with 
\begin{gather*}
\cD^\ast = \{ W =\sum_{n=0}^\infty e_n \otimes w_n \in \ell^2(\C^N) \mid \qquad\qquad\qquad\qquad\qquad \\
\qquad\qquad\qquad \|A_0 w_1 + B_0 w_0\|^2 
+ \sum_{n=1}^\infty \| A_{n-1}^\ast w_{n-1} +  B_nw_n + A_{n} w_{n+1}\|^2 <\infty \}, \\
J^\ast W = e_0\otimes \bigl( A_0 w_1 + B_0 w_0\bigr) 
+ \sum_{n=1}^\infty e_n \otimes \bigl( A_{n-1}^\ast w_{n-1} +  B_nw_n + A_{n} w_{n+1}\bigr).
\end{gather*}
\end{prop}

\begin{proof} Recall the definition of the adjoint operator  for 
an unbounded operator, see Section \ref{ssec:appunboundedsaoperators}. 
Take $W\in \ell^2(\C^N)$ and consider for $V=\sum_{n=0}^\infty e_n \otimes v_n\in \cD$, so the sum for $V$ is finite,  
\begin{equation*}
\begin{split}
& \langle JV, W\rangle = \langle A_0v_1 + B_0v_0, w_0\rangle 
+ \sum_{n=1}^\infty \langle A_n v_{n+1} + B_n v_n + A_{n-1}^\ast v_{n-1}, w_n\rangle \\
&= \sum_{n=1}^\infty \langle v_{n+1}, A_n^\ast w_n \rangle + \sum_{n=1}^\infty \langle v_n, B_nw_n\rangle 
+ \sum_{n=1}^\infty \langle v_{n-1}, A_{n-1} w_n\rangle + \langle v_1, A_0^\ast w_0\rangle + \langle v_0, B_0 w_0\rangle \\
& = \sum_{n=2}^\infty \langle v_{n}, A_{n-1}^\ast w_{n-1} \rangle +  \sum_{n=1}^\infty \langle v_n, B_nw_n\rangle 
+ \sum_{n=0}^\infty \langle v_{n}, A_{n} w_{n+1} \rangle + \langle v_1, A_0^\ast w_0\rangle + \langle v_0, B_0 w_0\rangle \\
& = \sum_{n=1}^\infty \langle v_{n}, A_{n-1}^\ast w_{n-1} +  B_nw_n + A_{n} w_{n+1} \rangle 
+ \langle v_0, A_0 w_1 + B_0 w_0\rangle
\end{split}
\end{equation*}
since $B_n$ is self-adjoint for all $n\in \N$ and all sums are finite since $V\in \cD$. 
First assume that $W\in \cD^\ast$, then by the above calculation we have 
\begin{equation*}
| \langle JV, W\rangle | \leq \| V \| \, \| J^\ast W\| \leq C \|V\|, \qquad \forall \, V\in \cD 
\end{equation*}
so that $\cD^\ast$ is contained in the domain of 
the adjoint of $(J,\cD)$.

Conversely, for $W$ in the domain of 
the adjoint of $(J,\cD)$, we have by definition that for all $V\in \cD$ 
\begin{equation}\label{eq:MVadjoint1}
| \langle JV, W\rangle | \leq C\, \| V \|
\end{equation}
for some constant $C$. Take $V= e_0\otimes (A_0w_1+B_0w_0) + \sum_{k=1}^M e_n \otimes \bigl(A_{n-1}^\ast w_{n-1} +  B_nw_n + A_{n} w_{n+1}\bigr)$
in \eqref{eq:MVadjoint1} and using the above calculation we find
\begin{gather*}
\left(  \|A_0w_1+B_0w_0\|^2 + \sum_{n=1}^M \| A_{n-1}^\ast w_{n-1} +  B_nw_n + A_{n} w_{n+1}\|^2\right)^{1/2} \leq C
\end{gather*}
Since $C$ is independent of $M$, by taking $M\to\infty$ we see $W\in \cD^\ast$. 
The expression for the action of the adjoint of $(J,\cD)$ follows from the above calculation. 
Hence, the lemma follows. 
\end{proof}

\subsection{The resolvent operator}\label{ssec:MVresolventoperator}

Define the Stieltjes transform of the matrix-valued measure by
\begin{equation*}
S(z) = \int_\R \frac{1}{x-z} \, d\mu(x) = \int_\R \frac{1}{x-z} W(x)\, d\tau_\mu(x),
\qquad z\in \C\setminus \R, 
\end{equation*}
and note that $S^\ast(z)=\bigl( S(\bar z)\bigr)^\ast = S(z)$, since the 
measure $\tau_\mu$ is positive and $W(x)$ is positive definite $\tau_\mu$-a.e.
So $S\colon \C\setminus \R \to M_N(\C)$. Note that 
$S$ is holomorphic in the upper and lower half plane, meaning that each 
of its matrix entries is holomorphic.
The Stieltjes transform encodes the moments as in the classical case, see \cite{ApteN}. 

Define for $z\in \C\setminus \R$ and $k\in \N$  
\begin{equation*}
F_k(z) = Q_k(z) +  P_k(z) S(z),  
\end{equation*}
then,  by Theorem \ref{thm:MVthreetermrecurrence} and Lemma \ref{lem:MV2ndkindsolverecurrence}, 
\begin{equation}\label{eq:MVfreesolution}
z F_n(z) =  A_n F_{n+1}(z) + B_n F_n(z) + A_{n-1}^\ast F_{n-1}(z), \qquad n\geq 1. 
\end{equation}
Moreover, by Lemma \ref{lem:MVLiouvilleOstrogradsky}, 
\begin{gather*}
A_{k-1}\Bigl( F_k(z)P_{k-1}^\ast(z) - P_k(z)F_{k-1}^\ast (z)\Bigr) =  \\
A_{k-1}\Bigl( Q_k(z)P_{k-1}^\ast(z) + P_k(z) S(z)P_{k-1}^\ast(z) - P_k(z)Q_{k-1}^\ast (z) - P_k(z) S^\ast(z)P_{k-1}^\ast(z)\Bigr) =  \\
A_{k-1}\Bigl( Q_k(z)P_{k-1}^\ast(z)  - P_k(z)Q_{k-1}^\ast (z)\Bigr) = I
\end{gather*}
since $S(z) = S^\ast(z)$. 

\begin{lemma}\label{lem:MVfreesolution}
For $v\in \C^N$, $\sum_{n=0}^\infty e_n \otimes F_n(z) v \in \ell^2(\C^N)$.  
\end{lemma}

\begin{proof} Start by rewriting 
\begin{equation*}
\begin{split}
F_k(z) & = Q_k(z) +  P_k(z) S(z) \\
& = \int_\R \frac{P_k(z) - P_k(x)}{z-x}\, d\mu(x)  + \int_\R \frac{1}{z-x}P_k(z)  \, d\mu(x) \\
& = \int_\R \frac{- P_k(x)}{z-x}\, d\mu(x)  =  \int_\R P_k(x)\, W(x) \, F^\ast(x)\, d\tau_\mu(x)  
\end{split}
\end{equation*}
where $F(x) = (x-\bar z)^{-1}I$. Note that $F\in L^2_C(\mu)$ for $z\in \C\setminus \R$, 
so that by the Bessel inequality for Hilbert $\text{C}^\ast$-modules, see Appendix \ref{ssec:HilbertCastmodules}, 
\begin{gather*}
\sum_{k=0}^\infty \bigl(F_k(z)\bigr)^\ast F_k(z) \leq  \langle F, F\rangle = 
\int_\R F(x)\, W(x)\, F^\ast(x)\, d\tau_\mu(x) \quad \Longrightarrow\\
\sum_{n=0}^\infty  \| F_n(z) v\|^2 \leq 
\int_\R v^\ast F(x)\, W(x)\, F^\ast(x)v\, d\tau_\mu(x) 
\leq \frac{v^\ast M_0 v}{|\Im(z)|^2}< \infty. \qedhere
\end{gather*}
\end{proof}

Since the series in Lemma \ref{lem:MVfreesolution} converges, we see that 
\begin{equation}\label{eq:MV-Markov}
S(z) = - \lim_{k\to \infty} P_k(z)^{-1} Q_k(z) \quad \text{for} \ z\in \C\setminus \R.
\end{equation}
Note that $P_k(z)$ is invertible by Lemma \ref{lem:MVeigenvaluesJ} for $z\in \C\setminus \R$.
The convergence \eqref{eq:MV-Markov} is in operator norm, and hence leads to entrywise 
convergence. \index{Markov's theorem}
This is a matrix-valued analogue of Markov's theorem \eqref{eq:Markovscalar}, 
see also \cite[\S 1.4]{ApteN}.

\begin{defn}\label{def:MVspmz}
Define for $z\in \C$ the vector space 
\begin{equation*}
S^+_z = \{ V=\sum_{n=0}^\infty e_n \otimes v_n \in \ell^2(\C^N) \mid \exists M\in \N\,  \forall n\geq M \quad  zv_n = A_nv_{n+1} + B_nv_n + A_{n-1}^\ast v_{n-1}\} 
\end{equation*}. 
\end{defn}

Since for linearly independent vectors in $\C^N$, the corresponding elements 
in Lemma \ref{lem:MVfreesolution} are linearly independent, 
we see that $\dim S^+_z \geq N$ for $z\in \C\setminus \R$. Note 
that the condition $V=\sum_{n=0}^\infty e_n\otimes v_n \in S^+_z$ 
only involves the behaviour of $v_n$ for $n\gg 0$, and we can recursively 
adapt $v_{M-1}, v_{M-2}, \cdots, v_0$ by requiring the recursion
relation. Note that in general $zv_0 \not= A_0v_1 +B_0v_1$, as can be
seen for the element $\sum_{n=0}^\infty e_n\otimes F_n(z)v$ of Lemma
\ref{lem:MV2ndkindsolverecurrence} from the explicit values 
for $P_0(z)$, $P_1(z)$, $Q_0(z)$, $Q_1(z)$ in Section 
\ref{ssec:MVmeasurespols}. 

So $S^+_z$ is not the deficiency space for $(J^\ast, \cD^\ast)$, 
since we do not require that it satisfies the 
recurrence for all $n\in \N$.
Moreover, any solution for the recurrence relation for all $n\in \N$ is of 
the form $\sum_{n=0}^\infty e_n\otimes P_n(z) v$, so we find for $z\in \C\setminus \R$ 
\begin{equation}\label{eq:MVdeficiencyspace}
N_z = \{ V\in \cD^\ast \mid J^\ast V= zV\} = 
\{\sum_{n=0}^\infty e_n \otimes P_n(z)v \mid v\in \C^N\} \cap S^+_z
\end{equation}
In particular, we see that deficiency indices $0\leq n_\pm\leq N$. 
In case $A_n, B_n \in M_N(\R)$ for all $n\in \N$ we see that $n_+=n_-$, since
conjugation induces an isomorphism of $N_z$ onto $N_{\bar z}$. 
Note that also $n_+=n_-$ if we can find a sequence $\{U_n\}_{n\in \N}$ of unitary 
operators such that $U_nA_n U_{n+1}^\ast, U_nB_n U_{n}^\ast\in M_N(\R)$ for all 
$n\in \N$, see Remark \ref{rmk:MVchangeofONpols}. Note that it is always 
possible to find unitary $U_n$ so that $U_nB_n U_{n}^\ast\in M_N(\R)$, since 
$B_n$ is self-adjoint. 
For $N=1$ this can always be done, so that in this case the deficiency 
indices are always the same; $(n_+,n_-)=(0,0)$ or $(1,1)$. 

Assumption \ref{ass:MVpolsolnotfree} 
says $N_z=\{0\}$ for all $z\in \C\setminus \R$. Hence, $(J,\cD)$ is essentially self-adjoint 
and thus $(J^\ast, \cD^\ast)$ is self-adjoint. 

\begin{ass}\label{ass:MVpolsolnotfree}
For all $v\in \C^N$, the element $\sum_{n=0}^\infty e_n \otimes P_n(z)v \notin S^+_z$ for $z\in\C\setminus \R$. 
\end{ass}

The assumption means that $\sum_{n=0}^\infty \|P_n(z) v\|^2$ diverges for all $v\in \C^N$. 

\begin{thm}\label{thm:MVresolvent} Define the operator $G_z \colon \cD \to \ell^2(\C^N)$ 
for $z\in \C \setminus \R$ 
by 
\begin{gather*}
G_zV = \sum_{n=0}^\infty e_n \otimes (G_zV)_n, \qquad (G_zV)_n = \sum_{k=0}^\infty (G_z)_{n,k} v_k, \\
M_N(\C) \ni (G_z)_{n,k}  = 
\begin{cases}
P_n(z) F_k^\ast (z),  & n\leq k \\
F_n(z) P_k^\ast (z), & n>k.
\end{cases}
\end{gather*}
Then $G_z$ is the resolvent operator for $J^\ast$, i.e. $G_z = (J^\ast-z)^{-1}$, so 
$G_z \colon \ell^2(\C^N) \to \ell^2(\C^N)$ extends to a bounded operator. 
\end{thm}

\begin{proof}
First, we prove $G_zV\in \cD^\ast \subset \ell^2(\C^\N)$ for 
$V=\sum_{n=0}^\infty e_n \otimes v_n\in\cD$. 
In order to do so we need to see that $(G_zV)_n$ is well-defined;
the sum over $k$ is actually finite and, by the Cauchy-Schwarz inequality, 
\begin{gather*}
\sum_{\underset{\scriptstyle{\text{finite}}}{k=0}}^\infty \| (G_z)_{n,k} v_k\| \leq  
\sum_{\underset{\scriptstyle{\text{finite}}}{k=0}}^\infty \| (G_z)_{n,k}\|\, \| v_k\|
\\ \leq  \Bigl( \sum_{\underset{\scriptstyle{\text{finite}}}{k=0}}^\infty \| (G_z)_{n,k}\|^2\Bigr)^{1/2} 
\Bigl( \sum_{\underset{\scriptstyle{\text{finite}}}{k=0}}^\infty \| v_k \|^2\Bigr)^{1/2} = 
\|V\| \Bigl( \sum_{\underset{\scriptstyle{\text{finite}}}{k=0}}^\infty \| (G_z)_{n,k}\|^2\Bigr)^{1/2}.
\end{gather*}
So in order to show that $G_zV\in \ell^2(\C^\N)$ we estimate
\begin{gather*}
\sum_{n=0}^\infty \| \sum_{\underset{\scriptstyle{\text{finite}}}{k=0}}^\infty (G_z)_{n,k} v_k\|^2 
\leq \sum_{n=0}^\infty \Bigl( \sum_{\underset{\scriptstyle{\text{finite}}}{k=0}}^\infty \| (G_z)_{n,k} v_k\|\Bigr)^2 
\leq \|V\|^2  \sum_{n=0}^\infty \sum_{\underset{\scriptstyle{\text{finite}}}{k=0}}^\infty \| (G_z)_{n,k}\|^2 
 \end{gather*}
Next note that the double sum equals, using $K$ for the maximum term occurring in the finite sum, 
\begin{gather*}
\sum_{\underset{\scriptstyle{\text{finite}}}{k=0}}^\infty \sum_{n=0}^\infty  \| (G_z)_{n,k}\|^2 \leq 
\sum_{\underset{\scriptstyle{\text{finite}}}{k=0}}^\infty \sum_{n=0}^K  \| (G_z)_{n,k}\|^2 +
\sum_{\underset{\scriptstyle{\text{finite}}}{k=0}}^\infty \| P^\ast_k(z)\|^2\sum_{n=K+1}^\infty  \|F_n(z)\|^2 
\end{gather*}
which converges by Lemma \ref{lem:MV2ndkindsolverecurrence}. Hence, $G_zV\in \ell^2(\C^\N)$.

Next we consider 
\begin{multline*}
(J^\ast - z) G_zV = e_0 \otimes (A_0 (G_zV)_{1}+ (B_0-z) (G_zV)_{0})+ \\
\sum_{n=1}^\infty e_n \otimes \Bigl( A_n (G_zV)_{n+1}+ (B_n-z) (G_zV)_{n} + A_{n-1}^\ast (G_zV)_{n-1} \Bigr) 
\end{multline*}
and we want to show that 
\begin{equation}\label{eq:MVGisresolvent}
A_0 (G_zV)_{1}+ (B_0-z) (G_zV)_{0} = v_0, \quad 
A_n (G_zV)_{n+1}+ (B_n-z) (G_zV)_{n} + A_{n-1}^\ast (G_zV)_{n-1} = v_n
\end{equation}
for $n\geq 1$. 
Note that \eqref{eq:MVGisresolvent} in particular implies that $G_zV\in \cD^\ast$. 

In order to establish \eqref{eq:MVGisresolvent} we use the definition of the operator $G$ to find 
for $n\geq 1$ 
\begin{equation*}
\begin{split}
& A_n (G_zV)_{n+1}+ (B_n-z) (G_zV)_{n} + A_{n-1}^\ast (G_zV)_{n-1}  \\
& = \sum_{k=0}^\infty \Bigl( A_n (G_z)_{n+1,k} v_k +  (B_n-z) (G_z)_{n,k} v_k + A_{n-1}^\ast (G_z)_{n-1,k} v_k \Bigr) \\
& = \sum_{k=0}^{n-1} \Bigl( A_n F_{n+1}(z)  +  (B_n-z) F_{n}(z) + A_{n-1}^\ast F_{n-1}(z) \Bigr)P^\ast_k(z) v_k \\
& \quad + A_n (G_z)_{n+1,n} v_n +  (B_n-z) (G_z)_{n,n} v_n + A_{n-1}^\ast (G_z)_{n-1,n} v_n \\
& \quad  \sum_{k=n+1}^\infty \Bigl( A_n P_{n+1}(z)  +  (B_n-z) P_n(z) + A_{n-1}^\ast P_{n-1}(z) \Bigr)F^\ast_{k}(z) v_k
\end{split}
\end{equation*}
where we note that all sums are finite, since we take $V\in \cD$.
But also for $V\in\ell^2(\C^\N)$ the series converges, because of Lemma \ref{lem:MV2ndkindsolverecurrence}.

Because of \eqref{eq:MVfreesolution} and Theorem \ref{thm:MVthreetermrecurrence}, the first and 
the last term vanish. For the middle term we use the definition for $G$ to find
\begin{equation*}
\begin{split}
& A_n (G_z)_{n+1,n} v_n +  (B_n-z) (G_z)_{n,n} v_n + A_{n-1}^\ast (G_z)_{n-1,n} v_n \\
= & \Bigl( A_n F_{n+1}(z) P^\ast_n(z) +  (B_n-z) P_{n}(z) F^\ast_n(z) + A_{n-1}^\ast P_{n-1}(z) F^\ast_n(z)\Bigr) v_n \\
= & \Bigl( A_n F_{n+1}(z) P^\ast_n(z) -  A_n P_{n+1}(z) F^\ast_n(z) \Bigr) v_n \\
= &  A_n \Bigl( F_{n+1}(z) P^\ast_n(z) -  P_{n+1}(z) F^\ast_n(z) \Bigr) v_n = v_n
\end{split}
\end{equation*}
where we use Theorem \ref{thm:MVthreetermrecurrence} once more and Lemma \ref{lem:MVLiouvilleOstrogradsky}. 
This proves \eqref{eq:MVGisresolvent} for $n\geq 1$. We leave the case $n=0$ for Exercise \ref{exer:MVGresolvent}.

So we find that $G_z\colon \cD \to  \cD^\ast$ and $(J^\ast-z)G_z$ is the identity on $\cD$. Since
$z\in \C\setminus \R$, $z\in \rho(J^\ast)$ and $(J^\ast-z)^{-1}\in B(\ell^2(\C^N))$ which coincides
with $G_z$ on a dense subspace. So $G_z= (J^\ast-z)^{-1}$. 
\end{proof}

\subsection{The spectral measure}

We stick with the Assumptions \ref{ass:finitemomentsandstrictlypos-ae}, \ref{ass:MVpolsolnotfree}.

Having Theorem \ref{thm:MVresolvent} we calculate the matrix entries of the 
resolvent operator $G_z$ for 
$V=\sum_{n=0}^\infty e_n\otimes v_n, W=\sum_{n=0}^\infty e_n\otimes w_n\in\cD$ and $z\in \C\setminus \R$;
\begin{gather*}
\langle G_zV, W\rangle = \sum_{n=0}^\infty \langle (G_zV)_n, w_n\rangle = \sum_{k,n=0}^\infty \langle (G_z)_{n,k}v_k, w_n\rangle   \\
= \sum_{\underset{\scriptstyle{n\leq k}}{k,n=0}}^\infty \langle P_n(z)F^\ast_k(z)v_k, w_n\rangle
+ \sum_{\underset{\scriptstyle{n> k}}{k,n=0}}^\infty \langle F_n(z)P^\ast_k(z)v_k, w_n\rangle \\
= \sum_{\underset{\scriptstyle{n\leq k}}{k,n=0}}^\infty \langle \bigl( Q^\ast_k(z) + S(z)P^\ast_k(z)\bigr)v_k, \bigl(P_n(z)\bigr)^\ast w_n\rangle
+ \sum_{\underset{\scriptstyle{n> k}}{k,n=0}}^\infty \langle P^\ast_k(z)v_k, \bigl(Q_n(z) + P_n(z)S(z)\bigr)^\ast w_n\rangle  \\
= \sum_{\underset{\scriptstyle{n\leq k}}{k,n=0}}^\infty \langle P_n(z)Q^\ast_k(z)v_k,  w_n\rangle 
+ \sum_{\underset{\scriptstyle{n> k}}{k,n=0}}^\infty \langle Q_n(z)P^\ast_k(z)v_k, w_n\rangle  
+ \sum_{k,n=0}^\infty \langle P_n(z)S(z)P^\ast_k(z)v_k,  w_n\rangle
\end{gather*}
where all sums are finite since $V,W\in \cD$. The first two terms are polynomial, hence analytic, in $z$, and 
do not contribute to the spectral measure 
\begin{equation}
\begin{split}
E_{V,W}((a,b)) &= \lim_{\de\downarrow 0} \lim_{\ep \downarrow 0} \frac{1}{2\pi i} \int_{a+\de}^{b-\de} 
\langle G_{x+i\ep}V, W\rangle - \langle G_{x-i\ep}V, W\rangle \, dx \\
\end{split}
\end{equation}

\begin{lemma}\label{lem:StieltjesPerroninversion}
Let $\tau_\mu$ be a positive Borel measure on $\R$, $W_{i,j}\in L^1(\tau_\mu)$ so that 
$x\mapsto x^kW_{i,j}(x)\in L^1(\tau_\mu)$ for all $k\in \N$. Define for $z\in \C\setminus \R$ 
\[
g(z) = \int_\R \frac{p(s) W_{i,j}(s)}{s-z} \, d\tau_\mu(s)
\]
where $p$ is a polynomial, 
then for $-\infty<a<b<\infty$  
\[
\lim_{\de\downarrow 0} \lim_{\ep \downarrow 0} \frac{1}{2\pi i} \int_{a+\de}^{b-\de} 
g(x+i\ep)-g(x-i\ep) \, dx = \int_{(a,b)} p(x) W_{i,j}(x) \, d\tau_\mu(x)
\]
\end{lemma}

The proof of Lemma \ref{lem:StieltjesPerroninversion}  is in Exercise \ref{exer:lemStieltjesPerroninversion}.

From Lemma \ref{lem:StieltjesPerroninversion} we find 
\begin{equation}
\begin{split}
E_{V,W}((a,b))  &= \sum_{k,n=0}^\infty \int_{(a,b)} 
w_n^\ast P_n(x)\, W(x)\, P^\ast_k(x)v_k \, d\tau_\mu(x) \\
&=
\sum_{k,n=0}^\infty  
w_n^\ast \left( \int_{(a,b)} P_n(x)\, W(x)\, P^\ast_k(x) \, d\tau_\mu(x)\right) v_k
\end{split}
\end{equation}
By extending the integral to $\R$ we find
\begin{equation}\label{eq:spectraldecomposition}
\begin{split}
\langle V, W\rangle   &= 
\sum_{k,n=0}^\infty  
w_n^\ast \left( \int_\R P_n(x)\, W(x)\, P^\ast_k(x) \, d\tau_\mu(x)\right) v_k
\end{split}
\end{equation}
so that in particular we find the orthogonality relations for the polynomials 
\begin{equation}
 \int_\R P_n(x)\, W(x)\, P^\ast_k(x) \, d\tau_\mu(x) = \de_{n,m} I.
\end{equation}

We can rephrase \eqref{eq:spectraldecomposition} as the following theorem. 

\begin{thm}\label{thm:MVspectraldecomposition}
Let $(J, \cD)$ be essentially self-adjoint, then the unitary map 
\begin{equation*}
\cU \colon \ell^2(\C^N) \to L^2_v(\mu), 
\quad V=\sum_{n=0}^\infty e_n\otimes v_n 
\mapsto \sum_{n=0}^\infty P^\ast_n(\cdot) v_n, 
\end{equation*}
intertwines its closure $(J^\ast, \cD^\ast)$  
with multiplication, i.e. $\cU\, J^\ast = M_z\, \cU$, where
$M_z \colon \cD(M_z)\subset L^2_v(\mu) \to L^2_v(\mu)$, $f\mapsto \bigl(z\mapsto zf(z)\bigr)$,
where $\cD(M_z)$ is its maximal domain. 
\end{thm}

\begin{remark}\label{rmk:MVmultiplicity} 
(i) Note that 
Theorem \ref{thm:MVspectraldecomposition} shows that 
the closure $(J, \cD)$ has spectrum equal to the support of $\tau_\mu$, 
and that each point in the spectrum has multiplicity $N$. 
According to general theory, see e.g. \cite[\S~VII.1]{Wern}, 
we can split the (separable) Hilbert space into $N$ invariant
subspaces $\cH_i$, $1\leq i\leq  N$, which are $J^\ast$-invariant and 
which can  each can be diagonalised with multiplicity $1$.
In this case we can take for $f\in L^2_v(\mu)$ the 
function $(P_if)$, where $P_i$ is the projection on the 
basis vector $e_i\in \C^N$. Note that because
$P_i^\ast W(x) P_i \leq W(x)$, we see that 
$P_if\in L^2_v(\mu)$ and note that 
 $(P_if)\in L^2(w_{i,i}d\tau_\mu)$. 
And the inverse image of the elements $P_if$ for $f\in L^2_v(\mu)$ 
under $\cU$ gives the invariant subspaces $\cH_i$. 
Note that in practice this might be hard to do, and for this reason 
it is usually easier to have an easier description, but with higher 
multiplicity. 

(ii) We have not discussed reducibility of the weight matrix.
If the weight can be block-diagonally decomposed, the same is valid
for the corresponding $J$-matrix (up to suitable normalisation, e.g. 
in the monic version). 
For the development as sketched here, this is not required.
We give some information on reducibility issues in 
Section \ref{ssec:MVreducibility}. 
\end{remark}

We leave the analogue of Favard's theorem \ref{cor:Favardsthm}
in this case as Exercise \ref{exer:MVFavard}.

\subsection{Exercises}

\begin{enumerate}[1.]
\item \label{exer:MVOPexistence}
Show that under Assumption \ref{ass:finitemomentsandstrictlypos-ae} there 
exist orthonormal matrix-valued polynomials satisfying \eqref{eq:MVOPorthonormal}.
Show that the polynomials are determined up to left multiplication by a unitary
matrix, i.e. if $\tilde{P}_n$ forms another set of polynomials satisfying 
\eqref{eq:MVOPorthonormal} then there exist unitary matrices $U_n$, $n\in \N$, 
with $\tilde{P}_n(z) = U_n P_n(z)$.
\item \label{exer:MVGegenbauercommutant}
In the context of Example \ref{exa:MVGegenbauerpols} define the map $J\colon \C^{2\ell+1} \to \C^{2\ell+1}$
by $J\colon e_n \mapsto e_{2\ell-n}$  (recall labeling of the basis $e_n$ with $n\in \{0,\cdots, 2\ell\}$).
Check that $J$ is a self-adjoint involution. 
Show that $J$ commutes with all the matrices $B^{(\nu)}_n$, $C^{(\nu)}_n$   in the recurrence
relation for the corresponding monic matrix-valued orthogonal polynomials, and with 
all squared norm matrices $H^{(\nu)}_n$. 
\item \label{exer:thmMVthreetermrecurrence} 
Prove Theorem \ref{thm:MVthreetermrecurrence}, and show that 
$A_n=\int_\R xP_n(x) W(x) P^\ast_{n+1}(x)\, d\tau_\mu(x)$ is invertible and 
$B_n=\int_\R xP_n(x) W(x) P^\ast_{n}(x)\, d\tau_\mu(x)$ is self-adjoint. 
\item\label{exer:lemMV2ndkindsolverecurrence} 
Prove Lemma \ref{lem:MV2ndkindsolverecurrence} by generalising 
Exercise \ref{exer:lemassociatedpols}. 
\item\label{exer:MVGresolvent} 
Prove the case $n=0$ of \eqref{eq:MVGisresolvent} in the proof of
Theorem \ref{thm:MVresolvent}. 
\item\label{exer:lemStieltjesPerroninversion} In this exercise we
prove Lemma \ref{lem:StieltjesPerroninversion}. 
\begin{enumerate}[(a)]
\item Show that for $\ep>0$ 
\[
g(x+i\ep)- g(x-i\ep) = \int_\R \frac{2i\ep}{(s-x)^2 + \ep^2} p(s) W_{i,j}(s) \, d\tau_\mu(s)
\]
\item Show that for $-\infty < a < b < \infty$ 
\[
\frac{1}{2\pi i}\int_a^b g(x+i\ep)- g(x-i\ep)\, dx  = \int_\R \frac{1}{\pi}\Bigl( \arctan\bigl(\frac{b-s}{\ep}\bigr) - \arctan\bigl(\frac{a-s}{\ep}\bigr)\Bigr)  
p(s) W_{i,j}(s) \, d\tau_\mu(s)
\]
\item Finish the proof of Lemma \ref{lem:StieltjesPerroninversion}. 
\end{enumerate}
\item\label{exer:ChristoffelDarbouxMVOP} 
Prove the Christoffel-Darboux formula\index{Christoffel-Darboux formula} for 
the matrix-valued orthonormal polynomials;
\begin{equation*}
(x-y) \sum_{k=0}^{n-1} P_k^\ast(x) P_k(y) = 
P^\ast_n(x)A^\ast_{n-1}P_{n-1}(y) - P^\ast_{n-1}(x)A_{n-1}P_n(y)
\end{equation*}
and derive an expression for 
$\sum_{k=0}^{n-1} P_k^\ast(x) P_k(x)$
as in Exercise \ref{sec:threetermN}.\ref{exer:ChristoffelDarbouxscal}.
\item \label{exer:MVFavard}
Assume that we have matrix-valued polynomials generated 
the recurrence as in Theorem \ref{thm:MVthreetermrecurrence}.
Moreover, assume that $\{ \|A_n\|\}_{n\in \N}$, $\{ \|B_n\|\}_{n\in \N}$ 
are bounded. Conclude that the corresponding Jacobi operator 
is a bounded self-adjoint operator. 
Apply the spectral theorem, and show that there exists a 
matrix-valued weight for which the matrix-valued polynomials 
are orthogonal. 
\item\label{exer:Carleman} Show that $\sum_{n=0}^\infty \| A_n\|^{-1} = \infty$ implies Assumption \ref{ass:MVpolsolnotfree}.
\end{enumerate}

\section{More on matrix weights, matrix-valued orthogonal polynomials and Jacobi operators}\label{sec:MoreMweightsMVOPS-Jacobi}

In Section \ref{sec:MVOP} we have made several assumptions, notably 
Assumption \ref{ass:finitemomentsandstrictlypos-ae} and 
Assumption \ref{ass:MVpolsolnotfree}.  
In this section we discuss how to weaken the 
Assumption \ref{ass:finitemomentsandstrictlypos-ae}.

\subsection{Matrix weights}\label{ssec:MoreMweights}

Assumption \ref{ass:finitemomentsandstrictlypos-ae} is related to the space
$L_C^2(\mu)$ for a matrix-valued measure $\mu$. 
We will keep the assumption that $\tau_\mu$ has infinite support as the 
case that $\tau_\mu$ has finite support reduces to the case 
that $L_C^2(\mu)$ will be finite dimensional and we are in a situation 
of finite discrete matrix-valued orthogonal polynomials.
The second assumption in 
Assumption \ref{ass:finitemomentsandstrictlypos-ae} is that 
$W$ is positive definite $\tau_\mu$-a.e.

\begin{defn} For a positive definite matrix $W\in P_N(\C)$ define 
the projection $P_W \in M_N(\C)$ on the range of $W$. 
\end{defn}

Note that $P_W W= WP_W=W$ and $W(I-P_W)=0=(I-P_W)W$. 

In the context of Theorem \ref{thm:RDforMVmeasure} we have a
Borel measure $\tau_\mu$, so we need to consider measurability 
with respect to the Borel sets of $\R$. 

\begin{lemma}\label{lem:MV-PWmeasurable}
Put $J(x) = P_{W(x)}$, then $J\colon \R \to M_N(\C)$ is measurable. 
\end{lemma}

\begin{proof} The matrix-entries $W_{i,j}$ are measurable by Theorem \ref{thm:RDforMVmeasure}, so $W$ is 
measurable. Then $p(W)\colon \R \to M_N(\C)$ for any polynomial $p$ is measurable.
Since we have observed that $0\leq W(x) \leq I$ $\tau_\mu$-a.e., we can 
use a polynomial approximation (in sup-norm) of $\sqrt[n]{\cdot}$ on the interval $[0,1]\supset \si(W(x))$ 
$\tau_\mu$-a.e.
Hence, $\sqrt[n]{W(x)}$ is measurable, and next observe that 
$J(x) = \lim_{n\to\infty} \sqrt[n]{W(x)}$ to conclude that $J$ is measurable.  
\end{proof}

\begin{cor}\label{cor:lemMV-PWmeasurable}
The functions 
$d(x) = \dim \Ran\bigl( J(x)\bigr)$ and $\sqrt{W(x)}$ are measurable. 
So the set $D_d =\{ x\in \R \mid \dim \Ran\bigl( W(x)\bigr) =d\}$ is measurable for all $d$. 
\end{cor}

We now consider all measurable $F\colon \R \to M_N(\C)$ such that 
$\int_\R F(x) W(x) F^\ast(x) \, d\tau_\mu(x) < \infty$, which we denote by $\cL_C^2(\mu)$,  
and we mod out by  
\[
\cN_C = \{ F \in \cL_C^2(\mu) \mid \langle F, F \rangle = 0\}
\]
and then the completion of $\cL_C^2(\mu)/\cN_C$ is the corresponding Hilbert $\text{C}^\ast$-module $L^2_C(\mu)$.  

\begin{lemma}\label{lem:MVdescriptionNC} 
$\cN_C$ is a left $M_N(\C)$-module, and 
\[
\cN_C = \{ F \in \cL^2_C(\mu) \mid  \Ran\bigl(J(x)\bigr)\subset \Ker(F(x)) \quad \tau_\mu\mathrm{-a.e.} \}
\]
\end{lemma}

By taking orthocomplements the condition can be rephrased as $\Ran(F^\ast(x))\subset \Ker(J(x))$, and since 
$\Ran(J(x))=\Ran(W(x))$ and 
$\Ker(J(x))=\Ker(W(x))$ it can also be rephrased in terms of the range and kernel of $W$. 

\begin{proof}
$\cN_C$ is a left $M_N(\C)$-module by construction of the $M_N(\C)$-valued inner product. 

Observe, with $J\colon \R \to M_N(\C)$ as in Lemma  \ref{lem:MV-PWmeasurable}, that 
we can split  a function $F\in \cL^2_C(\mu)$ in the functions 
$FJ$ and $F(I-J)$, both again in $\cL^2_C(\mu)$, so that $F=FJ + F(I-J)$ and 
\begin{gather*}
\langle F, F\rangle = 
\langle FJ, FJ\rangle + \langle F(I-J), FJ\rangle 
+ \langle FJ, F(I-J)\rangle + \langle F(I-J), F(I-J)\rangle \\ 
= \langle FJ, FJ\rangle =
\int_\R (FJ)(x) W(x) (JF)^\ast(x) \, d\tau_\mu(x) 
\end{gather*}
since $(I-J(x))W(x) = 0= W(x)(I-J(x))$ $\tau_\mu$-a.e.
It follows that for any $F \in \cL^2_C(\mu)$ with $\Ran(J(x))\subset \Ker(F(x))$ 
$\tau_\mu$-a.e. the function $FJ$ is zero, and then  $F\in \cN_C$. 

Conversely, if $F\in\cN_C$ and hence 
\begin{gather*}
0 = \Tr(\langle F, F\rangle) = 
\int_\R \Tr\bigl( F(x) W(x) F^\ast(x)\bigr) \, d\tau_\mu(x) 
\end{gather*}
Since $\Tr(A^\ast A) = \sum_{k,j=1}^N |a_{k,j}|^2$ we see that all matrix-entries
of $x\mapsto F(x)\bigl(W(x)\bigr)^{1/2}$ are zero $\tau_\mu$-a.e. 
Hence $x\mapsto F(x)\bigl(W(x)\bigr)^{1/2}$ is zero $\tau_\mu$-a.e. 
This gives $x\mapsto \langle W(x) F^\ast(x) v, F^\ast(x) v\rangle =0$ for all $v\in \C^N$ and $\tau_\mu$-a.e.
Hence, $\Ran(F^\ast(x))\subset \Ker(J(x))$ $\tau_\mu$-a.e., and 
so $\Ran\bigl(J(x)\bigr)\subset \Ker(F(x))$ $\tau_\mu$-a.e.
\end{proof}

Similarly, we define the space $\cL_v^2(\mu)$ of measurable functions 
$f\colon \R \to \C^N$ so that 
\[
\int_\R f^\ast(x)\, W(x)\, f(x)\, d\tau_\mu(x) < \infty
\]
where $f$ is viewed as a column vector and $f^\ast$ as a row vector. 
Then we mod out by $\cN_v = \{ f\in\cL_v^2(\mu) \mid \langle f, f\rangle =0\}$ and we 
complete in the metric induced from the inner product 
\[
\langle f, g\rangle = \int_\R g^\ast(x)\, W(x)\, f(x)\, d\tau_\mu(x) = 
\int_\R \langle W(x)\, f(x), g(x) \rangle \, d\tau_\mu(x)
\]
The analogue of Lemma \ref{lem:MVdescriptionNC} for $L^2_v(\mu)$ is discussed in detail in 
\cite[XIII.5.8]{DunfS}.

\begin{lemma}\label{lem:MVdescriptionNv} 
$\cN_v = \{ f \in \cL^2_v(\mu) \mid  f(x)\in \Ker\bigl(J(x)\bigr)\quad \tau_\mu\mathrm{-a.e.} \}$. 
\end{lemma}

The proof of Lemma \ref{lem:MVdescriptionNv} is Exercise \ref{exer:lemMVdescriptionNv}. 

\subsection{Matrix-valued orthogonal polynomials}\label{ssec:MoreMVOPs}

In general, for a not-necessarily positive definite
matrix measure $d\mu = W\, d\tau_\mu$ with finite moments we cannot perform a Gram-Schmidt procedure, so we 
have to impose another condition. 
Note that it is guaranteed by Theorem \ref{thm:RDforMVmeasure} that $W$ is positive semi-definite.

\begin{ass}\label{ass:finitemomentsnotstrictlypos-ae}
From now on we assume for Section \ref{sec:MoreMweightsMVOPS-Jacobi} that $\mu$ is a matrix measure for which $\tau_\mu$ 
has infinite support and for which all moments exist, i.e.
$(x\mapsto x^k W_{i,j}(x)) \in L^1(\tau_\mu)$ for all $1\leq i,j\leq N$ and all $k\in \N$.
Moreover, we assume that all 
even moments $M_{2k} = \int_\R x^{2k} \, d\mu(x) = \int_\R x^{2k} W(x)\, d\tau_\mu(x)$ are 
positive definite, $M_{2k} \in P_N^o(\C^N)$, for all $k\in \N$.
\end{ass}

Note that Lemma \ref{lem:MVdescriptionNC} shows that $x^k\notin \cN_C$ (except for the trivial case),
so that $M_{2k}\not=0$. This, however, does not guarantee that $M_{2k}$ is positive definite. 

\begin{thm}\label{thm:MVthreetermrecurrence-noAssstrictly}
Under the Assumption \ref{ass:finitemomentsnotstrictlypos-ae} there exists 
a sequence of matrix-valued orthonormal polynomials $\{ P_n\}_{n\in \N}$ with 
regular leading coefficients. There exist sequences of matrices 
$\{ A_n\}_{n\in \N}$, $\{ B_n\}_{n\in \N}$ so that
$\det(A_n)\not=0$ for all $n\in \N$ and $B_n^\ast=B_n$ for all $n\in \N$, so that 
\begin{equation*}
z P_n(z) = 
\begin{cases} A_n P_{n+1}(z) + B_n P_n(z) + A_{n-1}^\ast P_{n-1}(z), & n\geq 1 \\
A_0 P_{1}(z) + B_0 P_0(z), & n=0.
\end{cases}
\end{equation*}
\end{thm}

\begin{proof} Instead of showing the existence of the orthonormal polynomials
we show the existence of the monic matrix-valued orthogonal polynomials $R_n$
so that $\langle R_n, R_n\rangle$ positive definite for all $n\in \N$.
Then $P_n = \langle R_n, R_n\rangle^{-1/2} R_n$ gives a sequence of matrix-valued orthonormal polynomials.

We start with $n=0$, then $R_0(x)=I$, and $\langle R_0, R_0\rangle = M_0 >0$ by 
Assumption \ref{ass:finitemomentsnotstrictlypos-ae}. We now assume that 
the monic matrix-valued orthogonal polynomials $R_k$
so that $\langle R_k, R_k\rangle$ is positive definite have 
been constructed for all $k <n$. We now prove the statement for $k=n$.

Put, since $R_n$ is monic,  
\[
R_n(x) = x^nI + \sum_{m=0}^{n-1} C_{n,m} R_m(x), \qquad C_{n,m}\in M_N(\C)
\]
The orthogonality requires $\langle R_n, R_m\rangle =0$ for $m<n$. This gives 
the solution 
\[
C_{n,m} = - \langle x^n, R_m\rangle \langle R_m, R_m\rangle^{-1}, \qquad m<n
\]
which is well-defined by the induction hypothesis. It remains to show
that $\langle R_n, R_n \rangle>0$, i.e. $\langle R_n, R_n \rangle$ is positive definite. 
Write $R_n(x) = x^nI + Q(x)$, so that 
\begin{gather*}
\langle R_n, R_n \rangle = \int_\R x^n W(x) x^n \, d\tau_\mu(x) + 
\langle x^n, Q \rangle +  \langle Q, x^n \rangle + \langle Q, Q \rangle
\end{gather*}
so that the first term equals the  positive definite moment $M_{2n}$ 
by Assumption \ref{ass:finitemomentsnotstrictlypos-ae}. 
It suffices to show that the other three terms are positive semi-definite, so that 
the sum is positive definite. 
This is clear for $\langle Q, Q \rangle$, and a calculation shows 
\begin{gather*}
\langle x^n, Q \rangle +  \langle Q, x^n \rangle = 
2 \sum_{m=0}^{n-1} \langle x^n, R_m\rangle \langle R_m,R_m\rangle \langle R_m, x^n\rangle 
\end{gather*}
and, since with $B>0$ we have $ABA^\ast \geq 0$, the induction 
hypothesis shows that these terms are also positive definite. 
Hence $\langle R_n, R_n \rangle$ is  positive definite. 

Establishing that the corresponding orthonormal polynomials 
satisfy a three-term recurrence relation is done as in Section \ref{sec:MVOP}, 
see Exercise \ref{exer:MoreMVOPs3termrecurrence}.
\end{proof}

We can now go through the proofs of Section \ref{sec:MVOP} and see that we
can obtain in the same way the spectral decomposition of 
the self-adjoint operator $(J^\ast, \cD^\ast)$ 
in Theorem \ref{thm:MVspectraldecomposition}, where the 
Assumption \ref{ass:finitemomentsandstrictlypos-ae} is 
replaced by Assumption \ref{ass:finitemomentsnotstrictlypos-ae}
and the Assumption \ref{ass:MVpolsolnotfree} is still in force.

\begin{cor}\label{cor:thmMVspectraldecomposition} The 
spectral decomposition of the self-adjoint
extension $(J^\ast, \cD^\ast)$ of $(J,\cD)$ of 
Theorem \ref{thm:MVspectraldecomposition} remains valid. 
The multiplicity of the spectrum is given 
by the function $d\colon \si(J^\ast)\to \N$ $\tau_\mu$-a.e. 
where $d$ is defined in Corollary \ref{cor:lemMV-PWmeasurable}.
\end{cor}

Corollary \ref{cor:thmMVspectraldecomposition} means that 
the operator $(J^\ast, \cD^\ast)$ is abstractly realised as a 
multiplication operator on a direct integral 
of Hilbert spaces $\int H_{d(x)}\, d\nu(x)$,
where $H_{d}$ is the Hilbert space of dimension $d$ 
and $\nu$ is a measure on the spectrum of $(J^\ast, \cD^\ast)$,
see e.g. \cite[Ch.~VII]{Ston} for more information. 

\subsection{Link to case of $\ell^2(\Z)$}\label{ssec:MVOPlinktoell2Z}

In \cite[\S~VII.3]{Bere} Berezanski\u\i\  discusses how three-term recurrence operators on 
$\ell^2(\Z)$ can be related to $2\times 2$-matrix recurrence on $\N$, so that we 
are in the case $N=2$ of Section \ref{sec:MVOP}. Let us discuss briefly a 
possibility to do this, following \cite[\S~VII.3]{Bere}, see
also Exercise  \ref{sec:threetermZ}.\ref{exer:2times2}. 

We identify $\ell^2(\Z)$ with $\ell^2(\C^2) = \ell^2(\N) \hat \otimes \C^2$ by
\begin{equation}\label{eq:MVidentificationl2Zandl2NC2}
e_n \mapsto e_n \otimes \begin{pmatrix} 1 \\ 0 \end{pmatrix}, \qquad 
e_{-n-1} \mapsto e_n \otimes \begin{pmatrix} 0 \\ 1 \end{pmatrix}, \quad n\in \N,
\end{equation}
where $\{e_n\}_{n\in \Z}$ denotes the standard orthonormal
basis of $\ell^2(\Z)$ and $\{e_n\}_{n\in \N}$ the standard orthonormal
basis of $\ell^2(\N)$, as before. 
The identification \eqref{eq:MVidentificationl2Zandl2NC2} is highly non-canonical. 
By calculating $L(ae_n + be_{-n-1})$ using Section \ref{sec:threetermZ} 
we get the corresponding operator $J$ acting on $\cD \subset \ell^2(\C^2)$ 
\begin{gather*}
\sum_{n=0}^\infty e_n \otimes v_n \mapsto e_0\otimes (A_0v_1+B_0v_0) 
+ \sum_{n=1}^\infty e_n \otimes (A_nv_{n+1} + B_n v_n + A^\ast_{n-1}v_{n-1}) \\
A_n = \begin{pmatrix} a_n & 0 \\ 0 & a_{-n-2} \end{pmatrix}, \ n\in \N,  \quad
B_n = \begin{pmatrix} b_n & 0 \\ 0 & b_{-n-1} \end{pmatrix}, \ n\geq 1, \quad
B_0 = \begin{pmatrix} b_0 & a_{-1} \\ a_{-1} & b_{-1} \end{pmatrix}
\end{gather*}

Using the notation of Section \ref{sec:threetermZ}, let $S^\pm_z$ be spanned
by $\phi_z= \sum_{n\in\Z} (\phi_z)_n f_n \in S^+_z$ and 
$\Phi_z= \sum_{n\in\Z} (\Phi_z)_n f_n \in S^-_z$. Then under the 
correspondence of this section, the $2\times 2$-matrix-valued function 
\begin{gather*}
F_n(z) = \begin{pmatrix} (\phi_z)_n & 0 \\ 0 & (\Phi_z)_{-n-1}\end{pmatrix}  
\in S^+_z \\
z F_n(z) = A_n F_{n+1}(z) + B_n F_n(z) + A^\ast_{n-1} F_{n-1}(z), \qquad n\geq 1. 
\end{gather*}

The example discussed in Examples  \ref{exa:ASC-q1}, \ref{exa:ASC-q2}, 
\ref{exa:ASC-q3} shows that the multiplicity of each element in 
the spectrum is $1$, so we see that the corresponding 
$2\times 2$-matrix weight measure is purely discrete and that 
$d(\{q^n\})=1$ for each $n\in \N$.

\subsection{Reducibility}\label{ssec:MVreducibility}\index{reducibility} 

Naturally, if we have positive Borel measures $\mu_p$, $1\leq p \leq N$, 
we can obtain a matrix-valued measure $\mu$ by 
putting
\begin{equation}\label{eq:MVtrivmeasure}
\mu(B) = T 
\begin{pmatrix} 
\mu_1(B) & 0 & \cdots & 0 \\
0 & \mu_2(B) & \cdots & 0 \\
\vdots& & \ddots & \vdots \\
0 & \cdots & 0 & \mu_N(B)
\end{pmatrix} T^\ast 
\end{equation}
for an invertible $T\in M_N(\C)$. Denoting the scalar-valued
orthonormal polynomials for the measure $\mu_i$ by $p_{i;n}$, then 
\begin{equation*}
P_n(x) = 
\begin{pmatrix} 
p_{1;n}(x) & 0 & \cdots & 0 \\
0 & p_{2;n}(x) & \cdots & 0 \\
\vdots& & \ddots & \vdots \\
0 & \cdots & 0 & p_{N;n}(x)
\end{pmatrix} T^{-1} 
\end{equation*}
are the corresponding matrix-valued orthogonal polynomials. 
Similarly, we can build up a matrix-valued measure of size $(N_1+N_2)\times (N_1+N_2)$
starting from a $N_1\times N_1$-matrix measure and a  $N_2\times N_2$-matrix measure. 
In such cases the Jacobi operator $J$ can be reduced as well. 

We consider the real vector space 
\begin{equation}\label{eq:definition_Acal}
\scA = \scA(\mu) =\{ T\in M_{N}(\C)\mid T\mu(B) = \mu(B) T^* \,\,\, \forall B\in \scB \},
\end{equation}
and the commutant algebra 
\begin{equation}\label{eq:definition_A}
A=A(\mu)=\{ T\in M_{N}(\C)\mid  T\mu(B) = \mu(B) T \,\,\, \forall B\in \scB\}, 
\end{equation}
which is a $\ast$-algebra,  
for any matrix-valued measure $\mu$. 

Then, by Tirao and Zurri\'an \cite[Thm.~2.12]{TiraZ}, the weight splits into
a sum of smaller dimensional weights if and only of $\R I \varsubsetneq \scA$. 
On the other hand, the commutant algebra $A$ is easier to study, and 
in \cite[Thm.~2.3]{KoelP}, it is proved that $\scA\cap \scA^\ast = A_h$,
the Hermitean elements in the commutant algebra $A$, so that 
we immediately get that $\scA = A_h$ if $\scA$ is $\ast$-invariant.
The $\ast$-invariance of $\scA$ can then be studied using 
its relation to moments, quadratic norms, the monic polynomials, and 
the corresponding coefficients in the three-term recurrence relation,
see \cite[Lemma~3.1]{KoelP}. See also Exercise \ref{exer:commutant}. 

In particular, for the case of the matrix-valued Gegenbauer polynomials 
of Example \ref{exa:MVGegenbauerpols}, we have that 
$A=\C I \oplus \C J$, where $J\colon \C^{2\ell+1}\to \C^{2\ell+1}$,
$e_n\mapsto e_{2\ell-n}$ is a self-adjoint involution, see 
\cite[Prop.~2.6]{KoeldlRR}, and that $\scA$ is $\ast$-invariant, 
see \cite[Example~4.2]{KoelP}. See also Exercise \ref{exer:MVGegenbauercommutant}. 
So in fact, we can decompose the weight in Example \ref{exa:MVGegenbauerpols}
into a direct sum of two weights obtained by projecting on the 
$\pm 1$-eigenspaces of $J$, and then there is no further reduction possible.

\subsection{Exercises}

\begin{enumerate}[1.]
\item\label{exer:lemMVdescriptionNv} 
Prove Lemma \ref{lem:MVdescriptionNv} following Lemma \ref{lem:MVdescriptionNC}. 
\item\label{exer:MoreMVOPs3termrecurrence} Prove the 
statement on the three-term recurrence relation of Theorem 
\ref{thm:MVthreetermrecurrence-noAssstrictly}. 
\item\label{exer:commutant} 
Consider the following $2\times 2$-weight function on $[0,1]$ with respect to
the Lebesgue measure;
\begin{gather*}
W(x)=
\begin{pmatrix}
x^2+x & x \\ x & x 
\end{pmatrix}
\end{gather*}
Show that $W(x)$ is  positive definite a.e. on $[0,1]$.
Show that the commutant algebra $A$ is trivial, and 
that the vector space $\scA$ is non-trivial.
\end{enumerate}

\section{The $J$-matrix method}\label{sec:JMatrixmethod}

The $J$-matrix method\index{J-matrix method@$J$-matrix method} 
consists of realising an operator to 
be studied, e.g. a Schr\"odinger operator, as 
a recursion operator in a suitable basis.
If this recursion is a three-term recursion 
then we can try to bring orthogonal polynomials 
in play.
In case, the recursion is more generally a $2N+1$-term recursion we can 
use a result of Dur\'an and Van Assche \cite{DuraVA}, see also 
\cite[\S 4]{Berg}, to write it as a three-term recursion 
for $N\times N$-matrix-valued polynomials.
The $J$-matrix method is used for a number of physics models, 
see e.g. references in \cite{IsmaK}.

We start with the case of  a linear operator $L$ acting on a suitable 
function space; typically $L$ is a differential
operator, or a difference operator. 
We look for linearly independent functions $\{ y_n\}_{n=0}^\infty$ 
such that $L$ is tridiagonal with respect to these
functions, i.e. there exist constants $A_n$, $B_n$, $C_n$ 
($n\in\N$)  such that 
\begin{equation}\label{eq:gentridiagonalform}
L\, y_n  =\begin{cases} A_n\, y_{n+1} + B_n\, y_n + C_n\, y_{n-1}, & n\geq 1, \\
A_0\, y_1 + B_0\, y_0, & n=0. 
\end{cases} 
\end{equation}
Note that we do not assume that the functions $\{y_n\}_{n\in\N}$ 
form an orthogonal or orthonormal basis. 
We combine both equations by assuming $C_0=0$. 
Note also that in case some $A_n=0$ or $C_n=0$, we can have invariant subspaces
and we need to consider the spectral decomposition on 
such an invariant subspaces, and on its complement if 
this is also invariant and otherwise on the corresponding quotient space.
An example of this will be encountered in 
Section \ref{ssec:SchrodingerMorsepotential}. 

It follows that $\sum_{n=0}^\infty p_n(z) \, y_n$ is 
a formal eigenfunction of $L$ for the eigenvalue $z$ if
$p_n$ satisfies
\begin{equation}\label{eq:recurrencepn}
z\, p_n(z) = C_{n+1}\, p_{n+1}(z) + B_n\, p_n(z) + A_{n-1}\, p_{n-1}(z)
\end{equation}
for $n\in \N$ with the convention $A_{-1}=0$. In case $C_n\not= 0$ for
$n\geq 1$, we can define $p_0(z)=1$ and use \eqref{eq:recurrencepn}
recursively to find $p_n(z)$ as a polynomials of degree $n$ 
in $z$. 
In case $A_n C_{n+1}>0$, $B_n\in\R$, $n\geq 0$, the polynomials $p_n$ are orthogonal
with respect to a positive measure on $\R$ by Favard's theorem, see 
Corollary \ref{cor:Favardsthm}, 
and the measure
and its support then can give information on $L$ in case
$\{ y_n\}_{n=0}^\infty$ gives a basis for the function space
on which $L$ acts, or for $L$ restricted to the closure of the 
span $\{ y_n\}_{n=0}^\infty$ (which depends on the function
space under consideration).
Of particular interest is whether we can match the corresponding 
Jacobi operator to a well-known class of orthogonal polynomials, e.g. 
from the ($q$-)Askey scheme. 

We illustrate this method by a couple of examples.
In the first example in Section \ref{ssec:SchrodingerMorsepotential}, 
an explicit Schr\"odinger operator is considered. 
The Schr\"odinger operator with the Morse potential is used 
in modelling potential energy in diatomic molecules, and it is
physically relevant since it allows for bound states, which 
is reflected in the occurrence of an invariant 
finite-dimensional subspace of the corresponding Hilbert space
in Section \ref{ssec:SchrodingerMorsepotential}.

In the second example we use an explicit differential 
operator for orthogonal polynomials to construct 
another differential operator suitable for 
the $J$-matrix method. We work out the details in 
a specific case. 

In the third example we extend the method to 
obtain an operator for which we have a $5$-term recurrence 
relation, to which we associate $2\times 2$-matrix valued
orthogonal polynomials.


\subsection{Schr\"odinger equation with Morse potential}\label{ssec:SchrodingerMorsepotential}

The Schr\"odinger equation with Morse potential is studied by
Broad \cite{Broa-LNM} and Diestler \cite{Dies} in the study of 
a larger system of coupled equations used in
modeling atomic dissocation. The Schr\"odinger equation with Morse potential
is used to model a two-atom molecule in this larger system. 
We use the approach as discussed in \cite[\S 3]{IsmaK}. 

The Schr\"odinger equation with Morse potential\index{Schr\"odinger equation with Morse potential} is
\begin{equation}\label{eq:SchrodingerMorsepotential}
-\frac{d^2}{dx^2} + q, \qquad q(x) = b^2(e^{-2x}-2e^{-x}), 
\end{equation}
which is an unbounded operator on $L^2(\R)$. Here $b>0$
is a constant. 
It is a self-adjoint operator with respect to its form domain,
see \cite[Ch.~5]{Sche} and $\lim_{x\to\infty} q(x) =0$, and
$\lim_{x\to-\infty} q(x) =+\infty$. 
Note $\min (q)=-b^2$, so that by general results in
scattering theory the discrete spectrum is contained
in $[-b^2,0]$ and it consists of isolated points, and we 
show how they occur in this approach. 

We look for solutions to $-f''(x) + q(x)f(x)=\ga^2 f(x)$.
Put $z= 2be^{-x}$ so that $x\in\R$ corresponds to $z\in(0,\infty)$,
and let $f(x)$ correspond to $\frac{1}{\sqrt{z}} g(z)$, then
\begin{equation}\label{eq:Whittakereq}
g''(z) + \frac{(-\frac14 z^2+bz +\ga^2 +\frac14)}{z^2} g(z) = 0.
\end{equation}
which is precisely the Whittaker equation\index{Whittaker equation} 
with $\ka=b$, $\mu=\pm i\ga$, and
the Whittaker integral transform gives the spectral decomposition
for this Schr\"{o}dinger equation, see e.g. \cite[\S~IV]{Fara}.
In particular, depending on the value of $b$ the
Schr\"odinger equation has finite discrete spectrum, i.e.
bound states, see the Plancherel formula \cite[\S~IV]{Fara}, and
in this case the Whittaker function terminates and can be written
as a Laguerre polynomial of type $L_m^{(2b-2m-1)}(x)$,\index{orthogonal polynomials!Laguerre polynomials} 
for those $m\in \N$ such that $2b-2m>0$. 
So the spectral decomposition can be done directly using the 
Whittaker transform.

We now indicate how the spectral decomposition of 
three-term recurrence (Jacobi) operators can be used to find the 
spectral decomposition as well. 
The Schr\"{o}dinger operator is tridiagonal
in a basis introduced by Broad \cite{Broa-LNM} and Diestler \cite{Dies}.
Put $N=\# \{ n\in\N \, |\,  n < b-\frac12 \}$, i.e. 
$N= \lfloor b+\frac12 \rfloor$, so that $2b-2N > -1$, and we assume
for simplicity $b\notin \frac12+\N$. 
Let $T\colon L^2(\R)\to L^2((0,\infty); z^{2b-2N}e^{-z}dz)$ be the map
$(Tf)(z) = z^{N-b-\frac12}e^{\frac12 z}\, f(\ln(2b/z))$, then
$T$ is unitary, and
\[
T\bigl(-\frac{d^2}{dx^2} + q\bigr) T^\ast=L \qquad
L= M_A \frac{d^2}{dz^2} + M_B \frac{d}{dz} + M_C
\]
where $M_f$ denotes the operator of multiplication by $f$.
Here $A(z)=-z^2$, $B(z)=(2N-2b-2 + z)z$, $C(z)=-(N-b-\frac12)^2 + z(1-N)$. 
Using the second-order differential equation, see 
e.g. \cite[(4.6.15)]{Isma}, \cite[(1.11.5)]{KoekS}, \cite[(5.1.2)]{Szeg}, for
the Laguerre polynomials, the three-term recurrence
relation for the Laguerre polynomials, 
see e.g. \cite[(4.6.26)]{Isma}, \cite[(1.11.3)]{KoekS}, \cite[(5.1.10)]{Szeg},
and the differential-recursion formula 
\[
x\frac{d}{dx} L^{(\al)}_n(x) \, = \, n\, L^{(\al)}_n(x) \, 
- (n+\al)\, L^{(\al)}_{n-1}(x)
\]
see \cite[Case II]{AlSaC}, 
for the Laguerre polynomials we find that this
operator is tridiagonalized by the Laguerre polynomials
$L_n^{(2b-2N)}$. 

Translating this back to the Schr\"{o}dinger
operator we started with,  we obtain 
\begin{equation*}
y_n(x) = (2b)^{(b-N+\frac12)}\sqrt{\frac{n!}{\Ga(2b-2N+n+1)}}
e^{-(b-N+\frac12)x} e^{-be^{-x}}
\, L^{(2b-2N)}_n(2be^{-x}) 
\end{equation*}
as an orthonormal basis for $L^2(\R)$ such that 
\begin{equation}\label{eq:ttrpart13}
\begin{split}
\Bigl(-\frac{d^2}{dx^2} + q\Bigr) y_n \, = &
\, -(1-N+n)\sqrt{(n+1)(2b-2N+n+1)}\, y_{n+1}  \\ 
\, &+\, \Bigl( -(N-b-\frac12)^2 + (1-N+n)(2n+2b-2N+1) - n  \Bigr)\, y_n \\
\, &\, -(n-N) \sqrt{n(2b-2N+n)}\,\,  y_{n-1}.
\end{split}
\end{equation}
Note that \eqref{eq:ttrpart13} is written in a symmetric 
tridiagonal form. 

The space $\cH^+$ spanned by $\{ y_n\}_{n=N}^\infty$ and the
space $\cH^-$ spanned by $\{ y_n\}_{n=0}^{N-1}$ are invariant with
respect to $-\frac{d^2}{dx^2} + q$ which follows 
from \eqref{eq:ttrpart13}.
Note that $L^2(\R) = \cH^+ \oplus \cH^-$, $\dim (\cH^-)=N$.
In particular, there will be discrete eigenvalues, hence bound 
states, for the restriction to $\cH^-$. 

In order to determine the spectral properties of the
Schr\"{o}dinger operator, we first consider
its restriction on the finite-dimensional invariant subspace
$\cH^-$. We look for eigenfunctions $\sum_{n=0}^{N-1} P_n(z)\, y_n$ for
eigenvalue $z$, so we need to solve 
\begin{equation*}
\begin{split}
z\, P_n(z) &\, = (N-1-n)\sqrt{(n+1)(2b-2N+n+1)}\, P_{n+1}(z)  
\\ &\qquad + \, \Bigl( -(N-b-\frac12)^2 + (1-N+n)(2n+2b-2N+1) - n  \Bigr)\, P_n(z)  
\\ &\qquad +  (N-n) \sqrt{n(2b-2N+n)}\,  P_{n-1}(z), \qquad 0\leq n\leq N-1.
\end{split}
\end{equation*}
which corresponds to some orthogonal polynomials on a finite 
discrete set. These polynomials are expressible in terms
of the dual Hahn polynomials, see \cite[\S 6.2]{Isma}, \cite[\S 1.6]{KoekS}, 
\index{orthogonal polynomials!dual Hahn polynomials} 
and we find that
$z$ is of the form $-(b-m-\frac12)^2$, $m$ a nonnegative integer
less than $b-\frac12$, and 
\begin{equation*}
\begin{split}
&\, P_n(-(b-m-\frac12)^2) = \sqrt{\frac{(2b-2N+1)_n}{n!}}
\, R_n(\la(N-1-m); 2b-2N,0, N-1), 
\end{split}
\end{equation*}
using the notation of \cite[\S 6.2]{Isma}, \cite[\S 1.6]{KoekS}. 
Since we have now two expressions for the
eigenfunctions of the Schr\"{o}dinger operator for a 
specific simple eigenvalue, we obtain, 
after simplifications,
\begin{gather}
\sum_{n=0}^{N-1} \, R_n(\la(N-1-m); 2b-2N,0, N-1)\,
L^{(2b-2N)}_n(z)\,   =\, C\,   z^{N-1-m} \, L^{(2b-2m-1)}_m(z), 
\label{eq:expansionLaguerreanddualHahn} \\ 
 C\, = \, (-1)^{N+m+1} \left( (N+m-2b)_{N-1-m} \binom{N-1}{m}\right)^{-1} \nonumber
\end{gather}
where the constant $C$ can be determined by e.g. considering leading
coefficients on both sides. 

On the invariant subspace $\cH^+$ we look for formal eigenvectors 
$\sum_{n=0}^\infty P_n(z)\ y_{N+n}(x)$ for the eigenvalue $z$. 
This leads to the recurrence relation 
\begin{equation*}
\begin{split}
z\, P_n(z) &\, =  -(1+n)\sqrt{(N+n+1)(2b-N+n+1)}\,  P_{n+1}(z) \\
&\qquad + \bigl( -(N-b-\frac12)^2 +(1+n)(2n+2b+1)- n-N\bigr)\, P_n(z) \\
&\qquad
-n \sqrt{(N+n)(2b-N+n)}\, P_{n-1}(z).
\end{split}
\end{equation*}
This corresponds with the three-term recurrence relation for the
continuous dual Hahn polynomials, see \cite[\S 1.3]{KoekS}, \index{orthogonal polynomials!continuous dual Hahn polynomials} 
with $(a,b,c)$ replaced by  $(b+\frac12, N-b+\frac12, b-N+\frac12)$, and
note that the coefficients $a$, $b$ and $c$ are positive. 
We find, with $z=\ga^2\geq 0$ 
\begin{equation*}
\begin{split}
P_n(z) &\, = 
\frac{S_n(\ga^2;b+\frac12, N-b+\frac12, b-N+\frac12)}
{n! \sqrt{(N+1)_n\, (2b-N+1)_n}} 
\end{split}
\end{equation*}
and these polynomials satisfy 
\begin{gather*}
\int_0^\infty P_n(\ga^2) P_m(\ga^2) \, w(\ga)\, d\ga = \de_{n,m}, 
\\ 
\, w(\ga) =\frac{1}{2\pi\, N!\, \Ga(2b-N+1)} \left| \frac{\Ga(b+\frac12 +i\ga)\Ga(N-b+\frac12 +i\ga)\Ga(b-N+\frac12 +i\ga)}{\Ga(2i\ga)}\right|^2.
\end{gather*}
Note that the series $\sum_{n=0}^\infty P_n(\ga^2)\, y_{N+n}$ diverges in 
$\cH^+$ (as a closed subspace of $L^2(\R)$). Using 
the results on spectral decomposition of Jacobi operators
as in Section \ref{sec:threetermN}, we 
obtain the spectral decomposition
of the Schr\"{o}dinger operator restricted
to $\cH^+$ as 
\begin{gather*}
\Upsilon\colon  \cH^+ \to L^2((0,\infty); w(\ga)\, d\ga), 
\qquad \bigl( \Upsilon y_{N+n}\bigr)(\ga) = P_n(\ga^2), \\
\langle ( -\frac{d^2}{dx^2} + q ) f, g\rangle = 
\int_0^\infty \ga^2 (\Upsilon f)(\ga)  \overline{(\Upsilon g)(\ga)}\, w(\ga)\, d\ga
\end{gather*}
for $f,g\in \cH^+\subset L^2(\R)$ such that $f$ is in the domain
of the Schr\"odinger operator. 

In this way we have obtained the spectral decomposition of the Schr\"odinger
operator on the invariant subspaces $\cH^-$ and $\cH^+$, where 
the space $\cH^-$ is spanned by the bound states, i.e. by the eigenfunctions
for the negative eigenvalues, and $\cH^+$ is the reducing subspace on
which the  Schr\"odinger
operator has spectrum $[0,\infty)$. 
The link between the two approaches for the discrete spectrum is
given by \eqref{eq:expansionLaguerreanddualHahn}. For the
continuous spectrum it leads to the fact that the Whittaker
integral transform maps Laguerre polynomials to continuous
dual Hahn polynomials, and we can interpret
\eqref{eq:expansionLaguerreanddualHahn} also in this way.
For explicit formulas we refer to 
\cite[(5.14)]{Koor-LNM}. 

Koornwinder \cite{Koor-LNM} generalizes
this to the case of the Jacobi function transform mapping 
Jacobi polynomials to Wilson polynomials, which in turn has been 
generalized by Groenevelt \cite{Groe-WT} to the Wilson function
transform, an integral transformation with a ${}_7F_6$ as kernel, 
mapping Wilson polynomials to Wilson 
polynomials, which is at the highest level of the Askey-scheme,
see Figure  \ref{fig:Askeyscheme}. 
Note that conversely, we can define a unitary map $U\colon L^2(\mu)\to L^2(\nu)$ 
between two weighted $L^2$-spaces by mapping an orthonormal basis $\{\phi_n\}_{n\in \N}$ 
of  $L^2(\mu)$ to an orthonormal basis $\{\Phi_n\}_{n\in \N}$ of $L^2(\nu)$.
Then we can define formally a map $U_t\colon L^2(\mu)\to L^2(\nu)$ by 
\[
(U_tf)(\la) = \int_\R f(x) \sum_{k=0}^\infty t^k \phi_k(x)\Phi_k(\la) \, d\mu(x)
\]
and consider convergence as $t\to 1$. Note that the convergence of the (non-symmetric)
Poisson kernel\index{Poisson kernel} $\sum_{k=0}^\infty t^k \phi_k(x)\Phi_k(\la)$ 
needs to be studied carefully. 
In case of the Hermite functions as eigenfunctions of the Fourier transform, this 
approach is due to Wiener \cite[Ch.~1]{Wien}, in which the Poisson kernel is 
explicitly known as the Mehler formula. 
More information on explicit expressions of non-symmetric Poisson kernels 
for orthogonal polynomials from the $q$-Askey scheme can 
be found in \cite{AskeRS}.


\subsection{A tridiagonal differential operator}\label{ssec:JmatrixDO}
In this section we create tridiagonal operators from 
explicit well-known operators, and we show in 
an explicit example how this works. 
This is example is based on \cite{IsmaK-AA}, and we refer to 
\cite{IsmaK}, \cite{IsmaK-SIGMA} for more examples and general
constructions. 
Genest et al. \cite{GeneIVZ-PAMS2016} have generalised this 
approach and have obtained the full family of Wilson polynomials 
in terms of an algebraic interpretation. 

Assume now  $\mu$ and $\nu$ are orthogonality measures of infinite support 
for orthogonal 
polynomials;
\begin{equation*} 
\int_\R P_n(x)P_m(x)\, d\mu(x)\,=\, H_n\de_{n,m}, \qquad
\int_\R p_n(x)p_m(x)\, d\nu(x)\,=\, h_n\de_{n,m}.
\end{equation*}
We assume that both $\mu$ and $\nu$ correspond to a determinate 
moment problem, so that the space $\cP$ of polynomials is dense in $L^2(\mu)$ and $L^2(\nu)$. 
We also assume that $\int_\R f(x)\, d\mu(x) = \int_\R f(x)r(x)\, d\nu(x)$, 
where $r$ is a polynomial of degree $1$, 
so that the Radon-Nikodym derivative 
$\frac{d\nu}{d\mu}=\de = 1/r$. 
Then we obtain, using $\lc(p)$ for the leading coefficient of a polynomial $p$, 
\begin{equation}\label{eq:explexpgenpols}
p_n\, = \, \frac{\lc(p_n)}{\lc(P_n)}\, P_n \, + \, \lc(r)
\frac{h_n}{H_{n-1}}\frac{\lc(P_{n-1})}{\lc(p_n)}\, P_{n-1}
\end{equation}
by expanding $p_n$ in the basis $\{P_n\}_{n\in\N}$. Indeed, 
$p_n(x) = \sum_{k=0}^n c^n_kP_k(x)$ with 
\[
 c_k^n H_k = \int_\R p_n(x) P_k(x) \, d\mu(x) = 
\int_\R p_n(x) P_k(x) r(x) \, d\nu(x),
\]
so that $c_k^n=0$ for $k<n-1$ by orthogonality 
of the polynomials $p_n\in L^2(\nu)$. Then $c_n^n$ follows by comparing 
leading coefficients, and 
\[
c_{n-1}^n  = \int_\R p_n(x) P_k(x) r(x) \, d\nu(x) = \frac{\lc(P_{n-1}\lc(r)}{\lc(p_n)}h_n.
\]
By taking $\phi_n$, respectively $\Phi_n$, the corresponding 
orthonormal polynomials to $p_n$, respectively $P_n$, we see that
\begin{equation}\label{eq:OPphicombinationofPhi}
\begin{split}
&\phi_n \, = \, A_n\, \Phi_n \, + \, B_n \, \Phi_{n-1}, \qquad A_n \, = \, \frac{\lc(p_n)}{\lc(P_n)} \sqrt{\frac{H_n}{h_n}}, \quad
B_n \, = \, \lc(r)
\sqrt{\frac{h_n}{H_{n-1}}}\frac{\lc(P_{n-1})}{\lc(p_n)}.
\end{split}
\end{equation}
We assume the existence of a self-adjoint operator $L$ with domain $\cD=\cP$ on $L^2(\mu)$ with 
$LP_n =\La_n P_n$, and so $L\Phi_n =\La_n \Phi_n$, for eigenvalues $\La_n\in \R$. 
By convention $\La_{-1} =0$. So this means that 
we assume that $(P_n)_{n\in \N}$ satisfies a bispectrality property, and we can 
typically take the 
family $(P_n)_n$ from the Askey scheme or its $q$-analogue, 
see Figure \ref{fig:Askeyscheme}, \ref{fig:qAskeyscheme}. 

\begin{lemma}\label{lem:TDfromL}
The operator $T = r(L+\ga)$ with domain $\cD = \cP$ on $L^2(\nu)$
is tridiagonal with respect to the basis $\{\phi_n\}_{n\in\N}$. 
Here $\ga$ is a constant, and $r$ denotes multiplication by the polynomial $r$ 
of degree $1$. 
\end{lemma}

\begin{proof} Note that $(L+\ga)\Phi_n = \La_n^\ga \Phi_n = (\La_n + \ga)\Phi_n$ and 
\begin{equation*}
\begin{split}
\langle T\phi_n, \phi_m\rangle_{L^2(\nu)} \, =\, & 
\langle A_n\, T \Phi_n \, + \, B_n \, T \Phi_{n-1}, 
A_m\, \Phi_m \, + \, B_m \, \Phi_{m-1} \rangle_{L^2(\nu)} \\
 =\, & \langle A_n\, (L+\ga) \Phi_n \, + \, B_n \, (L+\ga) \Phi_{n-1}, 
A_m\, \Phi_m \, + \, B_m \, \Phi_{m-1} \rangle_{L^2(\mu)} \\
 =\, & \La^\ga_n A_n  B_{n+1}\de_{n+1,m} \, + \, (A_n^2\, \La^\ga_n +B_n^2 \La^\ga_{n-1}) \de_{n,m}
\, + \, \La^\ga_{n-1} A_{n-1}  B_{n}\de_{n,m+1}. 
\end{split}
\end{equation*}
so that 
\begin{gather*}
T\phi_n\, = \, a_n \phi_n \, +\,   b_n\phi_n\,  +\,  a_{n-1} \phi_{n-1}, \\
a_n \, = \, \La^\ga_n \lc(r) \frac{\lc(p_n)}{\lc(p_{n+1})} \sqrt{\frac{h_{n+1}}{h_n}}, 
\quad
b_n \, = \, \La^\ga_n \frac{H_n}{h_n} \left( \frac{\lc(p_n)}{\lc(P_n)}\right)^2 
\, + \, \La^\ga_{n-1} \lc(r)^2 \frac{h_n}{H_{n-1}} 
\left( \frac{\lc(P_{n-1})}{\lc(p_n)}\right)^2.  \qedhere
\end{gather*}
\end{proof}

So we need to solve for the orthonormal polynomials $r_n(\la)$ 
satisfying 
\begin{equation*} 
\begin{split}
\la r_n(\la) &= a_n r_n(\la) +  b_n r_n(\la) + a_{n-1} r_{n-1}(\la), 
\end{split}
\end{equation*}
where we assume that we can use the parameter $\ga$ in order ensure that 
$a_n\not=0$. If $a_n=0$, then we need to proceed as in Section \ref{ssec:SchrodingerMorsepotential} and split the space into invariant subspaces. 

This is a general set-up to find tridiagonal operators. In general, the three-term recurrence
relation of Lemma \ref{lem:TDfromL} needs not be matched with a known family of
orthogonal polynomials, such as e.g. from the Askey-scheme.
Let us work out a case where it does, namely for the Jacobi polynomials and the related hypergeometric differential operator. 
See \cite{IsmaK-AA} for other cases. 

For the Jacobi polynomials $P^{(\al,\be)}_n(x)$, we follow the standard notation 
\cite{AndrAR}, \cite{Isma}, 
\cite{KoekS}. \index{orthogonal polynomials!Jacobi polynomials}
We take the measures $\mu$ and $\nu$ to be the
orthogonality measures for the Jacobi polynomials for parameters $(\al+1,\be)$, 
and $(\al,\be)$ respectively. We assume $\al,\be>-1$. So we set $P_n(x)=P^{(\al+1,\be)}_n(x)$, 
$p_n(x)=P^{(\al,\be)}_n(x)$. This gives
\begin{gather*}
h_n = N_n(\al) = \frac{2^{\al+\be+1}}{2n+\al+\be+1}
\frac{\Ga(n+\al+1)\Ga(n+\be+1)}{\Ga(n+\al+\be+1)\, n!}, \quad H_n = N_n(\al+1), 
\\ 
\lc(p_n) = l_n(\al) 
= \frac{(n+\al+\be+1)_n}{2^n n!}, \quad \lc(P_n) = l_n(\al+1).
\end{gather*}
Moreover, $r(x)=1-x$. 
Note that we could have also shifted in $\be$, but due to the symmetry 
$P_n^{(\al,\be)}(x) = (-1)^n P_n^{(\be,\al)}(-x)$ 
of the Jacobi polynomials in $\al$ and $\be$ it suffices to consider the
shift in $\al$ only.

The Jacobi polynomials are eigenfunctions of a hypergeometric differential operator
\begin{gather}
L^{(\al,\be)} f (x)\, = \, (1-x^2)\, f''(x) + \bigl( \be-\al-(\al+\be+2)x\bigr) f'(x), 
\label{eq:DOforJacobipols}\\
L^{(\al,\be)}P^{(\al,\be)}_n \, = \, -n(n+\al+\be+1)\, P^{(\al,\be)}_n \nonumber
\end{gather}
and we take $L=L^{(\al+1,\be)}$ so that $\La_n = -n(n+\al+\be+2)$. 
We set $\ga= -(\al+\de+1)(\be-\de+1)$, so that
we have the factorisation $\La_n^\ga = -(n+\al+\de+1)(n+\be -\de+1)$.
So on $L^2([-1,1], (1-x)^\al (1+x)^\be\, dx)$  we study the operator 
$T= (1-x)(L+\ga)$.
Explicitly  $T$ is the second-order differential operator 
\begin{equation}\label{eq:JacobidefT}
T  =  (1-x)(1-x^2) \frac{d^2}{dx^2} + (1-x)\bigl( \be-\al-1-(\al+\be+3)x\bigr) \frac{d}{dx} - (1-x)(\al+\de+1)(\be-\de+1),
\end{equation}
which is tridiagonal by construction.
Going through the explicit details of Lemma \ref{lem:TDfromL} 
we find the explicit expression for the recursion coefficients in the three-term realisation of $T$; 
\begin{multline*} 
a_n =  \frac{2 (n+\al+\de+1)(n+\be-\de+1)}{2n+\al+\be+2} 
\sqrt{\frac{(n+1)\, (n+\al+1)\, (n+\be+1)\, (n+\al+\be+1)}
{(2n+\al+\be+1)\, (2n+\al+\be+3)}}
\\ 
b_n =  - \frac{2(n+\al+\de+1)(n+\be-\de+1)(n+\al+1)(n+\al+\be+1)}
{(2n+\al+\be+1)\, (2n+\al+\be+2)}   \qquad\qquad\qquad\qquad\qquad \\
\, - \, \frac{2 n (n+\be) (n+\al+\de+1)(n+\be-\de)}
{(2n+\al+\be)\, (2n+\al+\be+1)}. 
\end{multline*}
Then the recursion relation 
from Lemma \ref{lem:TDfromL} for $\frac12 T$ is solved by the orthonormal version of 
the Wilson polynomials \cite[\S 1.1]{KoekS}, \cite[\S 9.1]{KoekLS}, \index{orthogonal polynomials!Wilson polynomials}
\[
W_n(\mu^2; \frac12(1+\al), \frac12(1+\al)+\de, \frac12 (1-\al)+\be-\de, \frac12 (1+\al)),
\]
where the relation between the eigenvalue $\la$ of $T$ and $\mu^2$ is given by 
$\la = -2 \left(\frac{\al+1}{2}\right)^2-2\mu^2$. 
Using the spectral decomposition of a Jacobi operator as in Section \ref{sec:threetermN}
proves the following theorem. 

\begin{thm}\label{thm:Jacobi} Let $\al>-1$, $\be>-1$, and assume $\ga= -(\al+\de+1)(\be-\de+1)\in \R$. The unbounded operator $(T, \cP)$ defined by \eqref{eq:JacobidefT} on  $L^2([-1,1], (1-x)^\al (1+x)^\be\, dx)$
with domain the polynomials $\cP$  is essentially self-adjoint. The spectrum  
of the closure $\bar{T}$ is simple and given by 
\[
\begin{split}
(-\infty, &-\frac12(\al+1)^2) \cup 
\{ -\frac12(\al+1)^2+2(\frac12(1+\al)+\de+k)^2 \, \colon \,  k\in\N, \ \frac12(1+\al)+\de+k<0\} \\ & \cup 
\{ -\frac12(\al+1)^2+2(\frac12(1-\al)+\be-\de+l)^2 \, \colon \,  l\in\N, \ \frac12(1-\al)+\be-\de+l<0\} 
\end{split}
\]
where the first set gives the absolutely continuous spectrum and the other sets correspond to the 
discrete spectrum of the closure of $T$. 
The discrete spectrum consists of at most one of these sets, and can be empty.  
\end{thm}

Note that in Theorem \ref{thm:Jacobi} we require $\de\in \R$ or $\Re \de = \frac12(\be-\al)$. 
In the second case there is no discrete spectrum. 

The eigenvalue equation $T f_\la = \la f_\la$ is a second-order differential operator 
with regular singularities at $-1$, $1$, $\infty$. 
In the Riemann-Papperitz notation, see e.g. \cite[\S 5.5]{Temm}, it 
is 
\[
\cP \left\{ \begin{matrix} -1 & 1 & \infty &  \\
0 & -\frac12(1+\al) + i \tilde\la & \al+\de+1 & x \\
-\be & -\frac12(1+\al) + i \tilde\la & \be-\de+1 &  
                    \end{matrix}
\right\}
\]
with the reparametrisation $\la = -\frac12(\al+1)^2 -2\tilde\la^2$ of the spectral parameter. 
The case $\ga=0$, we can exploit this relation and establish a link to the 
Jacobi function transform mapping (special) Jacobi polynomials to (special) Wilson
polynomials, see \cite{Koor-LNM}. We refer to \cite{IsmaK-AA} for the details.
Going through this procedure and starting with the Laguerre polynomials and taking
special values for the additional parameter gives results relating Laguerre 
polynomials to Meixner polynomials involving confluent hypergeometric functions, i.e.
Whittaker functions. This is then related to the results of Section \ref{ssec:SchrodingerMorsepotential}.
Genest et al. \cite{GeneIVZ-PAMS2016} show how to extend this method in order to 
find the full $4$-parameter family of Wilson polynomials in this way. 

\subsection{$J$-matrix method with matrix-valued orthogonal polynomials}
\label{ssec:JMatrixMVOP}

We generalise the situation of Section \ref{ssec:Jacobi-operator} to 
operators that are $5$-diagonal in a suitable basis.
By Dur\'an and Van Assche \cite{DuraVA}, see also e.g. \cite{Berg}, \cite{DuraLR-Laredo}, 
a $5$-diagonal recurrence can be written as a three-term recurrence relation 
for $2\times 2$-matrix-valued orthogonal polynomials. More generally, 
Dur\'an and Van Assche \cite{DuraVA} show that 
$2N+1$-diagonal recurrence can be written as a three-term recurrence relation 
for $N\times N$-matrix-valued orthogonal polynomials, and we leave it to the 
reader to see how the result of this section can be generalised to 
$2N+1$-diagonal operators. 
The results of this section are based on \cite{GroeIK}, and we specialise 
again to the case of the Jacobi polynomials. 
Another similar example is based on the little $q$-Jacobi polynomials 
and other operators which arise as $5$-term recurrence operators in
a natural way, see \cite{GroeIK} for these cases. 

In Section \ref{ssec:JmatrixDO} we used known orthogonal polynomials, in particular
their orthogonality relations, in order to find spectral information on a 
differential operator. In this section we generalise the approach of 
Section \ref{ssec:JmatrixDO} by assuming now that the polynomial $r$, 
the inverse of the Radon-Nikodym derivative, is of degree $2$. 
This then leads to a $5$-term recurrence relation, see 
Exercise \ref{exer:5termrecurrence}.
Hence we have 
an explicit expression for the matrix-valued Jacobi operator. 
Now we assume that the 
resulting differential or difference operator leads to an operator of
which the spectral decomposition is known. Then we can find from this 
information the orthogonality measure for the matrix-valued 
polynomials. This leads to a case of matrix-valued 
orthogonal polynomials where both the orthogonality measure and the 
three-term recurrence can be found explicitly. 

So let us start with the general set-up. 
Let $T$ be an operator on a Hilbert space $\cH$ of functions, typically  
a second-order difference or differential operator. We assume that $T$ 
has the following properties;
\begin{enumerate}[(a)]
\item $T$ is (a possibly unbounded) self-adjoint operator on $\cH$ (with domain $D$ in 
case $T$ is unbounded);
\item there exists an orthonormal basis $\{f_n\}_{n=0}^\infty$ of $\cH$ so that 
$f_n\in D$ in case $T$ is unbounded and so that there exist sequences 
$(a_n)_{n=0}^\infty$, $(b_n)_{n=0}^\infty$, $(c_n)_{n=0}^\infty$ of 
complex numbers with 
$a_n>0$, $c_n\in\R$,  for all $n\in \N$ so that 
\begin{equation}\label{eq:Tis5term}
T\, f_n\, = \, a_n f_{n+2} \, + \, b_n f_{n+1} \, + \, c_n f_n \, + \, 
\overline{b_{n-1}}f_{n-1} \, + \, a_{n-2}f_{n-2}.
\end{equation}
\end{enumerate}

Next we assume that we have a suitable spectral decomposition of $T$. 
We assume that the spectrum $\si(T)$ is simple or at most of multiplicity $2$. 
The double spectrum is contained in 
$\Om_2\subset \si(T) \subset \R$, and the simple spectrum is contained in 
$\Om_1=\si(T)\setminus \Om_2\subset \R$. 
Consider functions $f$ defined on $\si(T)\subset \R$ so that 
$f\vert_{\Om_1} \colon \Om_1\to \C$ and $f\vert_{\Om_2} \colon \Om_2\to \C^2$.
We let $\si$ be a Borel measure on $\Om_1$ and $V\, \rho$ a 
$2\times 2$-matrix-valued measure 
on $\Om_2$ as in \cite[\S 1.2]{DamaPS}, so $V\colon \Om_2 \to M_2(\C)$ maps into the 
positive semi-definite matrices and $\rho$ is a positive Borel measure on $\Om_2$. 
We assume $V$ is positive semi-definite $\rho$-a.e., but not necessarily 
positive definite. 

Next we consider the weighted Hilbert space $L^2(\cV)$ of such functions for which
\[
\int_{\Om_1} |f(\la)|^2\, d\si(\la) \, + \, \int_{\Om_2} f^\ast(\la) 
V(\la) f(\la)\, d\rho(\la) \, < \, \infty
\]
and we obtain $L^2(\cV)$ by modding out by the functions of norm zero, see 
the discussion in Section \ref{ssec:MoreMweights}. 
The inner product is given by
\[
\langle f, g\rangle \, = \, \int_{\Om_1} f(\la)\overline{g(\la)}\, d\si(\la) \, + \, 
\int_{\Om_2} g^\ast(\la) V(\la) f(\la)\, d\rho(\la).
\]

The final assumption is then
\begin{enumerate}
 \item[(c)] there exists a unitary map $U\colon \cH \to L^2(\cV)$ so that $UT=MU$, where $M$ is 
 the multiplication operator by $\la$ on $L^2(\cV)$. 
\end{enumerate}

Note that assumption (c) is saying that $L^2(\cV)$ is the spectral decomposition of 
$T$, and since this also gives the spectral decomposition of polynomials in $T$, 
we see that all moments exist in $L^2(\cV)$. 

Under the assumptions (a), (b), (c) we link the spectral measure to an orthogonality measure 
for matrix-valued orthogonal polynomials.
Apply $U$ to the $5$-term expression \eqref{eq:Tis5term} for $T$ on the basis 
$\{f_n\}_{n=0}^\infty$, so that 
\begin{multline}\label{eq:5trrforUfn}
\la (Uf_n)(\la)\, = \, a_n (Uf_{n+2})(\la) \, + \, b_n (Uf_{n+1})(\la) \, 
\\ +\, c_n (Uf_n)(\la) \, + \, \overline{b_{n-1}}(Uf_{n-1})(\la) \, + \, a_{n-2}(Uf_{n-2})(\la)
\end{multline}
to be interpreted as an identity in $L^2(\cV)$. Restricted to $\Om_1$ \eqref{eq:5trrforUfn} 
is a scalar identity, and restricted to $\Om_2$ the components of 
$Uf(\la) = (U_1f(\la),U_2f(\la))^t$ satisfy \eqref{eq:5trrforUfn}. 

Working out the details for $N=2$ of \cite{DuraVA}, we see that we have to generate the 
$2\times 2$-matrix-valued polynomials by 
\begin{equation}\label{eq:3trr2by2MVOP}
\begin{split}
 \la\, P_n(\la) \, &= 
\begin{cases} \, A_n\, P_{n+1}(\la)
\, +\, B_n P_n(\la)
\, +\, A_{n-1}^\ast P_{n-1}(\la), & n\geq 1, \\
\, A_0\, P_{1}(\la)
\, +\, B_0 P_0(\la), & n=0, 
\end{cases}\\ 
A_n\, &= \, \begin{pmatrix} a_{2n} & 0 \\ b_{2n+1} & a_{2n+1} \end{pmatrix}, \qquad
B_n\, = \, \begin{pmatrix} c_{2n} & b_{2n} \\ \overline{b_{2n}} & c_{2n+1} \end{pmatrix}
\end{split}
\end{equation}
with initial conditions $P_{-1}(\la)=0$ and $P_0(\la)$ is a constant non-singular matrix, 
which we take to be the identity, so $P_0(\la)=I$. Note that $A_n$ is a non-singular 
matrix and $B_n$ is a Hermitian matrix for all $n\in\N$. Then the $\C^2$-valued functions 
\begin{equation*}
\cU_n(\la)\, =\, \begin{pmatrix} Uf_{2n}(\la) \\ Uf_{2n+1}(\la) \end{pmatrix}, \qquad
\cU^1_n(\la)\, =\, \begin{pmatrix} U_1f_{2n}(\la) \\ U_1f_{2n+1}(\la) \end{pmatrix}, \qquad
\cU^2_n(\la)\, =\, \begin{pmatrix} U_2f_{2n}(\la) \\ U_2f_{2n+1}(\la) \end{pmatrix}
\end{equation*}
satisfy \eqref{eq:3trr2by2MVOP} for vectors for $\la\in \Om_1$ in the first case and 
for $\la\in \Om_2$ in the last cases. Hence, 
\begin{equation}\label{eq:Un=PnU0}
\cU_n(\la)\, = \, P_n(\la)  \cU_0(\la),  \qquad
\cU_n^1(\la)\, = \, P_n(\la)  \cU^1_0(\la), \quad
\cU_n^2(\la)\, = \, P_n(\la)  \cU^2_0(\la), 
\end{equation}
where the first holds $\si$-a.e. and the last two hold $\rho$-a.e.
We can now state the orthogonality relations for the matrix-valued orthogonal polynomials. 

\begin{thm}\label{thm:genorth2t2MVOP} 
With the assumptions (a), (b), (c) as given above, 
the $2\times 2$-matrix-valued polynomials $P_n$ generated by \eqref{eq:3trr2by2MVOP} 
and $P_0(\la)=I$  satisfy 
\[
\int_{\Om_1} P_n(\la) \, W_1(\la)\, P_m(\la)^\ast\, d\si(\la)\, + \, 
\int_{\Om_2} P_n(\la) \, W_2(\la)\, P_m(\la)^\ast\, d\rho(\la)
=\, \de_{nm} I 
\]
where 
\begin{gather*}
W_1(\la)\, = \,  \begin{pmatrix} |Uf_0(\la)|^2 & Uf_0(\la)\overline{Uf_1(\la)} \\
\overline{Uf_0(\la)}Uf_1(\la) & |Uf_1(\la)|^2 \end{pmatrix}, \qquad
\si\text{-a.e.} \\
W_2(\la)\, = \,  \begin{pmatrix} \langle Uf_0(\la), Uf_0(\la)\rangle_{V(\la)} & 
 \langle Uf_0(\la), Uf_1(\la)\rangle_{V(\la)} \\
\langle Uf_1(\la), Uf_0(\la)\rangle_{V(\la)} & \langle Uf_1(\la), Uf_1(\la)\rangle_{V(\la)}
\end{pmatrix}, \qquad
\rho\text{-a.e.} 
\end{gather*}
and $\langle x, y\rangle_{V(\la)} = x^\ast V(\la) y$. 
\end{thm}

Since we stick to the situation with the assumptions (a), (b), (c), the 
multiplicity of $T$ cannot be higher than $2$. Note that the matrices $W_1(\la)$ and 
$W_2(\la)$ are Gram matrices.
In particular,  
$\det(W_1(\la)) = 0$ for all $\la$. So the weight matrix $W_1(\la)$ 
is semi-definite positive with eigenvalues 
$0$ and $\text{tr}(W_1(\la))= |Uf_0(\la)|^2 + |Uf_1(\la)|^2>0$. Note that
\[
\text{ker}(W_1(\la)) \, = \, \C \begin{pmatrix} 
\overline{Uf_1(\la)} \\ -\overline{Uf_0(\la)} \end{pmatrix} \, = \, 
\begin{pmatrix} Uf_{0}(\la) \\ Uf_{1}(\la)\end{pmatrix}^\perp, \quad
\text{ker}(W_1(\la)-\text{tr}(W_1(\la))) = \C \begin{pmatrix} Uf_{0}(\la) \\ Uf_{1}(\la)\end{pmatrix}
\]
Moreover, $\det(W_2(\la)) = 0$ if and only if $Uf_0(\la)$ and $Uf_1(\la)$ are multiples of each other. 

Denoting the integral in Theorem \ref{thm:genorth2t2MVOP} as $\langle P_n, P_m\rangle_W$, we see that all
the assumptions on the matrix-valued inner product, as in 
the definition of the Hilbert $\text{C}^\ast$-module $L^2_C(\mu)$ in Section \ref{ssec:MVmeasurespols},
are trivially satisfied, except for $\langle Q,Q\rangle_W = 0$ implies $Q=0$ for a matrix-valued polynomial $Q$. 
We can proceed by writing $Q= \sum_{k=1}^n C_kP_k$ for suitable 
matrices $C_k$, since the leading coefficient of $P_k$ is non-singular by \eqref{eq:3trr2by2MVOP}. Then by Theorem 
\ref{thm:genorth2t2MVOP} we have $\langle Q,Q\rangle_W= \sum_{k=0}^nC_kC_k^\ast$ which is a sum 
of positive definite elements, which can only give $0$ if each of the terms is zero.
So $\langle Q,Q\rangle_W = 0$ 
implies $C_k=0$ for all $k$, hence $Q=0$. 

\begin{proof} Start using the unitarity 
\begin{equation}\label{eq:startoforthoPn}
\begin{split}
\, \de_{nm} \begin{pmatrix} 1 &  0 \\ 0 & 1\end{pmatrix} \, &= \,
 \begin{pmatrix} \langle f_{2n},f_{2m}\rangle_{\cH} &  \langle f_{2n},f_{2m+1}\rangle_{\cH} 
\\ \langle f_{2n+1},f_{2m}\rangle_{\cH} & \langle f_{2n+1},f_{2m+1}\rangle_{\cH}\end{pmatrix}\\\, &= \, 
\begin{pmatrix} \langle Uf_{2n},Uf_{2m}\rangle_{L^2(\cV)} &  \langle Uf_{2n},Uf_{2m+1}\rangle_{L^2(\cV)} 
\\ \langle Uf_{2n+1},Uf_{2m}\rangle_{L^2(\cV)} & 
\langle Uf_{2n+1},Uf_{2m+1}\rangle_{L^2(\cV)}\end{pmatrix}\\
\end{split}
\end{equation}
Split each of the inner products on the right hand side of \eqref{eq:startoforthoPn} 
as a sum over two integrals, one over $\Om_1$ and the other over $\Om_2$. First the integral 
over $\Om_1$ equals
\begin{equation}\label{eq:startoforthoPnOm1}
\begin{split}
&\, \begin{pmatrix} \int_{\Om_1}Uf_{2n}(\la) \overline{Uf_{2m}(\la)}\, d\si(\la) &  
\int_{\Om_1}Uf_{2n}(\la) \overline{Uf_{2m+1}(\la)}\, 
d\si(\la) \\ \int_{\Om_1}Uf_{2n+1}(\la) \overline{Uf_{2m}(\la)}\, d\si(\la) &
\int_{\Om_1}Uf_{2n+1}(\la) \overline{Uf_{2m+1}(\la)}\, d\si(\la)\end{pmatrix} \\
=&\, \int_{\Om_1}\begin{pmatrix} Uf_{2n}(\la) \overline{Uf_{2m}(\la)} &  
Uf_{2n}(\la) \overline{Uf_{2m+1}(\la)} \\ 
Uf_{2n+1}(\la) \overline{Uf_{2m}(\la)} &
Uf_{2n+1}(\la) \overline{Uf_{2m+1}(\la)}\end{pmatrix} \, d\si(\la)\\
=&\, \int_{\Om_1} \begin{pmatrix} Uf_{2n}(\la) \\ Uf_{2n+1}(\la)\end{pmatrix}
\begin{pmatrix} Uf_{2m}(\la) \\ Uf_{2m+1}(\la)\end{pmatrix}^\ast \, d\si(\la)\\
=&\, \int_{\Om_1} P_n(\la) \begin{pmatrix} Uf_{0}(\la) \\ Uf_{1}(\la)\end{pmatrix}
\begin{pmatrix} Uf_{0}(\la) \\ Uf_{1}(\la)\end{pmatrix}^\ast P_m(\la)^\ast \, d\si(\la) \\
=&\, \int_{\Om_1} P_n(\la) W_1(\la) P_m(\la)^\ast \, d\si(\la),
\end{split}
\end{equation}
where we have used \eqref{eq:Un=PnU0}.  
For the integral over $\Om_2$ we write $Uf(\la) = (U_1f(\la), U_2f(\la))^t$ and 
$V(\la) = (v_{ij}(\la))_{i,j=1}^2$, so that the integral over $\Om_2$ can be written as 
\begin{equation}\label{eq:startoforthoPnOm2a}
\begin{split}
&\, \sum_{i,j=1}^2 \int_{\Om_2} \begin{pmatrix} U_jf_{2n}(\la) v_{ij}(\la) \overline{U_if_{2m}(\la)} 
& U_jf_{2n}(\la) v_{ij}(\la) \overline{U_if_{2m+1}(\la)} \\  
U_jf_{2n+1}(\la) v_{ij}(\la) \overline{U_if_{2m}(\la)} & 
U_jf_{2n+1}(\la) v_{ij}(\la) \overline{U_if_{2m+1}(\la)}\end{pmatrix} d\rho(\la) \\
=&\, \sum_{i,j=1}^2 \int_{\Om_2} \begin{pmatrix} U_j f_{2n}(\la) \\ U_j f_{2n+1}(\la)\end{pmatrix}
\begin{pmatrix} U_i f_{2m}(\la) \\ U_i f_{2m+1}(\la)\end{pmatrix}^\ast v_{ij}(\la)\, d\rho(\la) \\
=&\, \sum_{i,j=1}^2 \int_{\Om_2} P_n(\la) \begin{pmatrix} U_j f_{0}(\la) \\ U_j f_{1}(\la)\end{pmatrix}
\begin{pmatrix} U_i f_{0}(\la) \\ U_i f_{1}(\la)\end{pmatrix}^\ast  P_m(\la)^\ast v_{ij}(\la)\, d\rho(\la) 
\,\\ =&\,  \int_{\Om_2} P_n(\la) W_2(\la)  P_m(\la)^\ast \, d\rho(\la), \\
\end{split}
\end{equation}
where we have used \eqref{eq:Un=PnU0} again and with 
\begin{equation}\label{eq:startoforthoPnOm2b}
\begin{split}
W_2(\la) \, &=\, \sum_{i,j=1}^2 \begin{pmatrix} U_j f_{0}(\la) \\ U_j f_{1}(\la)\end{pmatrix}
\begin{pmatrix} U_i f_{0}(\la) \\ U_i f_{1}(\la)\end{pmatrix}^\ast v_{ij}(\la) \,\\ 
&=\, \sum_{i,j=1}^2 v_{ij}(\la) \begin{pmatrix} U_j f_{0}(\la)\overline{U_if_0(\la)} 
& U_j f_{0}(\la)\overline{U_if_1(\la)} \\
U_j f_{1}(\la)\overline{U_if_0(\la)} & 
U_j f_{1}(\la)\overline{U_if_1(\la)}
\end{pmatrix} \\
\, &=\, \begin{pmatrix} (Uf_0(\la))^\ast V(\la) Uf_0(\la) & (Uf_1(\la))^\ast V(\la) Uf_0(\la) \\
(Uf_0(\la))^\ast V(\la) Uf_1(\la) & (Uf_1(\la))^\ast V(\la) Uf_1(\la) \end{pmatrix}
\end{split}
\end{equation}
and putting \eqref{eq:startoforthoPnOm1} and \eqref{eq:startoforthoPnOm2a}, 
\eqref{eq:startoforthoPnOm2b} into \eqref{eq:startoforthoPn} proves the result.
\end{proof}

In case we additionally assume $T$ is bounded, so that the measures $\si$ and $\rho$ 
have compact support, the coefficients in \eqref{eq:Tis5term} and \eqref{eq:3trr2by2MVOP} 
are bounded. In this case the corresponding Jacobi operator is bounded and self-adjoint.

\begin{remark}
Assume that $\Om_1=\si(T)$ or $\Om_2=\emptyset$, so that $T$ has simple spectrum. Then 
\begin{equation}\label{eq:defL2Wsigma}
\mathcal{L}^2(W_1d\si) \, = \, \{ f\colon \R \to \C^2 \mid  
\int_\R f(\la)^\ast W_1(\la) f(\la) \, d\si(\la)<\infty\} 
\end{equation}
has the subspace of null-vectors 
\begin{multline*}
\mathcal{N}\, = \, \{ f\in \mathcal{L}^2(W_1d\si) \mid 
\int_\R f(\la)^\ast W_1(\la) f(\la) \, d\si(\la) = 0\} \\ 
\, = \, \{ f\in \mathcal{L}^2(W_1d\si) \mid f(\la) = c(\la) \begin{pmatrix} 
\overline{Uf_1(\la)} \\ -\overline{Uf_0(\la)} \end{pmatrix} \text{ $\si$-a.e.}\}, 
\end{multline*}
where $c$ is a scalar-valued function. In this case  
$L^2(\cV)=\mathcal{L}^2(W_1d\si)/\mathcal{N}$. Note that 
$\mathcal{U}_n\colon \R\to L^2(W_1d\si)$ is completely determined by $Uf_0(\la)$,
 which is a restatement of $T$ having simple spectrum. 
From Theorem \ref{thm:genorth2t2MVOP} we see that, cf. \eqref{eq:spectraldecomposition}, 
\[
\langle P_n(\cdot)v_1, P_m(\cdot)v_2\rangle_{L^2(W_1d\si)} \, = \, 
\de_{nm} \langle v_1, v_2\rangle
\]
so that $\{ P_n(\cdot)e_i\}_{i\in\{1,2\}, n\in \N}$ is linearly independent in 
$L^2(W_1d\si)$ for any basis $\{e_1,e_2\}$ of
$\C^2$, cf. \eqref{eq:spectraldecomposition}. 
\end{remark}

We illustrate Theorem \ref{thm:genorth2t2MVOP} with an example, and we refer
to Groenevelt and the author \cite{GroeK} and \cite{GroeIK} for details. 
We extend the approach of Section \ref{ssec:JmatrixDO} and Lemma 
\ref{lem:TDfromL} by now assuming that $r$ is a polynomial of degree
$2$. Then the relations \eqref{eq:explexpgenpols} and \eqref{eq:OPphicombinationofPhi}
go through, except that it also involves a term $P_{n-2}$, respectively 
$\Phi_{n-2}$. Then we find that $r(L+\ga)$ is a $5$-term recurrence
operator. Adding a three-term recurrence relation, so $T=r(L+\ga)+\rho\, x$,
gives a $5$-term recurrence operator, see Exercise \ref{exer:5termrecurrence}.  
However it is usually hard to establish the assumption that an explicit spectral
decomposition of such an operator is available. 
Moreover, we want to have an example of such an operator where the 
spectrum of multiplicity $2$ is non-trivial. 

We do this for the Jacobi polynomials, and we consider\index{orthogonal polynomials!Jacobi polynomials} 
$T=T^{(\al,\be;\ka)}$ defined by
\begin{equation} \label{eq:T}
T = (1-x^2)^2 \frac{d^2}{dx^2} + (1-x^2)\bigl( \be-\al-(\al+\be+4)x\bigr) \frac{d}{dx} +
\frac14\bigl( \ka^2-(\al+\be+3)^2\bigr) (1-x^2)
\end{equation}
as an operator in the weighted $L^2$-space for 
the Jacobi polynomials; $L^2((-1,1), w^{(\al,\be)})$ with 
$w^{(\al,\be)}$ the normalised weight function for the Jacobi polynomials as 
given below. 
Here $\al,\be > -1$ and $\ka \in \R_{\geq 0} \cup i\R_{>0}$.
Then we can use \eqref{eq:DOforJacobipols} to obtain 
\begin{equation*} 
T^{(\al,\be;\ka)} = r\big( L^{(\al+1,\be+1)} +\rho\big), \qquad \rho =\frac14\left(\ka^2-(\al+\be+3)^2\right),
\end{equation*}
where $r(x)=1-x^2$ is, up to a constant, the quotient of the normalised weight 
functions of the Jacobi polynomial, 
\begin{gather*} 
r(x) = K \frac{w^{(\al+1,\be+1)}(x) }{w^{(\al,\be)}(x)}, \qquad K=  \frac{4(\al+1)(\be+1)}{ (\al+\be+2)(\al+\be+3)} \\
w^{(\al,\be)}(x) = 2^{-\al-\be-1}\frac{\Ga(\al+\be+2)}{\Ga(\al+1,\be+1) }(1-x)^\al (1+x)^\be.
\end{gather*}
It is then clear from the analogue of Lemma \ref{lem:TDfromL} 
that $T$ is $5$-term recurrence relation with respect to 
Jacobi polynomials 

In order to describe the spectral decomposition, we have to introduce some 
notation. For proofs we refer to Groenevelt and the author \cite{GroeK}. 
We assume $\be\geq \al$. 
Let $\Om_1, \Om_2 \subset \R$ be given by
\begin{equation*} 
\Om_1=\big(-(\be+1)^2,-(\al+1)^2\big)\quad \text{and} \quad \Om_2=\big(-\infty,-(\be+1)^2\big). 
\end{equation*}
We assume $0\leq \ka < 1$ or $\ka \in i\R_{>0}$ for convenience, in order to avoid
discrete spectrum of $T$. For the additional case of the discrete spectrum,
which arises with multiplicity one, see \cite{GroeK}. 
We set
\begin{equation*} 
\begin{split}
\de_\la = i \sqrt{ -\la-(\al+1)^2},&  \qquad \la \in \Om_1 \cup \Om_2,\\
\eta_\la = i \sqrt{ -\la-(\be+1)^2},& \qquad \la \in \Om_2,\\
\de(\la) = \sqrt{\la +(\al+1)^2 },& \qquad \la \in \C \setminus \big(\Om_1 \cup \Om_2\big),\\
\eta(\la) = \sqrt{\la+(\be+1)^2}, & \qquad \la \in \C \setminus \Om_2.
\end{split}
\end{equation*}
Here $\sqrt{\cdot}$ denotes the principal branch of the square root.
We denote by $\si$ the set $\Om_2 \cup \Om_1$. 
Theorem \ref{thm:integraltransform} will show that $\si$ is the spectrum of $T$.

Next we introduce the weight functions that we need to define $L^2(\cV)$. First we define
\begin{equation*} 
c(x;y) = \frac{\Ga(1+y)\, \Ga(-x)}{\Ga(\frac12(1+y-x+\ka))\, \Ga(\frac12(1+y-x-\ka))}.
\end{equation*}
With this function we define for $\la \in \Om_1$
\begin{equation*} 
v(\la) = \frac{1}{ c\big(\de_\la;\eta(\la)\big)c\big(-\de_\la;\eta(\la)\big) }.
\end{equation*}
For $\la \in \Om_2$ we define the matrix-valued weight function $V(\la)$ by
\begin{equation*} 
V(\la)=\begin{pmatrix} 1 & v_{12}(\la) \\ v_{21}(\la) & 1 \end{pmatrix},
\end{equation*}
with
\begin{equation*} 
v_{21}(\la) = \frac{ c(\eta_\la;\delta_\la) }{c(-\eta_\la;\delta_\la)} 
=\frac{ \Ga(-\eta_\la)\, \Ga\bigl(\frac12(1+\de_\la+\eta_\la+\ka)\bigr)\, \Ga\bigl(\frac12(1+\de_\la+\eta_\la-\ka)\bigr) }
{\Ga(\eta_\la)\,\Ga\bigl(\frac12(1+\de_\la-\eta_\la+\ka)\bigr)\, \Ga\bigl(\frac12(1+\de_\la- \eta_\la-\ka)\big)},
\end{equation*}
and $v_{12}(\la) = \overline{v_{21}(\la)}$.

Now we are ready to define the Hilbert space $L^2(\cV)$. It consists of functions that are $\C^2$-valued on $\Om_2$ and $\C$-valued on $\Om_1$. The inner product on $L^2(\cV)$ 
is given by
\begin{gather*}
\langle f,g \rangle_{\cV} = \frac{1}{2\pi D} \int_{\Om_2} g(\la)^* V(\la)
f(\la) \frac{d\la}{-i\eta_\la} 
 + \frac{1}{2\pi D} \int_{\Om_1} f(\la) \overline{g(\la)} v(\la) \frac{d\la}{-i\de_\la},
\end{gather*}
where $D=\frac{4\Ga(\al+\be+2)}{\Ga(\al+1,\be+1)}$.

Next we introduce the integral transform $\mathcal F$. For $\la \in \Om_1$ and $x \in (-1,1)$ we define
\begin{gather*} 
\varphi_\la(x) = \left(\frac{1-x}{2}\right)^{-\frac12(\al- \de_\la+1)} \left(\frac{1+x}{2}\right)^{-\frac12(\be - \eta(\la)+1)} 
\\ \qquad \times 
\rFs{2}{1}{ \frac12(1 +\de_\la + \eta(\la)-\ka), \frac12(1+ \de_\la + \eta(\la)+\ka)}{1+ \eta(\la)}{\frac{1+x}{2}}.
\end{gather*}
By Euler's transformation, see e.g.~\cite[(2.2.7)]{AndrAR}, we 
have the symmetry $\de_\la \leftrightarrow -\de_\la$. Furthermore, we define for $\la \in \Om_2$ and $x\in (-1,1)$,
\begin{gather*} 
\varphi^\pm_\la(x) =\left(\frac{1-x}{2}\right)^{-\frac12(\al - \de_\la+1)} \left(\frac{1+x}{2}\right)^{-\frac12(\be \mp \eta_\la+1)} 
\\  \qquad \times \rFs{2}{1}{ \frac12(1 + \de_\la \pm \eta_\la-\ka), \frac12(1 + \de_\la \pm \eta_\la+\ka)}{1\pm \eta_\la}{\frac{1+x}{2}}.
\end{gather*}
Observe that $\overline{\varphi^+_\la(x)}=\varphi^-_\la(x)$,
again by Euler's transformation.
Now, let $\mathcal F$ be the integral transform defined by
\begin{equation*} 
(\cF f)(\la) =
\begin{cases}
\displaystyle \int_{-1}^1 f(x) \begin{pmatrix} \varphi^+_\la(x)\\ \varphi^-_\la(x) \end{pmatrix} w^{(\al,\be)}(x) \,dx,& \la \in \Om_2,\\
\displaystyle \int_{-1}^1 f(x) \varphi_\la(x) w^{(\al,\be)}(x) \,dx,& \la \in \Om_1,
\end{cases}
\end{equation*}
for all $f \in \cH$ such that the integrals converge. 
The following result says that $\cF$ is the required unitary operator $U$ intertwining $T$ with 
multiplication. 

\begin{thm} \label{thm:integraltransform}
The transform $\cF$ extends uniquely to a unitary operator $\cF\colon\cH \to L^2(\cV)$ such that $\cF  T = M \cF$,
where $M\colon L^2(\cV) \to L^2(\cV)$ is the unbounded multiplication operator given by $(Mg)(\la) = \la g(\la)$ for almost all $\la \in \si$.
\end{thm}

The proof of Theorem \ref{thm:integraltransform} is based on the fact that the eigenvalue equation 
$Tf_\la=\la f_\la$ can be solved in terms of hypergeometric functions since
it is a second-order differential equation with regular singularities at three points.
Having sufficiently many solutions available gives the opportunity to find the 
Green kernel, and hence the resolvent operator, from which one derives the spectral
decomposition, see \cite{GroeK} for details. 

Now we want to apply Theorem \ref{thm:genorth2t2MVOP} for the polynomials 
generated by \eqref{eq:3trr2by2MVOP}. For this it suffices 
to write down explicitly the coefficients $a_n$, $b_n$ and $c_n$ in the 
$5$-term recurrence realisation of the operator $T$, cf.
Exercise \ref{exer:5termrecurrence}, and to calculate 
the matrix entries in the weight matrices of Theorem \ref{thm:genorth2t2MVOP}. 

The coefficients $a_n$, $b_n$ and $c_n$ follow by keeping track 
of the method of Exercise \ref{exer:5termrecurrence}, and  
this worked out in Exercise \ref{exer:calanbncnJacobi}.
This then makes the matrix entries in the three-term recurrence relation 
\eqref{eq:3trr2by2MVOP} 
completely explicit. 

It remains to calculate the matrix entries of the weight functions
in Theorem \ref{thm:genorth2t2MVOP}. 
In \cite{GroeK} these functions  are calculated in terms of ${}_3F_2$-functions.

\subsection{Exercises}

\begin{enumerate}[1.]
\item \label{exer:5termrecurrence} 
Generalise the situation of Section \ref{ssec:JmatrixDO} 
to the case where the polynomial $r$ is of degree $2$. 
Show that in this case the analogue of \eqref{eq:explexpgenpols}
and \eqref{eq:OPphicombinationofPhi} involve
three terms in the right-hand side. 
Show that now the operator $T=r(L+\ga)+ \tau x$ 
is a $5$-term operator in the bases $\{\phi_n\}_{n\in \N}$ of $L^2(\nu)$.
Here $r$, respectively $x$, denotes multiplication by $r$, respectively $x$,
and $\ga, \tau$ are constants. 
\item \label{exer:eq3trr2by2MVOP}
Show that \eqref{eq:3trr2by2MVOP} and \eqref{eq:Un=PnU0} hold 
starting from \eqref{eq:5trrforUfn}.
\item \label{exer:calanbncnJacobi}
\begin{enumerate}[(a)]
\item Show that 
\[
\phi_n = \al_n \Phi_n + \be_n \Phi_{n-1}+ \ga_n \Phi_{n-2},
\]
where $\phi_n$, respectively $\Phi_n$, are the orthonormalised
Jacobi polynomials $P^{(\al,\be)}_n$, respectively 
$P^{(\al+1,\be+1)}_n$ and where 
\[
\begin{split}
\al_n  & = \frac{2}{\sqrt{K}} \frac{1}{2n+\al+\be+2}\sqrt{ \frac{ (\al+n+1)(\be+n+1)(n+\al+\be+1)(n+\al+\be+2)  }{ (\al+\be+2n+1)(\al+\be+2n+3)}},\\
\be_n &= (-1)^{n} \frac{2}{\sqrt{K}} \frac{(\be-\al) \sqrt{n(n+\al+\be+1)}}{(\al+\be+2n)(\al+\be+2n+2)},\\
\ga_n & = - \frac{2}{\sqrt{K}} \frac{1}{2n+\al+\be}  \sqrt{ \frac{n(n-1)(\al+n)(\be+n)}{(\al+\be+2n-1)(\al+\be+2n+1)} }.
\end{split}
\]
Here $K$ as in the definition of $r(x)$. 
\item Show that 
\begin{gather*}
a_n =  K \al_n \ga_{n+2} (\La_n +\rho), \qquad b_n=K \al_n \be_{n+1}(\La_n+\rho)+ K \be_n \ga_{n+1} (\La_{n+1}+\rho), \\
\qquad c_n= K \al_n^2 (\La_n+\rho) + K \be_n^2(\La_{n-1}+\rho) + K \ga_n^2 (\La_{n-2}+\rho),
\end{gather*}
where $\La_n = -n(n+\al+\be+3)$, 
$\rho$ as in the definition of $T=T^{(\al,\be;\ka)}$ 
and $\al_n,\be_n,\ga_n$ as in (a). 
\end{enumerate}
\end{enumerate}

\appendix

\section{The spectral theorem}\label{app:spectralthm}

In this appendix we recall some facts from functional analysis 
with emphasis on the spectral theorem. There are many sources 
for this appendix, or parts of it, see e.g. 
\cite{DunfS}, \cite{Lax}, \cite{ReedS}, \cite{Rudi}, \cite{Schm}, \cite{Ston}, 
\cite{Wern}, 
but many other sources are available. 

\subsection{Hilbert spaces and operators}

A vector space $\cH$ over $\C$ is an inner product space if
there exists a mapping $\langle\cdot,\cdot\rangle\colon
\cH\times \cH \to \C$ such that for all $u,v,w\in\cH$
and for all $a,b\in\C$ we have
(i) $\langle a v+b w, u\rangle = a \langle v,u\rangle
+b \langle w,u\rangle$, (ii) $\langle u,v\rangle = \overline{
\langle v,u\rangle}$, and (iii) $\langle v,v\rangle \geq 0$ and
$\langle v,v\rangle =0$ if and only if $v=0$. With
the inner product we associate the norm $\| v\|=\|v\|_\cH
=\sqrt{\langle v,v \rangle}$, and the topology from the
corresponding metric $d(u,v)=\|u-v\|$. The standard inequality
is the Cauchy-Schwarz inequality;\index{Cauchy-Schwarz inequality}
$|\langle u,v\rangle| \leq \| u\| \|v\|$.
A Hilbert space $\cH$
is a complete inner product space, i.e. for
any Cauchy sequence $\{x_n\}_n$ in $\cH$, i.e.
$\forall \ep>0$ $\exists N\in\N$ such that for
all $n,m\geq N$ $\| x_n-x_m\|<\ep$, there exists an element
$x\in\cH$ such that $x_n$ converges to $x$.
In these notes all Hilbert spaces are separable, i.e. there
exists a denumerable set of basis vectors. 
The Cauchy-Schwarz inequality can be extended to the Bessel \index{Bessel inequality}
inequality; for an orthonormal sequence $\{ f_i\}_{i\in I}$ in 
$\cH$, i.e. $\langle f_i, f_j\rangle = \de_{i,j}$, 
\[
\sum_{i\in I} | \langle x, f_i\rangle |^2 \leq \| x\|^2
\]

\begin{example}\label{exa:Hilbertspaces}
(i) The finite-dimensional inner product space $\C^N$ with its
standard inner product is a Hilbert space. 

(ii) $\ell^2(\Z)$, the space of square summable 
sequences $\{ a_k\}_{k\in\Z}$, and $\ell^2(\N)$, the space of square
summable sequences $\{ a_k\}_{k\in\N}$, are Hilbert spaces.
The inner product is given by
$\langle \{ a_k\}, \{ b_k\} \rangle = \sum_{k\in\N} a_k\overline{b_k}$.
An orthonormal basis is given by the sequences $e_k$ defined
by $(e_k)_l=\de_{k,l}$, so we identify $\{a_k\}$ with
$\sum_{k\in\N} a_ke_k$. \index{l@$\ell^2(\Z)$} \index{l@$\ell^2(\N)$}

(iii) We consider a positive Borel measure $\mu$ on the real line
$\R$ such that all moments exist, i.e.
$\int_\R |x|^m\, d\mu(x) < \infty$ for all $m\in\N$. Without loss
of generality we assume that $\mu$ is a probability measure,
$\int_\R d\mu(x)=1$.
By $L^2(\mu)$ we denote the space of square integrable functions
on $\R$, i.e. $\int_\R |f(x)|^2\, d\mu(x)<\infty$.
Then $L^2(\mu)$ is a Hilbert space (after identifying
two functions $f$ and $g$ for which
$\int_\R |f(x)-g(x)|^2\, d\mu(x)=0$) with respect to the
inner product $\langle f,g\rangle =\int_\R f(x)\overline{g(x)}\,
d\mu(x)$.
In case $\mu$ is a finite sum of discrete Dirac measures,
we find that $L^2(\mu)$ is finite dimensional.

(iv) For two Hilbert spaces $\cH_1$ and $\cH_2$ we can take 
its algebraic tensor product $\cH_1\otimes \cH_2$ 
and equip it with an inner product defined on simple tensors by  
\[
\langle v_1\otimes v_2, w_1\otimes w_2\rangle = \langle v_1,w_1\rangle_{\cH_1}
\langle v_2,w_2\rangle_{\cH_2}. 
\]
 Taking its completion gives 
the Hilbert space $\cH_1\hat\otimes \cH_2$. \index{tensor product of Hilbert spaces}
\end{example}

An operator $T$ from a Hilbert space $\cH$ into
another Hilbert space $\cK$ is linear if
for all $u,v\in \cH$ and for all $a,b\in\C$ we have
$T(au+bv)=aT(u)+bT(v)$. An operator $T$ is bounded if there exists
a constant $M$ such that $\| Tu\|_\cK\leq M\|u\|_\cH$
for all $u\in \cH$. The smallest $M$ for which this holds is
the norm, denoted by $\|T\|$, of $T$. A bounded linear operator
is continuous. The adjoint of a bounded linear operator $T\colon
\cH\to\cK$ is a map $T^\ast\colon \cK\to \cH$
with $\langle Tu,v\rangle_\cK =
\langle u, T^\ast v\rangle_\cH$. We call $T\colon
\cH\to \cH$ self-adjoint if $T^\ast=T$.
$T^\ast\colon \cK\to \cH$ is unitary if $T^\ast T = {\mathbf 1}_\cH$ and
$TT^\ast = {\mathbf 1}_\cK$.
A projection $P\colon \cH\to \cH$ is a self-adjoint bounded operator 
such that $P^2=P$.

An operator $T\colon \cH\to \cK$ is compact\index{compact operator} if the closure 
of the image of the 
unit ball $B_1=\{ v\in \cH \mid \|v\|\leq 1\}$ under $T$ is compact in $\cK$. 
In case $\cK$ is finite dimensional any bounded operator $T\colon \cH\to \cK$ is compact, 
and slightly more general, any operator which has finite rank\index{operator of finite rank},
i.e. its range is finite dimensional, is compact.
Moreover, any compact operator can be approximated in the operator norm by 
finite-rank operators.

\subsection{Hilbert $\text{C}^\ast$-modules}\label{ssec:HilbertCastmodules}
For more information on Hilbert $\text{C}^\ast$-modules, see 
e.g. Lance \cite{Lanc}. 
The space $B(\cH)$ of bounded linear operators $T\colon \cH \to \cH$
is a $\ast$-algebra, where the $\ast$-operation is given by the adjoint,  
satisfying $\|TS\| \leq \|T \| \|S\|$ and $\|T^\ast T\|= \|T \|^2$. 
With the operator-norm $B(\cH)$ is a metric space, and a $\text{C}^\ast$-algebra
\index{C-algebra@$\text{C}^\ast$-algebra} is a closed $\ast$-invariant 
subalgebra of $B(\cH)$. 
Examples of a $\text{C}^\ast$-algebra are $B(\cH)$ and the space of 
all compact operators $T\colon \cH \to \cH$. 
We only need $M_N(\C)=B(\C^N)$, the space of all linear maps from $\C^N$ to itself, 
as an example of a $\text{C}^\ast$-algebra. 
An element $a\in A$ in a $\text{C}^\ast$-algebra $A$ is positive if
$a= b^\ast b$ for some element $b\in A$, and we use the notation $a\geq 0$.
This notation is extended to $a\geq b$ meaning $(a-b)\geq 0$. 
In case of $A=M_N(\C)$, $T\geq 0$ means that $T$ corresponds to 
a positive semi-definite matrix,\index{positive semi-definite matrix} 
i.e. $\langle Tx,x\rangle \geq 0$ for all $x\in \C^N$. 
The positive definite matrices form the cone $P_N(\C)$\index{P1@$P_N(\C)$}
in $M_N(\C)$. 
We say $T$ is a positive matrix\index{positive matrix} or
a positive definite matrix\index{positive definite matrix}
if $\langle Tx,x\rangle > 0$ for all $x\in \C^N\setminus\{0\}$.
Note that terminology concerning positivity in $\text{C}^\ast$-algebras and 
matrix algebras does not coincide, and we follow the latter, see \cite{HornJ}.  

A Hilbert $\text{C}^\ast$-module\index{Hilbert C-module@Hilbert $\text{C}^\ast$-module} 
$E$ over the (unital) $\text{C}^\ast$-algebra $A$ is a 
left $A$-module $E$ equipped with an $A$-valued inner product 
$\langle \cdot, \cdot \rangle \colon E\times E \to A$ so that 
for all $v,w,u\in E$ and all $a,b\in A$ 
\begin{gather*}
\langle av+bw, u\rangle = a\langle v, u\rangle + b \langle w, u\rangle, \qquad 
\langle v,w \rangle = \langle w,v\rangle^\ast, \qquad 
\langle v,v\rangle \geq 0\ \text{and} \ \langle v,v\rangle = 0 \ \Leftrightarrow\ v=0
\end{gather*}
and $E$ is complete with respect to the norm $\|v\| = \sqrt{\| \langle v,v\rangle\|}$. 
The analogue of the Cauchy-Schwarz inequality\index{Cauchy-Schwarz inequality} then reads 
\[
\langle v,w\rangle \langle w,v \rangle \leq \| \langle w,w\rangle\| \, \langle v,v\rangle,
\qquad v,w \in E
\]
and the analogue of the Bessel inequality\index{Bessel inequality}
\[
\sum_{i\in I} \langle v, f_i\rangle \langle f_i, v\rangle \leq \langle v, v \rangle, 
\qquad v\in E 
\]
for $(f_i)_{i\in I}$ an orthonormal set in $E$, i.e. $\langle f_i, f_j\rangle =\de_{i,j}\in A$.
(Here we use that $A$ is unital.)

\subsection{Unbounded operators}
We are also
interested in unbounded linear operators.
In that case we denote $(T, \cD(T))$, where $\cD(T)$, the
domain of $T$, is
a linear subspace of $\cH$ and $T\colon \cD(T)\to \cH$.
Then $T$ is densely defined\index{densely defined} if the closure of $\cD(T)$ equals
$\cH$. All unbounded operators that we consider in these
notes are densely defined. If the operator $(T-z)$, $z\in\C$, has an
inverse $R(z)=(T-z)^{-1}$ which is densely defined and is bounded,
so that $R(z)$,
the resolvent operator,\index{resolvent operator}
extends to a bounded linear operator on $\cH$,
then we call $z$ a regular value. 
The set of all regular values
is the resolvent set $\rho(T)$. The complement of the resolvent
set $\rho(T)$ in $\C$ is the spectrum $\si(T)$ of $T$.
The point spectrum is the subset of the spectrum for which
$T-z$ is not one-to-one. In this case there exists a vector
$v\in\cH$ such that $(T-z)v=0$, and $z$ is an eigenvalue.
The continuous spectrum consists of the points $z\in\si(T)$ for
which $T-z$ is one-to-one, but for which $(T-z)\cH$ is dense
in $\cH$,
but not equal to $\cH$. The remaining part of the spectrum
is the residual spectrum. For self-adjoint operators, both
bounded and unbounded, 
the spectrum only consists of the discrete and continuous spectrum.

The resolvent operator is defined in the same way for a bounded operator. 
For a bounded operator $T$ the spectrum $\si(T)$
is a compact subset of the disk of radius $\| T\|$. Moreover,
if $T$ is self-adjoint, then $\si(T)\subset \R$, so that
$\si(T)\subset [-\|T\|, \|T\|]$ and the spectrum consists
of the point spectrum and the continuous spectrum.

\subsection{The spectral theorem
for bounded self-adjoint operators}

A resolution of the identity, say $E$, of a Hilbert space
$\cH$ is a projection valued Borel measure on $\R$ such that
for all Borel sets $A,B\subseteq \R$ we have
(i) $E(A)$ is a self-adjoint projection, (ii) $E(A\cap B)=
E(A)E(B)$, (iii) $E(\emptyset)=0$,
$E(\R)={\mathbf 1}_\cH$,
(iv) $A\cap B=\emptyset$ implies $E(A\cup B)=E(A)+E(B)$,
and (v) for all $u,v\in\cH$ the map $A\mapsto
E_{u,v}(A)=\langle E(A)u,v\rangle$ is a complex Borel measure.

A generalisation of the spectral
theorem for matrices
is the following theorem for compact self-adjoint operators, 
see e.g \cite[VI.3]{Wern}.

\begin{thm}[Spectral theorem for compact operators]\label{thm:App-spectralcompactthm} \index{spectral theorem}
Let $T\colon \cH\to \cH$ be a compact self-adjoint linear map, then
there exists a sequence of orthonormal vectors $(f_i)_{i\in I}$ such that
$\cH$ is the orthogonal direct sum of $\Ker(T)$ and the subspace spanned by $(f_i)_{i\in I}$ 
and there exists a sequence $(\la_i)_{i\in I}$ of non-zero real numbers converging to $0$ so that
\[
T v  = \sum_{i\in I}  \la_i \, \langle v, f_i\rangle f_i 
\]
\end{thm}

Here $I$ is at most countable, since we assume $\cH$ to be separable. 
In case $I$ is finite, the fact that the sequence $(\la_i)_{i\in I}$  
is a null-sequence is automatic. 

The following theorem is the corresponding statement for bounded self-adjoint operators,
see \cite[\S X.2]{DunfS}, \cite[\S 12.22]{Rudi}.

\begin{thm}[Spectral theorem]\label{thm:App-spectralthm} \index{spectral theorem}
Let $T\colon
\cH\to \cH$ be a bounded self-adjoint linear map, then
there exists a unique resolution of the identity such that
$T=\int_\R t \, dE(t)$, i.e. $\langle Tu,v\rangle =\int_\R t \,
dE_{u,v}(t)$. Moreover, $E$ is supported on the
spectrum $\si(T)$, which is contained in the interval
$[-\|T\|, \|T\|]$.
Moreover, any of the spectral projections $E(A)$, $A\subset \R$
a Borel set, commutes with $T$.
\end{thm}

A more general theorem of this kind holds for normal
operators, i.e. for those operators satisfying $T^\ast T=TT^\ast$.

For the case of a compact operator, we have 
in the notation of Theorem \ref{thm:App-spectralcompactthm} that for $\la_i$ 
the spectral measure evaluated at $\{\la_k\}$  is the 
orthogonal projection on the corresponding eigenspace;
\[
E(\{\la_k\})v = \sum_{i\in I; \la_i=\la_k} \langle v, f_i\rangle f_i. 
\]

Using the spectral theorem we define for
any continuous
function $f$ on the spectrum $\si(T)$ the operator $f(T)$ by
$f(T)=\int_\R f(t) \, dE(t)$, i.e.
$\langle f(T)u,v\rangle =\int_\R f(t) \, dE_{u,v}(t)$. Then $f(T)$
is bounded operator with norm equal to
the supremum norm of $f$ on the spectrum of $T$, i.e.
$\| f(T)\| = \sup_{x\in\si(T)} |f(x)|$.
This is known as the functional calculus for self-adjoint operators.
In particular, for
$z\in\rho(T)$ we see that $f\colon x\mapsto (x-z)^{-1}$
is continuous on the spectrum, and the corresponding operator is
just the resolvent operator $R(z)$. The
functional calculus can be extended to measurable functions,
but then $\| f(T)\| \leq \sup_{x\in\si(T)} |f(x)|$.

The spectral measure can be obtained from the
resolvent operators by the Stiel\-tjes-Perron inversion formula,
see \cite[Thm. X.6.1]{DunfS}.

\begin{thm}\label{thm:App-StieltjesPerron} \index{Stieltjes-Perron inversion}
The spectral measure of the open interval
$(a,b)\subset \R$ is given by
$$
E_{u,v}\bigl( (a,b)\bigr) = \lim_{\de\downarrow 0}
\lim_{\ep\downarrow 0} \frac{1}{2\pi i}
\int_{a+\de}^{b-\de}
\langle R(x+i\ep)u,v\rangle - \langle R(x-i\ep)u,v\rangle \, dx.
$$
The limit holds in the strong operator topology, i.e.
$T_nx\to Tx$ for all $x\in\cH$.
\end{thm}

Note that the right hand side of Theorem \ref{thm:App-StieltjesPerron} is like the 
Cauchy integral formula, where we integrate over a rectangular contour.

\subsection{Unbounded self-adjoint operators}\label{ssec:appunboundedsaoperators}

Let $(T, \cD(T))$, with $\cD(T)$ the
domain of $T$, be a densely defined unbounded operator on
$\cH$. We can now define the
adjoint operator\index{adjoint operator} $(T^\ast, \cD(T^\ast))$ as follows.
First define
$$
\cD(T^\ast) = \{ v\in \cH\mid
u\mapsto \langle Tu,v\rangle \text{\ is continuous on
$\cD(T)$} \}.
$$
By the density of $\cD(T)$
the map $u\mapsto \langle Tu,v\rangle$
for $v\in \cD(T^\ast)$ extends to a continuous linear
functional $\om\colon
\cH\to \C$, and by the Riesz representation
theorem there exists a unique $w\in\cH$ such that
$\om(u)=\langle u,w\rangle$ for all $u\in \cH$.
Now the adjoint $T^\ast$ is defined by $T^\ast v =w$, so that
$$
\langle Tu,v\rangle = \langle u, T^\ast v\rangle \qquad
\forall\, u\in\cD(T),\, \forall \, v\in \cD(T^\ast).
$$

If $T$ and $S$ are unbounded operators
on $\cH$, then $T$ extends $S$, notation
$S\subset T$, if $\cD(S)\subset \cD(T)$ and
$Sv=Tv$ for all $v\in\cD(S)$. Two unbounded operators
$S$ and $T$ are equal, $S=T$, if $S\subset T$ and $T\subset S$,
or $S$ and $T$ have the same domain and act in the same way.
In terms of the graph
$$
{\mathcal G}(T) = \{ (u, Tu)\mid u\in \cD(T)\} \subset
\cH\times \cH
$$
we see that $S\subset T$ if and only if ${\mathcal G}(S) \subset
{\mathcal G}(T)$. An operator $T$ is closed if its graph is
closed in the product topology of $\cH\times \cH$.
The adjoint of a densely defined operator is a closed
operator, since the graph of the adjoint is given as
$$
{\mathcal G}(T^\ast) = \{ (-Tu,u)\mid u\in \cD(T)\}^\perp,
$$
for the inner product
$\langle (u,v),(x,y)\rangle=\langle u,x\rangle
+\langle v,y\rangle$ on $\cH\times \cH$, see
\cite[13.8]{Rudi}.

A densely defined operator is symmetric\index{symmetric operator} if
$T\subset T^\ast$, or, 
$$
\langle Tu,v\rangle = \langle u,Tv\rangle , \qquad\forall \
u,v\in\cD(T).
$$
A densely defined operator is self-adjoint if $T=T^\ast$,
so that a self-adjoint operator is closed. The spectrum of
an unbounded self-adjoint operator is contained in $\R$.
Note that $\cD(T) \subset \cD(T^\ast)$, so that
$\cD(T^\ast)$ is a dense subspace and taking
the adjoint once more gives
$(T^{\ast\ast}, \cD(T^{\ast\ast}))$
as the minimal closed extension
of $(T, \cD(T))$, i.e. any densely defined symmetric operator
has a closed extension. We have
$T\subset T^{\ast\ast}\subset T^\ast$.
We say that the densely defined symmetric operator is
essentially self-adjoint if its closure is self-adjoint, i.e.
if $T\subset T^{\ast\ast} = T^\ast$.

In general, a densely defined symmetric operator $T$
might not have self-adjoint extensions. This can be measured
by the deficiency indices. Define for $z\in \C\backslash\R$
the eigenspace
$$
N_z = \{ v\in \cD(T^\ast)\mid T^\ast v=z\, v\}.
$$
Then $\dim N_z$ is constant for $\Im z>0$ and for $\Im z<0$,
\cite[Thm. XII.4.19]{DunfS}, and we put $n_+=\dim N_i$
and $n_-=\dim N_{-i}$. The pair $(n_+,n_-)$ are
the deficiency indices for the densely defined symmetric
operator $T$. \index{deficiency index}
Note that if $T^\ast$ commutes with complex conjugation, then we automatically have
$n_+=n_-$. Here complex conjugation is an antilinear mapping $f=\sum_n f_n e_n$ to 
$\sum_n \overline{f_n} e_n$, where $\{e_n\}_n$ is an orthonormal basis of the separable Hilbert space $\cH$. 
Note furthermore that if $T$ is self-adjoint then
$n_+=n_-=0$, since a self-adjoint operator cannot have
non-real eigenvalues. Now the following
holds, see \cite[\S XII.4]{DunfS}.

\begin{prop}\label{prop:App-extensionselfadjointopertaotrs} 
Let $(T, \cD(T))$
be a densely defined symmetric operator. \par\noindent
{\rm (i)}
$\cD(T^\ast)= \cD(T^{\ast\ast})
\oplus N_i \oplus N_{-i}$,
as an orthogonal direct sum with respect to the graph norm for
$T^\ast$ from
$\langle u,v\rangle_{T^\ast} = \langle u,v\rangle +
\langle T^\ast u,T^\ast v\rangle$. As a direct sum,
$\cD(T^\ast)= \cD(T^{\ast\ast}) + N_z + N_{\bar z}$
for general $z\in\C\backslash\R$.
\par\noindent
{\rm (ii)} Let $U$ be an isometric bijection $U\colon N_i\to N_{-i}$
and define $(S, \cD(S))$ by
$$
\cD(S) = \{ u + v +Uv\mid u\in \cD(T^{\ast\ast}), \
v\in N_i\}, \quad Sw=T^\ast w,
$$
then $(S, \cD(S))$ is a self-adjoint extension
of $(T,\cD(T))$,
and every self-adjoint extension of $T$ arises in this way.
\end{prop}

In particular, $T$ has self-adjoint extensions if and only if
the deficiency indices are equal; $n_+=n_-$.
However, $T$ has a self-adjoint extension to a bigger
Hilbert space in case the deficiency indices are unequal, 
see e.g. \cite[Prop.~3.17, Cor.~13.4]{Schm}, but 
we will not take this into account. 
$T^{\ast\ast}$ is a closed symmetric extension of
$T$. We can also
characterise the domains of the self-adjoint extensions of $T$
using the sesquilinear form
$$
B(u,v) =  \langle T^\ast u, v\rangle -
\langle u, T^\ast v\rangle, \qquad u,v\in \cD(T^\ast),
$$
then $\cD(S)=\{ u\in \cD(T^\ast)\mid B(u,v)=0
, \ \forall v\in \cD(S)\}$.

\subsection{The spectral theorem
for unbounded self-adjoint operators}

With all the preparations of the previous
subsection the Spectral
Theorem \ref{thm:App-spectralthm} goes through in the unbounded setting,
see \cite[\S XII.4]{DunfS}, \cite[Ch.~13]{Rudi}.

\begin{thm}[Spectral theorem] 
\label{thm:App-specrtathm-unbdd} \index{spectral theorem}
Let $T\colon
\cD(T)\to \cH$ be an unbounded self-adjoint linear map,
then there exists a unique resolution of the identity such that
$T=\int_\R t \, dE(t)$, i.e. $\langle Tu,v\rangle =\int_\R t \,
dE_{u,v}(t)$ for $u\in\cD(T)$, $v\in \cH$.
Moreover, $E$ is supported on the
spectrum $\si(T)$, which is contained in $\R$.
For any bounded operator $S$ that satisfies $ST\subset TS$ we
have $E(A)S=SE(A)$, $A\subset \R$
a Borel set.
Moreover, the Stieltjes-Perron inversion formula
of Theorem \ref{thm:App-StieltjesPerron} remains valid;
$$
E_{u,v}\bigl( (a,b)\bigr) = \lim_{\de\downarrow 0}
\lim_{\ep\downarrow 0} \frac{1}{2\pi i}
\int_{a+\de}^{b-\de}
\langle R(x+i\ep)u,v\rangle - \langle R(x-i\ep)u,v\rangle \, dx.
$$
\end{thm}

As in the case of bounded self-adjoint operators we can now define
$f(T)$ for any measurable function $f$ by
$$
\langle f(T)u,v\rangle = \int_\R f(t)\, dE_{u,v}(t),
\qquad u\in \cD(f(T)), \ v\in \cH,
$$
where $\cD(f(T)) = \{ u\in \cH\mid
\int_\R |f(t)|^2\, dE_{u,u}(t) <\infty\}$ is the domain
of $f(T)$. This makes $f(T)$ into a densely defined closed
operator. In particular, if $f\in L^\infty(\R)$, then
$f(T)$ is a continuous operator, by the closed graph
theorem. This in particular applies to $f(x)=(x-z)^{-1}$,
$z\in\rho(T)$, which gives the resolvent operator.



\section{Hints and answers for selected exercises}

\emph{Exercise \ref{sec:threetermZ}.\ref{exer:lembbdL}.}
See e.g. \cite[Lemma (3.3.3)]{Koel-Laredo}.

\emph{Exercise \ref{sec:threetermZ}.\ref{exer:adjointL}.}
See e.g. proof of Proposition \ref{prop:MVadjointJ} or 
\cite[Proposition (3.4.2)]{Koel-Laredo}.

\emph{Exercise \ref{sec:threetermZ}.\ref{exer:realcoefficients}.}
See e.g. \cite{Koel-IM}. 

\emph{Exercise \ref{sec:threetermZ}.\ref{exer:2times2}.}
See \cite[p.~583]{Bere}. 

\emph{Exercise \ref{sec:threetermN}.\ref{exer:thm3termrec-scalarOP}.}
See \cite{Koel-Laredo}.

\emph{Exercise \ref{sec:MVOP}.\ref{exer:MVOPexistence}.}
See e.g. \cite{DamaPS}, \cite{GrunT} or Section \ref{ssec:MoreMweights}.

\emph{Exercise \ref{sec:MVOP}.\ref{exer:lemMV2ndkindsolverecurrence}.}
See \cite{Berg}. 

\emph{Exercise \ref{sec:MVOP}.\ref{exer:lemStieltjesPerroninversion}.}
See e.g. \cite[\S 3.1]{Koel-Laredo}, Van Assche \cite[\S 22.1]{Isma}
for comparable calculations.

\emph{Exercise \ref{sec:MVOP}.\ref{exer:MVFavard}.}
See e.g. \cite[\S 1]{ApteN}.

\emph{Exercise \ref{sec:MVOP}.\ref{exer:Carleman}.} Mimick the proof of the 
Carleman condition for the scalar case, see \cite[Ch.~1]{Akhi}.

\emph{Exercise \ref{sec:MoreMweightsMVOPS-Jacobi}.\ref{exer:commutant}.}
See Tirao and Zurri\'an \cite{TiraZ}. 

\printindex


\end{document}